\documentclass[a4paper, 10pt]{amsart}

\usepackage[square, numbers, comma]{natbib} % Use the natbib reference package - read up on this to edit the reference style; if you want text (e.g. Smith et al., 2012) for the in-text references (instead of numbers), remove 'numbers'. Add sort&compress to sort references the way they appear in the bibliography, and compress references which are close in numbers.

\usepackage[usenames,dvipsnames,svgnames,table]{xcolor}
\usepackage{amsmath,amsthm,amssymb,amssymb,esint,verbatim,tabularx,graphicx}
\usepackage{fancyhdr}
\usepackage{enumerate}
\usepackage{amsmath}
\usepackage{fancyhdr}
\usepackage{epic}
\usepackage{pgf,tikz}
\usetikzlibrary{arrows}

\usepackage[utf8]{inputenc}
\usepackage{color}
\usepackage{hyperref}
\usepackage{verbatim}
\usepackage{pdfpages}
\usepackage{subcaption}
\usepackage{multicol}

\usepackage{lipsum,tabularx}

\hypersetup{urlcolor=blue, colorlinks=true} % Colors hyperlinks in blue - change to black if annoying

\usepackage{float} %force table placement
\usepackage{placeins} %force table placement

\newtheorem{theorem}{Theorem}[section]

\newtheorem{lemma}[theorem]{Lemma}
\newtheorem{proposition}[theorem]{Proposition}
\newtheorem{corollary}[theorem]{Corollary}
\newtheorem{definition}{Definition}[section]

\theoremstyle{remark}
\newtheorem{remark}[theorem]{Remark}
\theoremstyle{definition}

\numberwithin{equation}{section}

\newcommand{\R}{\ensuremath{\mathbb{R}}}
\newcommand{\N}{\ensuremath{\mathbb{N}}}

\newcommand{\Z}{\ensuremath{\mathbb{Z}}}
\newcommand{\J}{\ensuremath{\mathbb{J}}}

\newcommand{\vertiii}[1]{{\left\vert\kern-0.25ex\left\vert\kern-0.25ex\left\vert #1 
    \right\vert\kern-0.25ex\right\vert\kern-0.25ex\right\vert}}

\newcommand{\Levy}{\ensuremath{\mathcal{L}}}

\newcommand{\Operator}{\ensuremath{\mathfrak{L}^{\sigma,\mu}}}

\newcommand{\Levymu}{\ensuremath{\mathcal{L}^\mu}}

\newcommand{\dd}{\,\mathrm{d}}

\newcommand{\dell}{\partial}
\newcommand{\indikator}{\mathbf{1}_{|z|\leq 1}}
\newcommand{\indik}{\mathbf{1}}

\newcommand{\e}{\text{e}}

\DeclareMathOperator{\sgn}{\textup{sign}}

\newcommand{\Grid}{\mathcal{G}_h}
\newcommand{\GridT}{\mathcal{T}_{\Delta t}^T}

\begin{document}

\title[Numerical analysis of equations of porous medium type]{Robust
  numerical methods for nonlocal \\(and local) equations of porous medium
  type. \\Part II: Schemes and experiments}

\author[F. del Teso]{F\'elix del Teso}
\address[F. del Teso]{Basque Center for Applied Mathematics (BCAM)\\
Bilbao, Spain} 
\email[]{fdelteso\@@{}bcamath.org}
\urladdr{http://www.bcamath.org/es/people/fdelteso}

\author[J. Endal]{J\o rgen Endal}
\address[J. Endal]{Department of Mathematical Sciences\\
Norwegian University of Science and Technology (NTNU)\\
N-7491 Trondheim, Norway} 
\email[]{jorgen.endal\@@{}ntnu.no}
\urladdr{http://folk.ntnu.no/jorgeen}

\author[E. ~R.~Jakobsen]{Espen R. Jakobsen}
\address[E. R. Jakobsen]{Department of Mathematical Sciences\\
Norwegian University of Science and Technology (NTNU)\\
N-7491 Trondheim, Norway} 
\email[]{espen.jakobsen\@@{}ntnu.no}
\urladdr{http://www.math.ntnu.no/\~{}erj/}



\begin{abstract}
\noindent We develop a  unified and easy to use framework to study robust fully discrete numerical methods for nonlinear
degenerate diffusion equations
$$
\partial_t u-\mathfrak{L}[\varphi(u)]=f(x,t) \qquad\text{in}\qquad \R^N\times(0,T),
$$ 
where $\mathfrak{L}$ is a general symmetric  L\'evy 
type diffusion operator. Included are both local and nonlocal
problems with e.g. $\mathfrak{L}=\Delta$ or $\mathfrak{L}=-(-\Delta)^{\frac\alpha2}$, $\alpha\in(0,2)$, and porous
medium, fast diffusion, and Stefan type nonlinearities $\varphi$. By
robust methods we mean that they converge even for nonsmooth solutions
and under very weak assumptions on the data.
We show that they are $L^p$-stable for
$p\in[1,\infty]$, compact, and convergent in
$C([0,T];L_{\textup{loc}}^p(\R^N))$ for $p\in[1,\infty)$. The first
  part of this project is given in 
\cite{DTEnJa18a} and contains the unified and easy to use theoretical
framework. This paper is devoted
to schemes and testing. We study many different problems
and many different concrete discretizations, proving that  
the results of \cite{DTEnJa18a} apply and testing the 
schemes numerically. Our examples include fractional diffusions of
different orders, and Stefan problems, porous medium, and fast
diffusion nonlinearities. Most of the convergence
results and many schemes are completely new for nonlocal
versions of the equation, including results on high order methods, the
powers of the discrete Laplacian method, and discretizations of fast
diffusions. Some of the results and schemes are new even for linear
and 
local problems.
\end{abstract}

\maketitle

%\newpage
%\small
\tableofcontents
%\newpage

\section{Introduction}
We develop a unified and easy to use framework for
fully discrete monotone numerical methods of finite difference type
for a large class of 
possibly degenerate nonlinear diffusion equations
of porous medium type:
\begin{align}\label{E}
\dell_t
u-\Operator[\varphi(u)]&=f(x,t) &&\text{in}\qquad Q_T:=\R^N\times(0,T),\\
u(x,0)&=u_0(x) &&\text{on}\qquad \R^N \label{IC},
\end{align}
where $u$ is the solution, $\varphi$ 
continuous and nondecreasing, %% $f$ some right-hand side,
and $T>0$. The diffusion operator $\Operator$ is given~as  
\begin{equation}\label{defbothOp}
\Operator:=L^\sigma+\Levymu
\end{equation} 
with local and nonlocal (anomalous) parts, 
\begin{align}\label{deflocaldiffusion}
  L^\sigma[\psi](x)&:=\text{tr}\big(\sigma\sigma^TD^2\psi(x)\big),\\
\label{deflevy}
\Levy ^\mu [\psi](x)&:=\int_{\R^N\setminus\{0\} } \big(\psi(x+z)-\psi(x)-z\cdot D\psi(x)\indikator\big) \dd\mu(z),
\end{align}
%for
where  $\sigma=(\sigma_1,....,\sigma_P)\in\R^{N\times P}$ for $P\in \N$
and $\sigma_i\in \R^N$, $D$ and $D^2$ are the gradient and Hessian,
$\indikator$ is a characteristic function, and $\mu$ is a nonnegative
symmetric measure.

\begin{remark}\label{deflevyexplained}
By symmetry of $\mu$, $\lim_{r\to0^+}\int_{r<|z|\leq 1}z\dd\mu(z)=0$,
and we have an equivalent definition of $\Levy^\mu$ in \eqref{deflevy}
in terms of a principal value integral:
$$
\Levy^\mu[\psi](x)=\lim_{r\to0^+}\int_{|z|>r}\big(\psi(x+z)-\psi(x)\big)\dd\mu(z)=\textup{P.V.}\int_{|z|>0}\big(\psi(x+z)-\psi(x)\big)\dd\mu(z).
$$
\end{remark}

The assumptions we impose on $\Operator$ and $\varphi$ are so mild
that many different problems can be modelled by \eqref{E}: Flow in
porous media, nonlinear heat transfer, phase transitions,
and population dynamics -- see e.g. \cite{Vaz07} for
local problems 
and \cite{Woy01,MeKl04,Caf12,Vaz12} for nonlocal problems. Important examples are strongly degenerate Stefan problems with
$\varphi(u)=\max(0,au-b)$, $a\geq0$, and the full range of porous 
medium and fast diffusion equations with $\varphi(u)=u|u|^{m-1}$ for
any $m\geq0$. The class of diffusion operators  $\Operator$ coincides 
with the generators of the {\em symmetric} L\'evy processes  \cite{Ber96,Sch03,App09} and
includes e.g. the Laplacian $\Delta$, fractional Laplacians $-(-\Delta)^{\frac{\alpha}{2}}$, $\alpha\in(0,2)$,
 relativistic Schr\"odinger operators $m^\alpha 
 I-(m^2I-\Delta)^{\frac{\alpha}{2}}$, and even
 discretizations of these.
 Since $\sigma$ and
 $\mu$ may be degenerate or even identically zero, problem \eqref{E}
 can be purely nonlocal, purely local, or a combination. An additional challenge
 both analytically and numerically is the fact that solutions of
 \eqref{E} in general can be very irregular and even discontinuous.

 The numerical schemes will be defined on a grid $(x_\beta,t_j)$ as follows,
\begin{equation}\label{S}
\begin{split}
U_\beta^j=U_\beta^{j-1}+\Delta t_j\big(\Levy_1^h[\varphi_1^h(U_{\cdot}^j)]_\beta+\Levy_2^h[\varphi_2^h(U_\cdot^{j-1})]_\beta+F^j_\beta\big),
\end{split}
\end{equation}
where $U_\beta^j\approx u(x_\beta,t_j)$, $\Levy^h_1+\Levy^h_2 \approx
\Operator$, $\varphi^h_i\approx \varphi$, $F_\beta^j\approx
f(x_\beta,t_j)$  and $h$ and $\Delta t_j$ are the discretization
parameters,  and the
discrete diffusion operators $\Levy^h_i$ have a monotone difference representation 
\begin{equation*}%\label{eq:GenDiscOp}
\Levy^h_i[\psi](x)=\sum_{\beta\neq0}\left(\psi(x+z_{\beta})-\psi(x)\right)\omega_{i,\beta}
\quad \text{for}\quad\omega_{i,\beta}\geq0.
\end{equation*}
As we will see, different
choices of $\varphi^h_1, \varphi^h_2,\Levy^h_1,
\Levy^h_2$ lead to explicit, implicit, $\theta$-methods, and various explicit-implicit methods. In a simple one
dimensional case,
\begin{align*}
  \dell_t u&=\varphi(u)_{xx}-(-\partial_x^2)^{\alpha/2}\varphi(u),
\end{align*}
an example of a discretization in our class is given by
\begin{align*}
U^j_m=U^{j-1}_m&+\frac{\Delta
  t}{h^2}\Big(\varphi(U^{j}_{m+1})-2\varphi(U^{j}_{m})+\varphi(U^{j}_{m-1})\Big)\\
&+\Delta
    t\sum_{k\neq0}\Big(\varphi(U^{j-1}_{m+k})-\varphi(U^{j-1}_{m})\Big)\int_{(k-\frac12)h}^{(k+\frac12)h}\frac{c_{N,\alpha}\dd z}{|z|^{N+\alpha}}.
\end{align*}

The main result of the first part of this project \cite{DTEnJa18a} was a
unified, rigorous, and easy to use theoretical framework for  these schemes. This
novel analysis includes well-posedness, $L^p$-stability, equicontinuity,
compactness, and $L^p_{\textup{loc}}$-convergence results. These
results are very general since they hold for local and 
nonlocal, linear and nonlinear, non-degenerate and degenerate, and
smooth and nonsmooth problems. An important new idea is to work in a
sufficiently  general class of  solutions of \eqref{E} that allows for atomic
 (non-absolutely continuous) measures $\mu$ in the definition of
 $\Operator$. Since the discrete operator
 $\Levy^h$ is a nonlocal operator $\Levy^\nu$ with $\nu:=\sum_{\beta\not=0}
 (\delta_{z_\beta}+\delta_{z_{-\beta}})\omega_{\beta}$, it is in the form of
 $\Operator$ and can be analyzed with the same powerful PDE
techniques. This analysis requires recent uniqueness results for \eqref{E}
obtained by the authors in \cite{DTEnJa17a,DTEnJa17c} --  results for bounded
distributional solutions or very weak solutions of \eqref{E} in the
generality needed here. The fact that we
can use such a weak notion of solution both simplifies the analysis and
makes a global theory for all the different problems
and schemes we consider here possible. At this point the reader should
note that if \eqref{E} has more regular (bounded) solutions (weak,
strong, mild, or classical), then our results still apply because
these solutions will coincide with our (unique) distributional
solution.

Schemes that converge in such general circumstances are often said to be {\em
  robust}. Consistent numerical schemes are not robust in general,
i.e. they need not always converge, or can even converge to false
solutions. Such issues are seen especially in nonlinear, degenerate and/or
low regularity problems. Our 
general results are therefore only possible because we have (i) 
identified a class of schemes with good properties (including monotonicity)
and (ii) developed the new mathematical techniques needed to analyse
these schemes in the current generality.

In this paper, which is the second part of this project, we have two main
objectives: (1) to give many concrete discretizations that fall
into the theoretical framework of the first part \cite{DTEnJa18a}, and
(2) to test and verify numerically a number of these schemes for
a wide and representative number of examples of problems of the form
\eqref{E}.

The scheme \eqref{S} is
essentially determined as soon as we specify $\Levy_i^h$ and
$\varphi^h_i$, the discretizations of $\Operator$ and $\varphi$. The
whole of Section \ref{sec:discspa} is devoted to such concrete 
discretizations. We start by splitting the diffusion operator
$\Operator$ into local, singular nonlocal, and bounded nonlocal parts, and
then explain how these parts can be discretized separately. For the
local part, we consider classical finite difference methods and in
this context new semi-Lagrangian methods. For the singular nonlocal
part, we analyse the trivial discretization and the adapted vanishing
viscosity approximation, and finally, for the bounded nonlocal part,
we consider quadrature methods obtained from interpolation in two
different ways. For the first time we apply the
so-called powers of the discrete Laplacian method (when
$\Operator=-(-\Delta)^\alpha$) to diffusion problems, and we explain how non-Lipschitz
(including fast diffusion) nonlinearities $\varphi$ 
have to be approximated to get good explicit schemes.

In every case we check that the discretizations satisfy the
conditions of the theoretical framework of \cite{DTEnJa18a}, and hence we
prove that schemes \eqref{S} involving these discretization are
$L^p$-stable for $p\in[1,\infty]$ and $L^p_{\textup{loc}}$-convergent
for $p\in[1,\infty)$. We also compute the local truncation errors, and we
  explain how to combine the methods to get better than first order
  methods for problems involving  fractional Laplace like operators and very
high order methods for bounded nonlocal operators. The powers of the
discrete Laplacian method is shown to be an order $2$ method
regardless of the value of $\alpha\in(0,2)$. Many of these schemes and
most of the convergence results are new in this context, sometimes
even in the linear case. This is especially the case for nonlocal
problems. Some important examples here are:
\begin{enumerate}[(i)]
\item the first high 
order methods for nonlinear nonlocal diffusions of porous medium type
(but see also \cite{DrJa14}), 
\smallskip
\item the first time the powers of discrete Laplacian method is applied to
  nonlinear problems, and
  \smallskip
\item the first numerical methods and
simulations for nonlocal problems with non-Lipschitz (``fast diffusion'')
nonlinearities.
\end{enumerate}
We also mention that our results provide a rigorous
justification for the numerical simulations of Section 7 in
\cite{BrChQu12} for a nonlocal Stefan problem with discontinuous
solutions, see Remark \ref{Stefan} for more details.

Numerical tests are presented in Section \ref{sec:numsim}--\ref{sec:numsim2D}.  We focus on
nonlocal problems since there are much less results for such problems
in the literature, especially for porous medium type equations. For
simplicity, we take the diffusion operator $\Operator$ to be the
fractional Laplacian $-(-\Delta)^{\frac\alpha2}$  in Section \ref{sec:numsim}.  
All the well-known one dimensional  special cases of \eqref{E} are then considered:
the linear fractional heat equation, the fractional porous medium
equation \cite{DPQuRoVa12}, and fractional equations with fast diffusion and
Stefan type nonlinearities \cite{DPQuRoVa12,BrChQu12,A-VMaRoT-M10}. In each problem we test and compare four different
numerical schemes for different powers $\alpha$ of the
fractional Laplacian. Most test problems are set up to have smooth
exact solutions, and the numerical tests confirm the
theoretical results,  in most
cases also including the truncation error bounds and the expected
convergence rates. 

Note that in the Stefan case, we expect that $\varphi(u)\in C^\gamma$ for some $\gamma\in(0,1]$, but this is not enough to ensure that $-(-\Delta)^{\frac{\alpha}{2}}[\varphi(u)]$ exists pointwise when $\alpha\geq \gamma$.
If this is the case, the scheme does not converge in $L^\infty$, but
it still converges in $L^1$ with the expected rates given by the local
truncation error. See Section \ref{sec:StefanTyepProblemExplicit} for the details.
 
 We also produce a Stefan type
example where the numerical solution is nondifferentiable. Finally, in
Section \ref{sec:numsim2D} we test a much more complicated problem in
two dimensions: A Stefan problem with degenerate local and
nonlocal diffusion and nonsmooth castle like initial data. 
 
To perform the numerical computations mentioned above, we have
restricted the scheme to a (large) bounded domain and set the numerical solution
equal to zero outside. Convergence of the scheme then requires the size
of the computational domain to increase as the grid is refined. We
briefly discuss the error introduced by the restriction to a bounded domain in Section~\ref{sec:TruncationDomainEffect}.

\smallskip
\noindent {\bf Related work.}
In the  local linear case, when $\varphi(u)=u$ and $\mu\equiv0$ in
\eqref{E}, numerical methods and analysis can be found in
undergraduate text books. In the nonlinear case there is a very large
literature so we will focus only on some developments that are more
relevant to this paper. For porous medium nonlinearities
($\varphi(u)=u|u|^{m-1}$ with $m>1$), there are early results on
finite element and finite-difference interface tracking methods in
\cite{Ros83} and \cite{DBHo84} (see also \cite{Mon16}). There is
extensive theory for finite volume schemes, see \cite[Section
  4]{EyGaHe00} and references therein for equations with locally
Lipschitz $\varphi$. For finite element methods there is a number of
results, including results for fast diffusions ($m\in(0,1)$), Stefan
problems, convergence for strong and weak solutions, discontinuous
Galerkin methods, see
e.g. \cite{RuWa96, EbLi08,EmSi12, DuAlAn13,ZhWu09,NoVe88,MiSaSu05}. 
Note that the latter paper considers the general form of \eqref{E}
with $\Operator=\Delta$ and provides a convergence analysis in $L^1$. 
A number of results on finite
difference methods for degenerate convection-diffusion equations also yield
results for \eqref{E} in special cases, see e.g. \cite{EvKa00,BuCoSe06,KaRiSt16,JePa17}. 
In particular the results of \cite{EvKa00,KaRiSt16} imply our
convergence results for a particular scheme when $\varphi$ is
locally Lipschitz,  $\Operator=\Delta$, and solutions have a certain
additional BV regularity. Finally, we mention very
general results on so-called gradient schemes \cite{DrEyGaHe13,DrEyHe16} for
doubly or triply degenerate parabolic equations. This class of
equations include local porous medium type equations as a special case.

In the nonlocal case, the literature is more recent and not so
extensive. For the linear case we refer somewhat arbitrarily to \cite{CoTa04,HuOb14, HuOb16, NoOtSa15}  and references therein. Here we also
mention \cite{CuDTG-GPa17a} and its novel finite element plus semigroup
subordination approach to
discretizing $\Operator=-(-\Delta)^{\frac{\alpha}{2}}$.
Some early
results for nonlocal problems came from  finite difference quadrature
schemes for Bellman equations and fractional conservation laws, see
\cite{JaKaLC08,CaJa09,BiJaKa10} and \cite{Dro10}. For the latter case
discontinuous Galerkin and spectral methods were later studied in
\cite{CiJaKa11,CiJa13,XuHe14}. The first results that include
nonlinear nonlocal versions of \eqref{E} was
probably given in \cite{CiJa11}. There,  convergence of finite
difference quadrature schemes was proven for a convection-diffusion
equation. This result is extended to more general equations and error
estimates in \cite{CiJa14} and a higher order discretization in
\cite{DrJa14}. In some cases our convergence results follow
from these 
results (for two particular schemes, $\sigma=0$, and $\varphi$ locally
Lipschitz). However, the analysis there is different and more 
complicated since it involves entropy solutions and Kru\v{z}kov doubling
of variables arguments.

In the purely parabolic case \eqref{E}, the behaviour of the solutions
and the underlying theory is different from the convection-diffusion
case (especially so in the nonlocal case, see e.g. \cite{DPQuRoVa11, DPQuRoVa12, Vaz14,DPQuRo16,DPQuRoVa17} and \cite{DrIm06, ChCzSi10, AlAn10,CiJa11,AlCiJa14,IgSt18}). It is therefore important to develop numerical methods and
analysis that are specific for this setting. The first (nonlocal)
results in this direction seems to be  \cite{DTe14, DTVa14}. These papers are based on the extension method
\cite{CaSi07}, and introduce and analyze finite difference methods for the Fractional Porous Medium
Equation. The present work is possibly the
first not to use the extension method or the regularity of the solution.

\smallskip

\noindent{\bf Outline.} The next section is a short section where we
collect the assumptions and well-posedness results for the porous
medium type equation \eqref{E}. In section \ref{sec:prevRes} we
formulate the numerical schemes and state the main theoretical
results. This is a slightly 
simplified version of the theoretical framework of part 1 of this project
\cite{DTEnJa18a}. We also give a couple of new results that will greatly
simplify the verification of the assumptions of this framework. The
main contributions of this paper are then given in the two sections
that follows. In Section \ref{sec:discspa} we introduce the
concrete 
discretizations and prove rigorously that they fall into our
theoretical framework, while in Sections \ref{sec:numsim}--\ref{sec:numsim2D}  we present
our numerical simulations for all the well-known special cases of \eqref{E}.

\section{Preliminaries}
\label{sec:assumptionsprelim}

In this section we present the assumptions and well-posedness results
for the initial value problem \eqref{E}--\eqref{IC}. 
In this paper we work in the setting of bounded distributional
solutions. This is very convenient for numerical analysis 
since it leads to an easy to work with convergence theory that applies
even to very bad problems. By uniqueness it also applies to situations
where solutions are more regular, e.g. classical, strong, weak/energy, or mild
solutions.

Following \cite{DTEnJa17c} (see also \cite{BrCr79,DTEnJa17a}) we use the assumptions:
\begin{align}
&\varphi:\R\to\R\text{ is nondecreasing and continuous};
\tag{$\textup{A}_\varphi$}&
\label{phias}\\
&f\text{ is measurable and}  \int_0^T\|f(\cdot,t)\|_{L^1(\R^N)}+\|f(\cdot,t)\|_{L^\infty(\R^N)}\dd t<\infty ;
\label{gas} \tag{$\textup{A}_f$} \\
&u_0\in L^1(\R^N)\cap L^{\infty}(\mathbb{R}^N);\text{ and}
\tag{$\textup{A}_{u_0}$}&
\label{u_0as}\\
&\label{muas}\tag{$\textup{A}_{\mu}$} \mu \text{ is a nonnegative symmetric Radon measure on
}\R^N\setminus\{0\}
\text{ satisfying}
\nonumber\\ 
&\quad\int_{|z|\leq1}|z|^2\dd \mu(z)+\int_{|z|>1}1\dd
\mu(z)<\infty\nonumber.
\end{align}
Sometimes we will need stronger assumptions than \eqref{phias} and
\eqref{muas}:
\begin{align}
&\label{philipas}\tag{$\textup{Lip}_\varphi$} \varphi:\R\to\R\text{ is
    nondecreasing and locally Lipschitz.}\\
&\label{alp}\tag{$\textup{A}_{\mu_\alpha}$} \text{There are constants
    $\alpha\in(0,2)$ and $C\geq0$ such that for all $r\in(0,1)$,}\\
 &\nonumber \int_{|z|<r}|z|^k\dd\mu(z)\leq
  Cr^{k-\alpha},\ k=2, 3, 4,\quad\text{and}\quad\int_{r<|z|<1}\dd\mu(z)\leq
  Cr^{-\alpha}.\\ 
&\label{nuas}\tag{$\textup{A}_{\nu}$} \nu \text{ is a nonnegative symmetric Radon measure satisfying } \nu(\R^N)<\infty. \nonumber
\end{align}
\begin{remark}\label{addingconstants}
\begin{enumerate}[{\rm (a)}]
\item
  Without loss of generality, we can assume $\varphi(0) = 0$
  (replace $\varphi(u)$ by $\varphi(u)-\varphi(0)$), and when
  \eqref{philipas} holds, that $\varphi$ is globally Lipschitz (since
  $u$ is bounded). In the latter case we let $L_\varphi$ denote the
  Lipschitz constant.
  \smallskip
\item Under assumption \eqref{muas}, for any $p\in[1,\infty]$ and any
  $\psi\in C_\textup{c}^\infty(\R^N)$,
\begin{equation}\label{eq:boundOp}\|\Operator[\psi]\|_{L^p}\leq c\|D^2\psi\|_{L^p}
  \Big(|\sigma|^2+\int_{|z|\leq1}|z|^2\dd\mu(z)\Big)+ 2\|\psi\|_{L^p}
  \int_{|z|>1}\dd\mu(z).\end{equation}
\item When $\mu$ is absolutely  continuous w.r.t. the Lebesgue measure
  $\dd z$, assumption \eqref{alp} means
  that $\dd \mu(z)\leq \frac{C}{|z|^{N+\alpha}}\dd z$ for
  $|z|<1$.
  The nonlocal operator $\Levy^\mu$ then typically would be a fractional differential
  operator of order $\alpha\in(0,2)$ (a pseudo differential operator ), like e.g. the fractional Laplacian
  $(-\Delta)^{\frac\alpha2}$.
  \smallskip
\item Assumption \eqref{gas} is equivalent to requiring 
$
f\in L^1(0,T;L^1(\R^N)\cap L^\infty(\R^N))
$, an iterated $L^{p}$-space as in e.g.
\cite{BePa61}. Note that $L^1(0,T;L^1(\R^N))=L^1(Q_T)$.
\end{enumerate}
\end{remark}

\begin{definition}[Distributional solution]\label{distsol} 
Let  $u_0\in L_\textup{loc}^1(\R^N)$ and  $f\in
L_\textup{loc}^1(Q_T)$. Then $u\in L_\textup{loc}^1(Q_T)$ is a
distributional (or very weak) solution of \eqref{E} if for all $\psi\in
C_\textup{c}^\infty(\R^N\times[0,T))$, $\varphi(u)\Operator[\psi]\in
L^1(Q_T)$ and
\begin{align}\label{def_eq}
\int_{0}^{T}\int_{\R^N} \big(u\dell_t\psi+ \varphi(u)\Operator[\psi]+f\psi\big)\dd x \dd t+\int_{\R^N}u_0(x)\psi(x,0)\dd x=0
\end{align}
\end{definition}
By Remark \ref{addingconstants} (b), $\varphi(u)\Operator[\psi]\in
L^1$ if e.g.~$u\in L^\infty$ and $\varphi$ continuous. Distributional solutions exist  and are unique in $L^1\cap L^\infty$.
\begin{theorem}[Theorem 2.8 in \cite{DTEnJa17a} and Theorem 3.1 \cite{DTEnJa17c}]\label{unique}
 Assume \eqref{phias}, \eqref{gas}, \eqref{u_0as}, and
\eqref{muas}. Then there exists a unique distributional solution $u$
of \eqref{E}-\eqref{IC} such that 
\begin{equation}\label{eq:regsol}
u\in L^1(Q_T)\cap L^\infty(Q_T)\cap C([0,T];L_\textup{loc}^1(\R^N)).
\end{equation}
 \end{theorem}

Note that by \eqref{def_eq} and \eqref{eq:regsol}, $u(x,t)\to u_0(x)$ in $L_\textup{loc}^1(\R^N)$ as $t\to 0^+$.

\section{Numerical schemes -- general theory}
\label{sec:prevRes}

 We introduce and discuss the class of numerical methods that
we consider and state the main results about well-posedness, stability,
equicontinuity, compactness, and convergence. The proofs of most 
results in this section are given in \cite{DTEnJa18a}.

\subsection{The numerical method}
\label{ssec:num}
Our schemes will be defined on time-space grids, nonuniform
in time, but uniform in space
for simplicity. Our discrete diffusion operators will
then have weights and stencils not depending on the position $x$. 
Let $h>0$, the cube $R_h=h(-\frac{1}{2},\frac{1}{2}]^N$, and $\Grid$ be the
  uniform spatial grid
\[
\Grid:= h\Z^N=\{x_\beta:=h \beta: \beta\in \Z^N\}.
\]
The nonuniform time grid is
$$\mathcal{T}_{\Delta t}^T=\{t_j\}_{j=0}^{J}\qquad\text{for}\qquad
0=t_0<t_1<\ldots <t_J=T.$$
Let
$\J:=\{1,\ldots,J\}$, and denote the time steps by
\begin{equation*}\label{def:timesteps}
\Delta t_j=t_j-t_{j-1}, \quad j\in \J, \qquad \textup{and}  \qquad \Delta t=\max_{j\in \J}\Delta t_j.
\end{equation*}

On the grid $\Grid \times \mathcal{T}_{\Delta t}^T$ we define a class of numerical approximations of
\eqref{E} by discretizing in time and space using monotone finite
difference (quadrature) approximations. Using a 
$\theta$-method in time, the resulting scheme can be written as
\begin{equation}\label{FullyDiscNumSch1}
U_\beta^j=U_\beta^{j-1}+\Delta t_j\big(\theta \Levy^h[\varphi(U_{\cdot}^j)]_\beta+(1-\theta)\Levy^h[\varphi^h(U_\cdot^{j-1})]_\beta+F_\beta^j\big)\\
\end{equation}
for $\beta\in \Z^N$ and $j\in\J$ where the discrete diffusion operator
$\Levy^h$ is given by
\begin{equation}\label{FD}\tag{FD}
\Levy^h[\psi](x)=\sum_{\beta\neq0}\left(\psi(x+z_{\beta})-\psi(x)\right)\omega_{\beta,h}\quad\text{with}\quad
z_\beta\in \Grid.
\end{equation} 
We will always assume $z_\beta=-z_{-\beta}$,
$\omega_{\beta,h}=\omega_{-\beta,h}\geq0$, and
$\sum_{\beta\not=0}\omega_{\beta}<+\infty$, see Definition \ref{def:suitdisc} and
Lemma \ref{lemma:suitdisc} below. Then $\Levy^h$ is a monotone  finite
difference operator with {\em stencil} $\mathcal
S=\{z_\beta\}_{\beta}$ and {\em weights} 
$\{\omega_{\beta,h}\}_\beta$. Note that the scheme is explicit when
$\theta=0$, implicit when $\theta=1$, and Crank-Nicholson like when
$\theta=\frac12$. 

Formally we want $U_\beta^j\approx u(x_\beta,t_j)$, $\Levy^h \approx
\Operator$, $\varphi^h\approx \varphi$ and $F_\beta^j\approx
f(x_\beta,t_j)$. For $\Levy^h$ and $\varphi^h$ this means that we have
to impose {\em consistency} assumptions, see Definition
\ref{def:suitdisc}(ii) and Definition \ref{phidef}(ii) below. But
since $u$ and $f$ need not be continuous and point values are not always
defined or useful, we will interpret $U$ and $F$ as piecewise
polynomial approximations. In this paper we restrict to piecewise constant
approximations defined from cell-averages for simplicity. 
Hence as initial data for the scheme we take
\begin{equation*}\label{averaged1}
U^0_\beta:=\frac{1}{h^N}\int_{x_\beta+R_h}u_0(x)\dd x,\quad\  F^j_\beta:=\frac{1}{h^N \Delta t_j}\int_{t_j-\Delta t_j}^{t_j}\int_{x_\beta+R_h}f(x,\tau)\dd x \dd \tau.
\end{equation*}
Of course if $f$ and $u_0$ are continuous, we could use
$U^0_\beta:=u_0(x_\beta)$ and $F_\beta^j=f(x_\beta,t_j)$ instead  and all the results below would remain valid.

\subsection{The discretizations $\Levy^h$ and $\varphi^h$}
An admissible discretization $\Levy^h$ of $\mathfrak L$ should be
(i) monotone,  symmetric, 
(ii) consistent, and (iii) satisfy some uniform Levy integrability condition
(which is trivial in the local case). 
In the next definition we will use that $\Levy^h=\Levy^{\nu_h}$ where
$\Levy^{\nu_h}$ is a Levy operator like $\mathfrak L$ defined as
\[\tag{FD2}\label{FD2}
\Levy^{\nu_h}[\psi]:= \int_{|z|>0}(\psi(x+z)-\psi(x))\dd \nu_h(z) \quad \textup{with} \quad \nu_h(z)=\sum_{\beta\not=0} \delta_{z_\beta}(z) \omega_{\beta,h}.
\]
This surprising observation along with the sufficiently general
well-posedness result in Section \ref{sec:prevRes}, are key
ingredients that make our theory work.

\begin{definition}\label{def:suitdisc}
A family $\{\Levy^h\}_{h>0}$ of discretizations of  $\mathfrak L$ is
\textbf{admissible} if it is
\smallskip
\begin{enumerate}[{\rm (i)}]
\item {\bf{in the class \eqref{nuas}}:} $\Levy^h=\Levy^{\nu_h}$ for a measure $\nu_h$ satisfying \eqref{nuas} for all $h>0$,\!\!
\smallskip
\item {\bf consistent:} for every $\psi \in C_c^\infty(\R^N)$,
\[
\|\mathfrak{L}[\psi]-\Levy^h[\psi]\|_{L^1(\R^N)} \to 0 \quad \textup{as} \quad h\to0^+,
\]
\item {\bf uniformly in \eqref{muas}:}  
\begin{equation}\label{eq:unifLevy}\tag{UL}
\sup_{h<1}\sum_{\beta\not=0} (|z_\beta|^2\wedge 1 )\omega_{\beta,h}<+\infty.
\end{equation}

\end{enumerate}
\end{definition}
Note that $\|\mathfrak L[\psi]-\Levy^h[\psi]\|_{L^1(\R^N)}$ is the
{\em Local Truncation Error} of $\Levy^h$ (in $L^1$), and that in view of
\eqref{FD2}, condition \eqref{eq:unifLevy} can equivalently be written as
\begin{equation}\tag{UL2}\label{eq:unifLevy2}
\sup_{h<1}\int_{|z|>0} |z|^2 \wedge1\ \dd \nu_h(z)<+\infty.
\end{equation}
\begin{lemma}\label{lemma:suitdisc}
The operators $\{\Levy^h\}_{h>0}$ defined in
\eqref{FD} are in the class \eqref{nuas} if and
only if $z_\beta=-z_{-\beta}$,
$\omega_{\beta,h}=\omega_{-\beta,h}\geq0$, and
$\sum_{\beta\not=0}\omega_{\beta}<+\infty$.
\end{lemma}
\begin{proof}
Since $\nu_h(z)=\sum_{\beta\not=0} \delta_{z_\beta}(z)
\omega_{\beta,h}$, equivalence for the symmetry and nonnegativity part
of \eqref{nuas} follows immediately. Equivalence for the boundedness follows from
$\nu_h(\R^N)=\sum_{\beta\not=0} \delta_{z_\beta}(\R^N) \omega_{\beta,h}=\sum_{\beta\not=0} \omega_{\beta,h}.$
\end{proof}

Assumption \eqref{eq:unifLevy} may
seem unusual, but it is in fact very natural in view of
\eqref{muas}. It is trivial  
to verify for local problems, and we now provide a very easy to use
sufficient condition for it to hold in the general case.

\begin{proposition}\label{prop:equivcons}
Assume \eqref{muas}, $\mathfrak L$ is defined by
\eqref{defbothOp}--\eqref{deflevy}, and $\{\Levy^{\nu_h}\}_{h>0}$
defined by \eqref{FD} is in the class \eqref{nuas}. Then
\eqref{eq:unifLevy2} holds if   
\begin{equation}\label{eq:pointwiseconv}
\Levy^{\nu_h}[\psi](x)\stackrel{h\to0^+}{\longrightarrow} \mathfrak L[\psi](x)
 \qquad  \textup{for all }\qquad x\in\R^N,\quad \psi
\in C_\textup{c}^\infty(\R^N).
\end{equation}
\end{proposition}

\begin{remark}

  \eqref{eq:pointwiseconv} follows e.g. from $L^\infty$-consistency, $\| \mathfrak L[\psi]
  -\Levy^{\nu_h}[\psi]\|_{L^\infty(\R^N)}\to0$ as $h\to 0^+$ for all
  $\psi \in C_\textup{c}^\infty(\R^N)$.
\end{remark}

\begin{proof}[Proof of Proposition \ref{prop:equivcons}]
By \eqref{FD2}, the  Taylor expansion
$$
\psi(x+z)=\psi(x)+z\cdot D\psi(x)+\int_0^1 (1-t) z^T D^2 \psi (x+t z) z \dd t,
$$
and since $\int_{|z|<1}z \dd\nu_h(z)=0$ by the symmetry of $\nu_h$, we find
that  
\begin{equation*}\label{eq:altdefL}
\begin{split}
\Levy^{\nu_h}[\psi](x)=&\int_{|z|\leq1}\int_0^1 (1-t) z^T D^2 \psi (x+t z) z \dd t \dd \nu_h(z)\\
&+\int_{|z|>1}(\psi(x+z)-\psi(x) )\dd \nu_h(z).
\end{split}
\end{equation*}
Then we take $\psi\in C_c^\infty$ such that
$\psi(x)=-1+|x|^2$ for 
  $|x|\leq1$ and $\psi(x)\geq0$ for $|x|>1$.
 Since $|tz|<1$ in the first integral above, 
\begin{equation*}
\begin{split}
\Levy^{\nu_h}[\psi](0)&=\int_{|z|\leq1}|z|^2 \dd \nu_h(z)+\int_{|z|>1}\psi(z)\dd\nu_h(z)-\psi(0)\int_{|z|>1}\dd\nu_h(z) \\
&\geq \int_{|z|\leq1}|z|^2 \dd \nu_h(z)+0+\int_{|z|>1}\dd\nu_h(z)\geq \int_{|z|>0} (|z|^2 \wedge 1) \dd \nu_h(z).
\end{split}
\end{equation*}
By \eqref{muas}, the bound \eqref{eq:boundOp} holds, and then by
\eqref{eq:pointwiseconv} we conclude that 
\begin{equation*}
\begin{split}
\sup_{h<1} \int_{|z|>0} (|z|^2 \wedge 1) \dd \nu_h(z)&\leq \sup_{h<1} \Levy^{\nu_h}[\psi](0) \\
&\leq |\mathfrak L[\psi](0)|+\sup_{h<1}|(\Levy^{\nu_h}-\mathfrak L )[\psi](0)|<+\infty.
\end{split}
\end{equation*}
The proof is complete.
\end{proof}

In most situations we can simply take the nonlinearity
$\varphi_h=\varphi$, but sometimes it is useful to approximate also
$\varphi$. We will see below that this is true especially for fast diffusions.

\begin{definition}\label{phidef}
  A family $\{\varphi^h\}_{h>0}$ of approximation of $\varphi$ is
  \textbf{admissible} if it is
\begin{enumerate}[{\rm (i)}]
\item {\bf{in the class \eqref{philipas}}} for every $h>0$, $\varphi^h$ satisfy \eqref{philipas},
\item {\bf{consistent}:}  $\varphi^h\to \varphi$ locally uniformly as $h\to0^+$.
\end{enumerate}
\end{definition}

\subsection{CFL condition for the explicit part}
\label{sec:CFL}
A crucial property in our convergence analysis is monotonicity. Our
schemes are monotone under the CFL condition:
\begin{equation}\label{CFL}\tag{CFL}
\Delta t(1-\theta)L_{\varphi^h} \nu^h(\R^N)\leq1,
\end{equation}
where $L_{\varphi^h}$ denotes the Lipschitz constant of
$\varphi^h$. Note that the condition always holds if $\theta=1$ and
the scheme is implicit. If the scheme has some explicit part,
$\theta\in[0,1)$, then this condition gives a relation between $\Delta
  t$ and $h$. In the local case, we typically have  $\nu_h(\R^N)=O(h^{-2})$, and
  the CFL condition becomes the classical
  $$\Delta t\leq Ch^2.$$
In the (explicit and) nonlocal case, typically
$\nu_h(\R^N)=O(h^{-\alpha })$ for some $\alpha\in(0,2)$ (e.g $\nu_h\sim |z|^{-N-\alpha}$), $\nu_h(\R^N)=O(|\log (h)|)$   (e.g $\nu_h\sim |z|^{-N} e^{-|z|}$) or $\nu_h(\R^N)<\tilde{C}$ (e.g $\nu_h\sim f \in L^1(\R^N)$), and the CFL
condition becomes
$$\Delta t\leq Ch^\alpha, \quad \Delta t\leq C \frac{1}{|\log (h)|} \quad \textup{or} \quad \Delta t\leq C.$$
 We refer to \cite{DTEnJa18a} for more details and
  the origin of such conditions.

\begin{remark}
Note that $\varphi^h$  has to be Lipschitz for the CFL condition to
make sense. Hence if $\varphi$ is not Lipschitz (the fast
diffusion case), it must be replaced by a Lipschitz approximation to
obtain a monotone explicit scheme. The
Lipschitz constant $L_{\varphi^h }$ will then blow up as $h\to 0$, and
the overall CFL condition is worse than in the Lipschitz case. See
Section \ref{explFD} for examples.
\end{remark}

\subsection{Comparison, stability, and convergence of the method}
By \cite{DTEnJa18a} our numerical method has the following list of properties.

\begin{theorem}[Existence and uniqueness]\label{thm:existuniqueNumSchFully}
Assume \eqref{gas}, \eqref{u_0as} and \eqref{phias}, $\Levy^h$ defined
in \eqref{FD} satisfies \eqref{nuas}, $\varphi^h$ is in the class
\eqref{philipas}, and $h,\Delta t>0$ are such that \eqref{CFL} holds. Then there exists a unique solution $U_\alpha^j$ of \eqref{FullyDiscNumSch1} such that
$$
\sum_{j\in\J}\sum_{\beta}|U_\alpha^j|<\infty.
$$

\end{theorem}
\begin{theorem}[Properties and convergence] \label{thm:fullydisc}
  Assume \eqref{muas}, \eqref{phias}, $u_0,v_0$ satisfy \eqref{u_0as},  $f,g$
  satisfy \eqref{gas}, $\{\Levy^h\}_{h>0}$ and $\{\varphi^h\}_{h>0}$
  are admissible approximations of $\mathfrak L$ and $\varphi$, $\Delta
t=o_h(1)$ such that \eqref{CFL} holds, and $U_\beta^j,V_\beta^j$ are
solutions of the scheme \eqref{FullyDiscNumSch1} with data $u_0,v_0$
and $f,g$. 
Then
\begin{enumerate}[{\rm (a)}]
\addtocounter{enumi}{1}
\item \textup{(Monotone)} If
  $U^0_\beta\leq V^0_\beta$ and $F^j_\beta\leq G^j_\beta$ for all
  $\beta$, \ then
  $U_\beta^{j}\leq V_\beta^{j}$ for all $\beta$, $j\geq0$.
\item  \textup{($L^1$-stable)}
  $\displaystyle \sum_\beta |U_\beta^j|\leq \sum_\beta |U_\beta^0|+
  \sum_{l=1}^j\sum_\beta |F_\beta^l| \Delta t_{l}$.
\item  \textup{($L^\infty$-stable)}
  $\displaystyle \sup_\beta |U_\beta^j|\leq \sup_\beta |U_\beta^0|+
  \sup_\beta \sum_{l=1}^j|F_\beta^l| \Delta t_{l}$.
\item \textup{(Conservative)} If $\varphi^h$   satisfy \eqref{philipas}, 
$\displaystyle \sum_\beta U_\beta^j=\sum_\beta U_\beta^0+ \sum_{l=1}^j
  \sum_\beta F_\beta^l\Delta t_{l}.$
\item \textup{($L^1$-contractive)}
  $\displaystyle 
\sum_\beta (U_\beta^j-V_\beta^j)^+
\leq \sum_\beta (U_\beta^0-V_\beta^0)^+ + \sum_{l=1}^j\sum_\beta
(F_\beta^l-{G}^l_\beta)^+ \Delta t_{l}.$
\item \textup{(Equicontinuity in
  time)}  For all compact sets
  $K\subset \R^N$ there exists a modulus of continuity $\Lambda_K$
  (independent of $h$ and $\Delta t$) such that 
\begin{equation*}
\begin{split}
h^N\sum_{x_\beta \in \Grid\cap K} |U_\beta^j-U_\beta^{j-k}|\leq \Lambda_K(t_j-t_{j-k})+|K|\int_{t_{j-k}}^{t_j}\|f(\cdot,\tau)\|_{L^\infty(\R^N)}\dd \tau.
\end{split}
\end{equation*}
\item \textup{(Convergence)} There
exists a unique distributional solution $u\in L^1(Q_T)\cap
L^\infty(Q_T)\cap C([0,T];L_\textup{loc}^1(\R^N))$ of \eqref{E}-\eqref{IC}
and for all compact sets $K\subset \R^N$, 
\begin{equation*}
\vertiii{U-u}_{K}:=\max_{t_j\in\GridT}\left\{\sum_{x_\beta \in \Grid\cap K} h^N |U_{\beta}^j-u(x_\beta,t_j)|\right\} \to 0 \quad  \text{as} \quad h\to0^+.
\end{equation*}
\end{enumerate}
\end{theorem}

Note that our schemes are stable in $L^p$ for any $p\in[1,\infty]$ by
interpolation. The discrete norm convergence results is equivalent to
convergence in $C([0,T];L^1_{\textup{loc}}(\R^N))$ for interpolants of the
numerical solution (piecewise constant in space and piecewise linear
in time), see \cite{DTEnJa18a} for the details.
Convergence was proved through a compactness argument in this space,
where equicontinuity results like Theorem \ref{thm:fullydisc} (f) and
(g) were needed. By the $L^p$ stability, convergence also holds in $C([0,T];L_\textup{loc}^p(\R^N))$ for all $p\in[1,\infty)$.

\subsection{Some extensions}
As shown in \cite{DTEnJa18a}, the results of Theorem \ref{thm:existuniqueNumSchFully} and Theorem \ref{thm:fullydisc} also hold for a larger class of schemes,
\begin{equation*}
\begin{split}
U_\beta^j=U_\beta^{j-1}+\Delta t_j\bigg(\sum_{k=1}^n\Levy_k^h[\varphi_k^h(U_{\cdot}^j)]_\beta+\sum_{k=n+1}^m\Levy_k^h[\varphi_k^h(U_\cdot^{j-1})]_\beta+F^j_\beta\bigg),
\end{split}
\end{equation*}
where $n,m\in\N$ with $n\leq m$ and
$$
\sum_{k=1}^n\Levy_k^{h}[\varphi_k^h(U_h^j)](x)+\sum_{k=n+1}^m\Levy_k^{h}[\varphi_k^h(U_h^{j-1})](x)\approx
\mathfrak L[\varphi(u)](x,t_j).
$$
Depending on the choices of
   $\Levy_k^{h}$ and $\varphi^h_k$, we can then get many different
   schemes:
\begin{enumerate}[{\rm (1)}]
 \item Discretizing separately the different parts of the operator
   $$\mathfrak L=\Operator=L^\sigma+\mathcal L^\mu_{\text{sing}}+\mathcal
   L^\mu_{\text{bnd}},$$
   e.g. the local, singular nonlocal, and bounded nonlocal
   parts, corresponds to different choices for  $\Levy_k^h$. See
   Section \ref{sec:discspa} for a detailed discussion.\smallskip  
\item Explicit schemes ($\theta=0$),
   implicit schemes ($\theta=1$), or combinations like Crank-Nicholson
   ($\theta=\frac12$), follow by the choices
   $$\Levy_1^{h}=\theta\Levy^{h}\qquad\text{and}\qquad
   \Levy_2^{h}=(1-\theta)\Levy^{h}.$$
 \item Combinations of type (1) and (2) schemes,
   e.g. implicit discretization of the unbounded part of $\Operator$
   and explicit discretization of the bounded part.
\end{enumerate}

\section{Numerical schemes -- discretizations}
\label{sec:discspa}

In this section we explore known and find new approximations of
$\mathfrak L$ and $\varphi$, and in every case we show that they are
admissible in the sense of Definitions \ref{def:suitdisc} and
\ref{phidef} (see also Lemma \ref{lemma:suitdisc}) and hence yield
monotone, stable, and convergent numerical schemes 
for \eqref{E} -- \eqref{IC} by Theorem \ref{thm:fullydisc}.
Many of these schemes are completely new in this setting, and even for
many of the known schemes, the convergence results are new.
For the diffusion
operator $\mathfrak L=\Operator=\Levy^\mu + L^\sigma$, we present a
series of possible discretizations of both the nonlocal part $\Levymu$
of the form \eqref{deflevy} satisfying \eqref{muas} and the local
second order elliptic operator $L^\sigma$ given by
\eqref{deflocaldiffusion}. Most of these discretizations apply  to
all such operators and are not restricted to the fractional
Laplacian/Laplacian. We end the exposition by discussing how to 
handle non-Lipschitz merely continuous nonlinearities $\varphi$, and
hence also fast diffusions. 

The nonlocal operator $\Levy^\mu$ has a possibly singular part and a nonsingular
part that can (and often should) be discretized separately: for 
$r>0$,

\begin{equation*}
\begin{split}
\Levy^\mu[\psi](x)&=\int_{0<|z|\leq r}\big(\psi(x+z)-\psi(x)-z\cdot
D\psi(x)\big)\dd \mu(z)\\
&\quad +\int_{|z|>r}\big(\psi(x+z)-\psi(x)\big)\dd \mu(z)\\
&=:\Levy_{r}^\mu[\psi](x)+\Levy^{\mu,r}[\psi](x).
\end{split}
\end{equation*}
 
When we discretize this operator, we have to take $h\leq r=o_h(1)$
where $h$ is the discretization in space parameter. Often we can
simply take $r=h$, but in some cases a different choice can produce
higher order discretizations. 
 We will present admissible discretizations for general measures $\mu$
 and state their Local Truncation Error (LTE). We also give the LTE
 when the nonlocal operator is a fractional derivative in the sense of
 \eqref{alp}.
  
\subsection{Lagrange interpolation}\label{subsec:interp}
Let $\{p^k_\beta\}_{\beta\in\Z^N}$ be the basis of piecewise
Lagrange polynomials of order $k$ on the grid $\Grid$, i.e. $\sum_{\beta} p_{\beta}^k(x)\equiv1$ for all $x\in\R^N$ and
$p_\beta^k(z_\gamma)=1$ for $\beta=\gamma$ and zero otherwise. Since
the grid is uniform, we may define these functions in a tensorial way
for $N>1$ (direction by direction). On $\Grid$ the Lagrange polynomial
interpolant of order $k$ of a function $\psi$ is given by
\begin{equation}\label{eq:LagInterp}
I_h^k[\psi](z):=\sum_{\beta\neq0}\psi(z_\beta)p_\beta^k(z),
\end{equation}
and if $\psi \in C_c^\infty(\R^N)$, the corresponding interpolation error is
\begin{equation}\label{eq:LagInterpError}
\|I_h^k[\psi] -\psi\|_{L^p(\R^N)}=C \|D^{k+1}\psi\|_{L^p(\R^N)} h^{k+1}
\end{equation}
where $C=C(k,p)$ and $p=\{1,\infty\}$ (cf. e.g. \cite{CiRa72}).
Since the grid is uniform, the $p^k_\beta$-basis will have a lot of symmetries. E.g. when
$k=1$ (linear interpolation), $0\leq
p^1_{\beta}(x)=p^1_{0}(x-z_\beta)$ for $z_\beta=\beta h\in\Grid$ and
$\beta\in\Z^N$, and
$p^1_0(x_1,\dots,x_i,\dots,x_N)=p^1_0(x_1,\dots,-x_i,\dots,x_N)$ for
$x_i\in\R$ and $i=1,\dots,N$.

%%%%%%%%%%%%%%%%%%%%%%%%%%%%%%%%%%%%%%%%%%%%%%%%%%%%

\subsection{Discretizations of the local operator
  $L^\sigma$}\label{sec:discLocal} The operator $ L^\sigma$ given by
\eqref{deflocaldiffusion} is a \emph{local}, self-adjoint, and
possibly degenerate operator that can be written as 
\begin{equation}\label{localalt}
L^\sigma [\psi](x):=\text{tr}\big(\sigma\sigma^TD^2\psi(x)\big)=\sum_{i=1}^P \sigma_i^T D^2\psi(x) \sigma_i
  =\sum_{i=1}^P (\sigma_i^TD)^2\psi(x)
\end{equation}
where $\sigma=(\sigma_1,....,\sigma_P)\in\R^{N\times P}$ and $\sigma_i\in \R^N$.
For $\eta>0$, we approximate $L^\sigma$ by
\begin{equation}\label{AproxLoc}
L^\sigma_\eta[\psi](x):=\sum_{i=1}^P \frac{\psi(x+\sigma_i \eta) + \psi(x-\sigma_i \eta)-2\psi(x)}{\eta^2}.
\end{equation}
In general $x+\sigma_i\eta \not \in \Grid$, not even when $x\in
\Grid$, and hence this discretization is in the form \eqref{FD} only for special
choices of $\sigma$.
We can overcome this problem by replacing $\psi$ by its interpolant on $\Grid$,
\begin{equation}\label{AproxLoc2}
L^\sigma_{\eta,h}[\psi](x)=\sum_{i=1}^P \frac{I_h^1[\psi](x+\sigma_i \eta) + I_h^1[\psi](x-\sigma_i \eta)-2\psi(x)}{\eta^2}
\end{equation}
where $I^1_h$ denotes the first order Lagrange interpolation defined
in \eqref{eq:LagInterp} for $k=1$. This type of discretizations have
been studied before e.g. in \cite{CaFa95,DTEnJa17a}.

\begin{remark}
\begin{enumerate}[(a)]
\item If $\eta=h$ and $\sigma_i=e_i$, the standard
  basis in $\R^N$, then
  $$L^\sigma=\Delta\qquad \text{and}\qquad L^\sigma_{h}=L^\sigma_{h,h}= \Delta_h,$$
where $\Delta_h$ is the classical second order finite difference
approximation
\begin{equation}\label{eq:discLap}
 \Delta_h\psi(x):=\sum_{i=1}^N\frac{\psi(x+he_i)-2\psi(x)+\psi(x-he_i)}{h^2}.
  \end{equation}

\smallskip 
\item By a coordinate transformation $x=Ay$ (diagonalization),
  $L^\sigma$ can always be transformed into
\begin{equation*}
L^{I_0}\qquad\text{where}\qquad I_0:=\left[\begin{array}{c|c}
I & 0\\
\hline
0 & 0\\
\end{array}\right]\in \R^{N\times N},
\end{equation*}
where $I$ is an identity matrix. $L^{I_0}$ corresponds to a Laplacian
operator in $\R^M$ for some $M\leq N$. In the transformed coordinates
$y$, $L^{I_0}=\Delta_{\R^M}$, the $\R^M$-Laplacian, for some $M\leq N$, and again $L^{I_0}_h=L^{I_0}_{h,h}=\Delta_{\R^M,h}$ (see \cite{DTEnJa17a}).
\end{enumerate}
\end{remark}

We have the following general result.
\begin{lemma}\label{discretelocal}
Let $h,\eta>0$, $h=o(\eta)$, and $L^\sigma$ be defined by \eqref{deflocaldiffusion}. The family of operators $\{L^\sigma_{\eta,h}\}_{\eta,h>0}$ given by \eqref{AproxLoc2}  is an \textbf{admissible} approximation of $L^\sigma$  with \textbf{Local Truncation Error} $O(\frac{h^2}{\eta^2}+\eta^2 )$ (or $O(h)$ with the optimal $\eta=\sqrt{h}$).
\end{lemma} 
Lemma \ref{discretelocal} is close to results e.g. in
\cite{DeJa13a}, but below we give a proof for completeness. 
For the discretization $\Delta_h$ we have the following result.
\begin{lemma}
Let $h>0$.  The family of operators $\{\Delta_h\}_{h>0}$ given by \eqref{eq:discLap}  is an \textbf{admissible} approximation of $\Delta$  with \textbf{Local Truncation Error} $O(h^2)$.
\end{lemma}
Admissibility and the improved (and classical!) rate follows as in the
proof of Lemma \ref{discretelocal} since there is no interpolation
now. 

\begin{proof}[Proof of Lemma \ref{discretelocal}] 
  We show that $L^\sigma_{\eta,h}$ can be written in the finite
  difference form \eqref{FD}. For simplicity, let us set
  $\sigma_{-i}=-\sigma_i$ and let $x_\alpha\in \Grid$. Then
\begin{equation*}
\begin{split}
L^\sigma_{\eta,h}[\psi](x_\alpha)&=\sum_{i=1}^P \frac{I_h^1[\psi](x_\alpha+\sigma_i \eta) + I_h^1[\psi](x_\alpha-\sigma_i \eta)-2\psi(x_\alpha)}{\eta^2}\\
&=\sum_{0<|i|\leq P}\left(\sum_{\beta}\psi(z_\beta)p_0^1(x_\alpha+\sigma_i \eta-z_\beta)- \psi(x_\alpha) \right) \frac{1}{\eta^2}\\
&=\sum_{0<|i|\leq P}\left(\sum_{\gamma}\psi(x_\alpha+z_\gamma)p_{0}^1(\sigma_i \eta-z_\gamma)- \psi(x_\alpha) \right) \frac{1}{\eta^2}\\
&= \sum_{\gamma\neq0}\left(\psi(x_\alpha+ z_\gamma)-\psi(x_\alpha)\right) \left(\frac{1}{\eta^2}\sum_{0<|i|\leq P} p_0^1(\sigma_i \eta-z_\gamma)\right).
\end{split}
\end{equation*}
The weights are $\omega_{\eta,h,\beta}=\frac{1}{\eta^2}\sum_{0<|i|\leq P}
p_0^1(\sigma_i \eta-z_\beta)$, and we immediately find that  $0\leq
\omega_{\eta,h,\beta}=\omega_{\eta,h,-\beta}$ since $0\leq p^1_0$ is
even, $\sigma_{-i}=-\sigma_i$, and $z_{-\beta}=-z_{\beta}$. Moreover, 
\begin{equation*}
\begin{split}
\sum_{\beta\not=0} \omega_{\eta,h,\beta}=\frac{1}{\eta^2} \sum_{0<|i|\leq P} \sum_{\beta\not=0} p_\beta^1(\sigma_i \eta)=\frac{1}{\eta^2} \sum_{0<|i|\leq P} (1-p_0^1(\sigma_i \eta))\leq \frac{2P}{\eta^2}<+\infty.
\end{split}
\end{equation*}
To show consistency we split the error in the following way,
\[
\|L^\sigma\psi-L^\sigma_{\eta,h}[\psi]\|_{L^1(\R^N)}\leq\|L^\sigma\psi-L^\sigma_{\eta}[\psi]\|_{L^1(\R^N)}+\|L^\sigma_{\eta}[\psi]-L^\sigma_{\eta,h}[\psi]\|_{L^1(\R^N)}
\]
where $L^\sigma_{\eta}[\psi]$ is given by \eqref{AproxLoc}. The first
term on the right hand side is the classical error of a second order
approximation of second derivatives,
\begin{equation}\label{eq:newineq}
\|L^\sigma\psi-L^\sigma_{\eta}[\psi]\|_{L^1(\R^N)}\leq C\|D^4\psi\|_{L^1(\R^N)}\eta^2|\sigma|^4 .
\end{equation}
For the second term, we have (cf. \eqref{eq:LagInterpError})
\begin{equation*}
\begin{split}
\|L^\sigma_{\eta}[\psi]-L^\sigma_{\eta,h}[\psi]\|_{L^1(\R^N)}&=\frac{1}{\eta^2}\left\| \sum_{0<|i|\leq P} \left(I_h^1[\psi](\cdot+\sigma_i \eta)- \psi(\cdot+\sigma_i\eta) \right)\right\|_{L^1(\R^N)}\\
&\leq\frac{1}{\eta^2} \sum_{0<|i|\leq P} \left\|I_h^1[\psi](\cdot+\sigma_i \eta)- \psi(\cdot+\sigma_i\eta)\right\|_{L^1(\R^N)}\\
&=\frac{1}{\eta^2}\sum_{0<|i|\leq P}  C \|D^{2}\psi\|_{L^1(\R^N)}h^{2}= O\left(\frac{h^2}{\eta^2}\right).
\end{split}
\end{equation*}
Thus, for any choice of $h=o(\eta)$ we get a consistent
scheme. Moreover, one can observe that the last two estimates also hold in
$L^\infty$ by a trivial adaptation. By Proposition \ref{prop:equivcons}
it then follows that the uniform integrability condition
\eqref{eq:unifLevy} holds. In view of Lemma \ref{lemma:suitdisc}, the scheme
is admissible (cf. Definition \ref{def:suitdisc}) and the proof is complete.
\end{proof}

%%%%%%%%%%%%%%%%%%%%%%%%%%%%%%%%%%%%%%%%%%%%%%%%%%%%%%%%%%

\subsection{Discretizations of the singular nonlocal operator $\Levy_{r}^\mu$}

We present discretizations of the
singular/unbounded part of the nonlocal operator (recall Remark \ref{deflevyexplained}) 
\begin{equation}\label{eq:singpart}
\Levy_{r}^\mu[\psi](x)=\int_{0<|z|\leq r}\Big(\psi(x+z)-\psi(x)-z\cdot
D\psi(x) \Big)\dd \mu(z), \quad r\in[0,1].
\end{equation}
We start with the {\em trivial discretization} where we discretize
$\Levy_{r}^\mu$ by   
\begin{equation}\label{eq:trivialDiscSing}
\Levy^h[\psi](x)	\equiv 0. 
\end{equation}
This crude discretization is  computationally efficient, and
depending on the order of the operator and the other discretizations
involved, the error could be satisfactory. 

\begin{lemma} Assume \eqref{muas}, $h\leq r =o_h(1)$, and
  $\Levy_{r}^\mu$ is given by \eqref{eq:singpart}. Then 
  $\{\Levy^h\}_{h>0}$ given by \eqref{eq:trivialDiscSing} is an
  \textbf{admissible} approximation of $\Levy_{r}^\mu$. Moreover, if
  also \eqref{alp} holds, then the \textbf{Local Truncation Error} is
  $O(r^{2-\alpha})$. 
\end{lemma}
\begin{proof}
Since $\Levy^h=\Levy^{\nu_h}$ with $\nu_h\equiv0$, it is a
discretization of the form \eqref{FD} in the class
\eqref{nuas}. It is consistent by the dominated
convergence theorem,
\begin{equation*}
\begin{split}
\|\Levy_{r}^\mu[\psi]- \Levy^h[\psi]\|_{L^1(\R^N)}&=\|\Levy_{r}^\mu[\psi]\|_{L^1(\R^N)}\\
&\leq \frac{1}{2}\|D^2 \psi\|_{L^1(\R^N)}\int_{0<|z|\leq r}|z|^2\dd \mu(z)\to 0 \quad \text{ as  } \quad h\to 0^+.
\end{split}
\end{equation*}
If also \eqref{alp} holds, then the LTE is  $O(r^{2-\alpha})$. 
Moreover \eqref{eq:unifLevy} also holds since
$\sup_{h<1}\sum_{\beta\not=0} (|z_\beta|^2\wedge 1
)\omega_{\beta,h}=0<+ \infty$. 
\end{proof}

We now show how to get a more accurate discretization of $\Levy^\mu_r$
through the {\em adapted vanishing viscosity discretization}
\cite{AsRo01,CoTa04,JaKaLC08}: 
Approximate $\Levy_r^\mu$ by a local second order operator and then
discretize. 
To do this, note that by Taylor's theorem,
\begin{equation*}
\begin{split}
\psi(x+z)-\psi(x)-z\cdot D\psi(x)=\sum_{2\leq|\beta|\leq3}\frac{1}{\beta!}D^\beta\psi(x)z^\beta+ \sum_{|\beta|=4}R_\beta(x,z) z^\beta.
\end{split}
\end{equation*}
where $R_\beta(x,z)= \frac{|\beta|}{\beta!}
\int_0^1(1-s)^{|\beta-1|}D^\beta\psi(x+sz)\dd s$.
Since terms with
$|\beta|=3$ are odd and the measure $\mu$ is symmetric, 
\begin{equation*}
\begin{split}
  \Levy_{r}^\mu[\psi](x)
&=T_2^r(x)+0+R_4^r(x),
\end{split}
\end{equation*}
where $R_4^r(x)=\sum_{|\beta|=4}\int_{|z|<r} R_\beta(x,z)z^\beta\dd\mu(z)$,
\begin{gather*}
T_2^r(x)=\sum_{|\beta|=2}\frac{D^\beta\psi(x)}{\beta!}\int_{|z|<r}  z^\beta\dd
\mu(z)
=\frac{1}{2}\textup{tr}\big[ZD^2\psi(x)\big],
\end{gather*}
and $Z_{ij}=\int_{|z|<r}z_iz_j\dd \mu(z)$. Observe that $Z=Z(r)$  is a 
symmetric and positive semidefinite matrix, \(\xi^T Z \xi=\int_{|z|<r} \xi^T (zz^T) \xi\dd \mu(z)=\int_{|z|<r} (\xi\cdot z)^2 \dd \mu(z)\geq0
\)
for $\xi \in \R^N$, and that $Z(r) \to 0$ as $r \to 0$ by the dominated convergence theorem.  Hence it has a square root $\sqrt{Z}$ with entries
$(\sqrt{Z})_{ij}$ and columns $(\sqrt{Z})_{i}$ (our notation).
This allows us to write $T_2^r$ as a sum of directional derivatives
\begin{equation*}
  T_2^r(x)=\frac{1}{2}\sum_{i=1}^N ( (\sqrt{Z})_{i}\cdot D )^2\psi(x).
\end{equation*}
This is an operator of the form $L^\sigma$ with
$\sigma_i=(\sqrt{Z})_{i}$, $i=1,\ldots,N$ (cf. \eqref{localalt}),
and we discretize it as in Section \ref{sec:discLocal}:
\begin{equation}\label{eq:disSingHO}
\Levy^{\eta,h}[\psi](x)=\sum_{i=1}^P \frac{I_h^1[\psi](x+(\sqrt{Z})_{i}\eta) + I_h^1[\psi](x-(\sqrt{Z})_{i}\eta)-2\psi(x)}{\eta^2}
\end{equation}
Note that \eqref{eq:disSingHO} depends on $r$ through the terms
$(\sqrt{Z})_{i}$, and that (linear)
interpolation is needed since $x+(\sqrt{Z})_{i}\eta$ do not belong
to the grid $\Grid$ in general. 

\begin{lemma}
  \label{lem:gooddiscorig}
Assume \eqref{muas}, $h,\eta,r>0$, $h=o(\eta)$, $h\leq
r =o_h(1)$, and $\Levy^\mu_r$ is defined by
\eqref{eq:singpart}. Then the family 
$\{\Levy^{\eta,h}\}_{h,\eta>0}$ given by \eqref{eq:disSingHO}  is an
\textbf{admissible} approximation of $\Levy^\mu_r$  with \textbf{Local
  Truncation Error} $O(\frac{h^2}{\eta^2}+\eta^2|Z(r)|^2 +r^2)$.  Moreover, if
  also \eqref{alp} holds, then the LTE is
  $O(\frac{h^2}{\eta^2}+\eta^2r^{4-2\alpha}+r^{4-\alpha})$. 
\end{lemma} 
In general,  the optimal choice will turn
out to be 
$\eta=r=h^{\frac12}$ which gives a linear rate 
$O(h)$. If  also \eqref{alp} holds, the optimal choice
will turn out to be $r=h^{\frac{1}{2}}$ and $\eta=h^{\frac{\alpha}{4}}$ which
 gives the superlinear rate $O(h^{2-\frac\alpha2})$. The 
proof will be given below. Now we show how this rate can be improved
when the measure $\mu$ satisfies stronger symmetry assumptions.
\begin{definition}
  A Borel measure $\mu$ is \textbf{symmetric in coordinate
    directions} if $\mu(A)=\mu(T_{x_i}(
  A) )$ for all $i=1,\ldots, N$ and Borel sets $A$, where
  $T_{x_i}(A)=\{(x_1,\dots,-x_i,\dots,x_N):
  (x_1,\dots,x_i,\dots,x_N)\in A\}$.
  \end{definition}

\begin{remark}
Obviously radial  symmetry $\Rightarrow$ symmetry in coordinate
directions $\Rightarrow$ symmetry ($\mu(A)=\mu(-A)$). All reverse
implications are false. If e.g. $\dd\mu(z_1,z_2)=\sgn (z_1z_2)^+\dd z_1 \dd
z_2$, then $\mu$ is symmetric but not symmetric in coordinate
directions.
\end{remark}

If $\mu$ is symmetric in coordinate directions, then  mixed 
  derivatives vanish and
  $T_2^r=\frac12\sum_{i=1}^N\frac{\partial^2\psi}{\partial 
    x_i^2}\int_{|z|<r}z_i^2\dd\mu(z)$. A natural difference 
  approximation of $\Levy^\mu_r$ is then
\begin{equation}\label{eq:dissing2}
\Levy^h[\psi](x):=\frac{1}{2}\sum_{i=1}^N\frac{\psi(x+e_ih)+\psi(x-e_ih)-2\psi(x)}{h^2}\int_{|z|<r}z_i^2\dd \mu(z).
\end{equation}
If $\mu$ is also radially symmetric we even have $\int_{|z|<r}
z_i^2\dd \mu(z)=\frac{1}{N}\int_{|z|<r} |z|^2\dd \mu(z)$.
\begin{lemma}
Assume \eqref{muas}, $\mu$ is symmetric in coordinate directions,
$h,r>0$, $h\leq r=o_h(1)$, and $\Levy_{r}^\mu$ is
defined by  \eqref{eq:singpart}. Then the family $\{\Levy^{h}\}_{h>0}$
given by \eqref{eq:dissing2} is an \textbf{admissible} approximation
of $\Levy^\mu_r$  with \textbf{Local Truncation Error} $O(h^2+r^2)$. Moreover, if
  also \eqref{alp} holds, then the LTE is
  $O(h^2r^{2-\alpha}+r^{4-\alpha})$.
\end{lemma}
The optimal choice  will turn out to be 
$r=h$  which  gives a
quadratic rate $O(h^2)$ or 
superquadratic rate $O(h^{4-\alpha})$ if also \eqref{alp} holds.
The proof is similar to the proof of  Lemma \ref{lem:gooddiscorig}.
\begin{proof}[Proof of Lemma \ref{lem:gooddiscorig}]
Since $\Levy^{\eta,h}$ is in the form of \eqref{AproxLoc2}, it is
already in finite difference form \eqref{FD}. By Lemma
\ref{discretelocal} it is in the class \eqref{nuas}. We now check consistency.
\[
\|\Levy_{r}^\mu[\psi]-\Levy^{\eta,h}[\psi]\|_{L^1(\R^N)}\leq \|T_2^r-\Levy^{\eta,h}[\psi]\|_{L^1(\R^N)}+ \|R_4^r\|_{L^1(\R^N)},
\]
and by Lemma \ref{discretelocal} and inequality  \eqref{eq:newineq}  with $\sigma_i=(\sqrt Z)_i$,  the definition of
$R_4^r$, and Fubini,
\begin{equation*}
\begin{split}
\|T_2^r-\Levy^{\eta,h}[\psi]\|_{L^1(\R^N)}&\leq
  C\|D^2\psi\|_{L^1(\R^N)}\frac{h^2}{\eta^2}+C\|D^4\psi\|_{L^1(\R^N)}\eta^2|Z|^2 \\
  &=O\Big(\frac{h^2}{\eta^2}+\eta^2|Z|^2 \Big), \\
\|R_4^r\|_{L^1(\R^N)}&\leq C  \|D^4 \psi\|_{L^1(\R^N)}\int_{|z|<r} |z|^4 \dd \mu(z) = O(r^2),
\end{split}
\end{equation*}
leading to the desired Local Truncation Error. Since $|Z_i|\leq C\int_{|z|<r}|z|^2\dd \mu(z)$, the modifications when also
\eqref{alp} holds are obvious. Once again, the estimates also hold 
in $L^\infty$ with a trivial adaptation and thus Proposition
\ref{prop:equivcons} ensures that \eqref{eq:unifLevy} is satisfied. 
\end{proof}

%%%%%%%%%%%%%%%%%%%%%%%%%%%%%%%%%%%%%%%%%%%%%%%%%%%%%%

\subsection{Discretizations of the bounded nonlocal operator
  $\Levy^{\mu,r}$} We present discretization of the
bounded/nonsingular part of the nonlocal operator, i.e. of
\begin{equation}\label{nonsingop}
\Levy^{\mu,r}[\psi](x):=\int_{|z|>r}\big(\psi(x+z)-\psi(x)\big)\dd
\mu(z),\qquad r\in[0,1].
\end{equation}
Note that $\Levy^{\mu,r}$ is an operator in the form \eqref{muas} with measure $\dd\mu_r(z)=  \indik_{|z|>r}\dd\mu(z)$.

%%%%%%%%%%%%%%%%%%%%%%%%%%%%%%%%%%%%%%%%%%%%%%%%%%%%%%%%%%

\subsubsection{Midpoint quadrature rule}

The idea is to approximate $\psi$ on each cube $z_\beta+R_h$ by its
midpoint value $\psi (z_\beta)$. 
This gives the following quadrature formula for the cube,
 $$\int_{z_\beta+R_h}\big(\psi(x+z)-\psi(x)\big)\dd \mu(z)\simeq
(\psi(x+z_\beta)-\psi(x))\mu(z_\beta+R_h),$$
and a discretization of $\Levy^{\mu,r}$ given by
\begin{equation}\label{sillydisc}
\Levy^h[\psi](x)=\sum_{\beta\not=0} \left(\psi(x+z_\beta)-\psi(x)\right) \mu \left(z_\beta+R_h\right).
\end{equation}

\begin{lemma}\label{discretenonlocal}
Assume \eqref{muas}, $h\leq r=o_h(1)$, and $\Levy^{\mu,r}$ is defined
by \eqref{nonsingop}. The family $\{\Levy^h\}_{h>0}$
given by \eqref{sillydisc} is an \textbf{admissible} approximation of
$\Levy^{\mu,r}$. Moreover, if
  also \eqref{alp} holds, then the \textbf{Local Truncation Error} is
  $O(h+r^{2-\alpha})$. 
\end{lemma} 
The result was proved in Lemma 5.3 and 5.4 in \cite{DTEnJa17a}. Note
that since $\mu$ is not translation invariant as the
Lebesgue measure, the midpoint rule is no longer a second order
method in general. Moreover, the LTE is dominated by the contribution
from the integral over $|z|>1$ and can therefore not be improved by
assuming \eqref{alp}. 

\begin{remark}
Lemma \ref{discretenonlocal} is consistent with the  numerical experiments in Section
\ref{sec:numsim} for $\alpha>1$, while for $\alpha<1$ the
numerical results are better. Note that if $N=1$, 
$\alpha=1$ and $r=h$, then the midpoint rule coincides with the second
order discretization of Section \ref{sec:pdl} below. This time the
observed LTE is $O(h^2)$. 
\end{remark}

\subsubsection{Quadrature from interpolation -- general
  $\mu$}\label{sec:exactinterp} 
The idea is to replace $\psi(x+z)-\psi(x)$ by a Lagrange polynomial
interpolant on $\Grid$ (defined in Section \ref{subsec:interp}),
\[
I_h^k[\psi(x+\cdot)-\psi(x)](z)=\sum_{\beta\not=0} \left(\psi(x+z_\beta)-\psi(x)\right)p_\beta^k(z),
\]
and integrate with respect to the measure $\mu$ to obtain the
 discretization 
\begin{equation}\label{eq:discInterp}
\begin{split}
\Levy^h[\psi](x)&=\int_{|z|>r}I_h^k[\psi(x+\cdot)-\psi(x)](z) \dd \mu(z)
\\ &= \sum_{\beta\not=0} \left(\psi(x+z_\beta)-\psi(x)\right)\int_{|z|>r} p_\beta^k(z)\dd \mu(z).
\end{split}
\end{equation}
This is a classical idea that has been used for nonlocal operators
before, e.g. \cite{BiJaKa10,HuOb14}.

To have a more understandable presentation, we divide the proof of
admissibility into several lemmas. We begin by noting that the  
operator \eqref{eq:discInterp} obviously is in the finite difference
form  \eqref{FD} with symmetric weights  
\[
\omega_{\beta,h}=\int_{|z|>r} p_\beta^k(z)\dd \mu(z), \quad \beta \in \Z^N.
\] 
However, the nonnegativity of the weights is only guaranteed when
$k=0$ and $k=1$ since these are  the only cases where the basis functions
$p_\beta^k$ are all nonnegative.

\begin{lemma}\label{lem:interpIsDisc}
Assume \eqref{muas}, $h,r>0$, and $k=0$ or $k=1$. Then the family
$\{\Levy^h\}_{h>0}$ given by \eqref{eq:discInterp} is in the class \eqref{nuas}. 
\end{lemma}
\begin{proof}
  In view of the above discussion and Lemma \ref{lemma:suitdisc}, it
  remains to check that
\[
\sum_\beta \omega_{\beta,h}= \sum_\beta \int_{|z|>r} p_\beta^k(z)\dd \mu(z)= \int_{|z|>r}  \sum_\beta p_\beta^k(z)\dd \mu(z)=\int_{|z|>r}  \dd \mu(z)<+\infty.
\]
The proof is complete.
\end{proof}

\begin{remark}
  \label{rem:pos}
Note that $p_\beta^k\geq0$ implies $\omega_{\beta,h}\geq0$, but the
other implication depends on $\mu$ and is not true in
general. But if $\mu$ were the Lebesgue measure supported on a cube,
then our quadrature would coincide with the Newton-Cotes quadratures
which are known to have nonnegative weights for orders
$k\leq6$. 

Moreover, if $\dd\mu(z)=\frac{\dd z}{|z|^{N+\alpha}}$
(fractional Laplace), then explicit  nonnegative weights are found and presented in a nice way in 
\cite{HuOb14} for $N=1$, $k=1$ and $k=2$. In this case quadratic interpolation combined with \eqref{eq:dissing2} yields an admissible
discretization of $-(-\Delta)^{\frac{\alpha}{2}}$ with a LTE of $O(h^{3-\alpha})$. Numerical
evidence for this rate is given in Section \ref{sec:numsim}.
\end{remark}

The following result on local truncation error is valid for any $k\in \N$.

\begin{lemma}\label{lemma:lagInterpCons}
Assume \eqref{muas}, $h\leq r=o_h(1)$, $k\geq0$, and $\Levy^{\mu,r}$
is defined by \eqref{nonsingop}. Then the family $\{\Levy^h\}_{h>0}$
given by \eqref{eq:discInterp} has {\bf Local
  Truncation Error}
%C_c^\infty(\R^N)$, 
\begin{equation}\label{eq:estimInterpMu}
\|\Levy^{\mu,r}[\psi]- \Levy^h[\psi]\|_{L^p}\leq
Ch^{k+1}\|D^{k+1}\psi\|_{L^p} \mu\left(\{|z|>r\right\}), \quad p\in\{1,\infty\}.
\end{equation}
Moreover, if
  also \eqref{alp} holds, then the LTE is
  $O(h^{k+1}r^{-\alpha})$.
\end{lemma}

\begin{proof}
We use the notation $P_\psi(x,z)=I_h^k[\psi(x+\cdot)-\psi(x)](z)$ and use
\eqref{eq:LagInterpError} to get
\begin{equation*}
\begin{split}
  \|\Levy^{\mu,r}[\psi]- \Levy^h[\psi]\|_{L^p(\R^N)}
&\leq \int_{|z|>r}\| P_\psi(\cdot,z) -(\psi(\cdot+z)-\psi(\cdot))\|_{L^p(\R^N)} \dd \mu (z)\\
&\leq C \|D^{k+1}\psi\|_{L^p(\R^N)} h^{k+1} \int_{|z|>r}\dd \mu (z).
\end{split}
\end{equation*}
The proof is complete.
\end{proof}

To be consistent we need to impose that
$h^{k+1}\mu\left(\{|z|>r\right\})\to 0$ as $h\to0$. When also
\eqref{alp} holds, this is always satisfied for $k\geq1$ and $r=o_h(1)$,
while for $k=0$ we need $hr^{-\alpha}=o_h(1)$. 
We are now in a position to state and prove admissibility for the cases $k=0$
and $k=1$.
\begin{lemma}
Assume \eqref{muas}, $h\leq r=o_h(1)$, and $\Levy^{\mu,r}$ is defined by
\eqref{nonsingop}. If either $k=0$ and
$h=o_r(\frac1{\mu(\{|z|>r\})})$ or $k=1$, then the family
$\{\Levy^h\}_{h>0}$ given by \eqref{sillydisc} is an
\textbf{admissible} approximation of $\Levy^{\mu,r}$.  
\end{lemma}
\begin{proof}
By the above discussion and Lemma \ref{lem:interpIsDisc},
$\{\Levy^h\}_{h>0}$ is a \eqref{FD}  type discretization in the class
\eqref{nuas}. For consistency we use the error estimate
\eqref{eq:estimInterpMu}. When $k=0$ we can conclude from the extra
assumption $h=o_r(\frac1{\mu(\{|z|>r\})})$. When $k=1$, we observe that  
\[
h^{k+1} \int_{|z|>r} \dd\mu(z)=  \int_{|z|>r} h^2\dd\mu(z)\leq
\int_{|z|>0} |z|^2\wedge r^2\dd\mu(z)
\]
and conclude by the Dominated Convergence Theorem
since $|z|^2\wedge r^2\to0$ pointwise as $r\to0^+$ and $|z|^2\wedge
r^2\leq |z|^2\wedge1$ ($r\leq1$) which is integrable by \eqref{muas}.  
Uniform integrability \eqref{eq:unifLevy} follows from Proposition
\ref{prop:equivcons} and the $L^\infty$ version of the above estimate.
\end{proof}
For admissibility and higher order interpolation, see Remark \ref{rem:pos}.
\subsubsection{Quadrature from interpolation -- absolute
  continuous $\mu$.}
If $\mu$ is
absolutely continuous with respect to the Lebesgue measure $\dd z$
with density also called $\mu(z)$, then (cf. \cite{JaKaLC08}) we can approximate $\Levy^{\mu,r}$ by 
\begin{equation}\label{eq:interp2}
\begin{split}
\Levy^h[\psi](x)&:=\int_{|z|>r}I_h^{k}\Big[(\psi(x+\cdot)-\psi(x))\mu(\cdot)\Big](z)\dd z\\
&=\sum_{|z_\beta|>r} \big(\psi(x+z_\beta)-\psi(x)\big)\mu(z_\beta)\int_{|z|>r} p^k_\beta(z)\dd z.
\end{split}
\end{equation} 
Note that \eqref{eq:interp2} is in the finite difference form 
\eqref{FD} with symmetric weights
\[
\omega_{\beta,h}=\mu(z_\beta) \int_{|z|>r} p_\beta^k(z)\dd z
\quad\text{for}\quad |z_\beta|>r \qquad \text{and}\qquad \omega_{\beta,h}=0 \quad\text{otherwise.} 
\] 
The weights are nonnegative for $k\leq 6$ since
$\mu\geq0$ and $\tilde\omega_{\beta,h}=\int_{\R^N} p_\beta^k(z)\dd z$ coincides with
Newton-Cotes quadrature weights that are known to be nonnegative.
As we will describe later, this discretization can also be combined with
\eqref{eq:dissing2} to further improve the order of accuracy.

\begin{lemma}\label{mumuas}
Assume \eqref{muas}, $0\leq k\leq 6$ , $h, r>0$, and $\dd \mu(z)=\mu(z)\dd z$
for a density $\mu\in C_b(\R^N\backslash B_r)$. Then the family
$\{\Levy^h\}_{h>0}$ given by \eqref{eq:interp2} is in the class \eqref{nuas}.
\end{lemma}
\begin{proof}
By the previous discussion we know that $\Levy^h$ is of the form 
\eqref{FD} with symmetric nonnegative weights, so by Lemma
\ref{lemma:suitdisc}, we just need to check: 
\begin{equation*}
\begin{split}
\sum_{|z_\beta|>r} \mu(z_\beta) \int_{|z|>r} p_\beta^k(z)\dd \mu(z)&\leq \|\mu\|_{L^\infty(\R^N\backslash B_r)} \int_{|z|>r}  \sum_\beta p_\beta^k(z)\dd \mu(z)\\
&=  \|\mu\|_{L^\infty(\R^N\backslash B_r)} \int_{|z|>r}  \dd \mu(z)<+\infty.
\end{split}
\end{equation*}
The proof is complete.\end{proof}

Since also $\mu$ is interpolated now, the local truncation error will 
depend on the regularity of $\mu$. We state the result using
standard Sobolov spaces $W^{k,1}$ and $W^{k,\infty}$.

\begin{lemma}\label{lemma:lag2}
Assume \eqref{muas}, $0\leq k\leq 6$, $h,r>0$ , and that $\dd \mu( z)=\mu(z)\dd
z$ for a density $\mu\in C(\R^N\backslash B_r)\cap
W^{k+1,1}(\R^N \backslash B_r)$. If $\Levy^{\mu,r}$ is defined by
\eqref{nonsingop} and $\{\Levy^h\}_{h>0}$ by \eqref{eq:interp2}, then
there is a constant $C>0$ such that for $p\in\{1,\infty\}$,
\[
\|\Levy^{\mu,r}[\psi]- \Levy^h[\psi]\|_{L^p(\R^N)}\leq C h^{k+1} \|\psi \|_{W^{k+1,p}(\R^N)}\|\mu \|_{W^{k+1,1}(\R^N \backslash B_r)}.
\]
\end{lemma}
\begin{proof}
Let $Q_\alpha:=z_\alpha + \frac{kh}2(-1,1]^N$ for $\alpha \in \Z^N$,
$y_i^\alpha$ denote the $(k+1)^N$ gridpoints in $Q_\alpha\cap \Grid$,
and $F(x,z)=(\psi(x+z)-\psi(x))\mu(z)$.
Note that then $I^k_h\psi(y_i^\alpha)=\psi(y_i^\alpha)$, 
\[
\displaystyle\bigcup_{\alpha \in k \Z^N} Q_\alpha=
\R^N,\qquad\text{and}\qquad\Levy^h[\psi](x):=\int_{|z|>r}I_h^{k}[F(x,\cdot) ](z)\dd z. 
\]
Using Taylor expansions with integral remainder terms at every point
$y_i^\alpha$, we get 
\begin{equation*}
\begin{split}
&|\Levy^h[\psi](x)-\Levy^{\mu,r}[\psi](x)|\leq \sum_{\alpha \in k \Z^N, |z_\alpha|>r}\int_{Q_\alpha}|I_h^{k}[F(x,\cdot) ](z)-F(x,z)|\dd z\\
&\leq h^{k+1}\sum_{i=1}^{(k+1)^N}  \sum_{\alpha \in k \Z^N, |z_\alpha|>r} \int_{Q_\alpha} \int_0^1 |D^{k+1}_z F(x,z(1-s)+sy_i^\alpha)|\dd s \dd z.
\end{split}
\end{equation*}
By Fubini, the definition of $F$, and the chain rule, for
$p=\{1,\infty\}$ it follows that
\begin{equation*}
\begin{split}
&\|\Levy^h[\psi]-\Levy^{\mu,r}[\psi]\|_{L^p(\R^N)}\\
&  \leq  h^{k+1}\sum_{i=1}^{(k+1)^N}  \sum_{\alpha \in k \Z^N, |z_\alpha|>r} \int_{Q_\alpha} \int_0^1\|D^{k+1}_z F(\cdot,z(1-s)+sy_i^\alpha)\|_{L^p(\R^N)}\dd s \dd z\\
& \leq Ch^{k+1} \| \psi \|_{W^{k+1,p}(\R^N)}   \sum_{i=1}^{(k+1)^N}\sum_{\alpha \in k \Z^N, |z_\alpha|>r} \sum_{l=0}^{k+1}\int_0^1\int_{Q_\alpha}  |D^{l} \mu(z(1-s)+sy_i^\alpha)|\dd z \dd s
\end{split}
\end{equation*}
Now we do the change of variables $y=z(1-s)+sy_i^\alpha$, which has
the change in measure $\dd y=(1-s)^N \dd z$ and maps $Q_\alpha$ into
$\tilde{Q}_\alpha=sy_i^\alpha+(1-s)Q_\alpha$. Then since $s\in[0,1]$ and
$Q_\alpha$ is convex, $\tilde{Q}_\alpha\subset Q_\alpha$ and 
\begin{align*}
\int_{Q_\alpha}  |D^{l} \mu(z(1-s)+sy_i^\alpha)|\dd z
=\int_{\tilde{Q}_\alpha}  (1-s)^N|D^{l} \mu(y)|\dd y \leq \int_{Q_\alpha}  |D^{l} \mu(y)|\dd y.
\end{align*}
For $0\leq l\leq k+1$,
$$\sum_{\alpha \in k \Z^N, |z_\alpha|>r}\int_{Q_\alpha}  |D^{l}
\mu(y)|\dd y\leq \|\mu\|_{W^{k+1,1}(\R^N \backslash B_r)},$$
so we conclude that for $p\in\{1,\infty\}$,
\begin{equation*}
  \|\Levy^h[\psi]-\Levy^{\mu,r}[\psi]\|_{L^p(\R^N)}
 \leq Ch^{k+1} \|\psi \|_{W^{k+1,p}(\R^N)} \|\mu\|_{W^{k+1,1}(\R^N \backslash B_r)}.
\end{equation*}
The proof is complete.
\end{proof}

From Lemmas \ref{mumuas}, \ref{lemma:lag2}, and \ref{lemma:suitdisc}, we have the following result on admissibility.

\begin{corollary}
Assume \eqref{muas}, $0\leq k\leq 6$, $0<r=o_h(1)$, $\dd \mu( z)=\mu(z)\dd
z$ for a density $\mu\in C(\R^N\backslash B_r)\cap
W^{k+1,1}(\R^N \backslash B_r)$, and $\Levy^{\mu,r}$ is defined by
\eqref{nonsingop}. If $h^{k+1}\|\mu\|_{W^{k+1,1}(\R^N \backslash
  B_r)}\to 0$ as $h\to0$,
then the family
$\{\Levy^h\}_{h>0}$ given by \eqref{eq:interp2} is an
\textbf{admissible} approximation of $\Levy^{\mu,r}$.
  \end{corollary}

We expect to have more precise results when $\mu$ comes from a
fractional differential operator, i.e. a precise estimate on how
$\|\mu\|_{W^{k+1,1}(\R^N \backslash B_r)}$ depends on $r$. However
stronger assumptions than \eqref{alp} are needed here. It is easy to
find such conditions in general, but for simplicity we only focus on
the fractional Laplacian case. 

\begin{corollary}[Fractional Laplace]\label{lemma:lagfl}
Let $\alpha\in(0,2)$, $0\leq h\leq r=o_h(1)$, $0 \leq k\leq6$,
$\Levy^{\mu,r}$ be defined by \eqref{nonsingop} with density
$\mu(z)=\frac{1}{|z|^{N+\alpha}}$. If $h=o(r^{\frac{\alpha+k+1}{k+1}})$,
then the family $\{\Levy^h\}_{h>0}$ 
given by \eqref{eq:interp2} is an \textbf{admissible} approximation of
$\Levy^{\mu,r}$ . Moreover, for $p=\{1,\infty\}$,
\[
\|\Levy^{\mu,r}[\psi]- \Levy^h[\psi]\|_{L^p(\R^N)}\leq C h^{k+1}{r^{-\alpha-k-1}}.
\]
\end{corollary}
\begin{proof}
Since $D^k\mu(z)=O(\frac1{|z|^{N+\alpha+k}})$ and
  $\int_{|z|>r}\frac1{|z|^{N+\alpha+k}}dz=c\int_r^\infty
\frac1{r^{1+\alpha+k}}\dd r=C\frac1{r^{\alpha+k}}$, we find that
$ \|\mu\|_{W^{k+1,1}(\R^N \backslash B_r)}=O(\frac1{r^{\alpha+k+1}})$.
\end{proof}

\begin{remark}[Fractional Laplace]
Combining the discretizations of Corollary \ref{lemma:lagfl}
(Newton-Cotes for the nonsingular part) and Lemma
\ref{lem:gooddiscorig} (adapted vanishing viscosity for the singular
part) we get a high order monotone discretization of the fractional
Laplacian with (combined) local truncation error,
\[
E=O(r^{4-\alpha}+h^{k+1} r^{-\alpha-k-1}).
\]
The optimal choice of $r$ is $r=h^{\frac{k+1}{k+5}}$, which leads to
a LTE of $O(h^{\frac{k+1}{k+5}(4-\alpha)})$. Note that for all
$\alpha\in(0,2)$, this rate is increasing in $k$, superlinear for
$k\geq 3$, and superquadratic in the limit $k\to\infty$. The best
choice giving an admissible (monotone) scheme is $k=6$ and
$r=O(h^{\frac{7}{11}})$ with a LTE of $O(h^{\frac{7}{11}(4-\alpha)})$.
\end{remark}

\begin{remark}[Random walk approximation]
In \cite{Val09, BuVa16} the Fractional Heat
Equation is formally derived from a random walk approximation with
arbitrarily long jumps,
$$
U^j(x_\beta)=U^{j-1}(x_\beta)+\Delta
t\sum_{\gamma\in\Z^N\setminus\{0\}}\big(U^{j-1}(x_\beta+z_\gamma)-U^{j-1}(x_\beta)\big)\omega_{\gamma,h},$$
where $\omega_{\gamma,h}:=\mu^\alpha(z_\gamma)h^N$,
$\mu^\alpha(z):=\frac{c_{N,\alpha}}{|z|^{N+\alpha}}$ for $z\neq0$,
$\mu^\alpha(0)=0$, and $c_{N,\alpha}$ is such that
$\sum_{\gamma\in\Z^N}\mu^\alpha(\gamma)=1$. Note that here $\Delta t=h^\alpha$ and
$\frac{\mu^\alpha(\gamma)}{h^\alpha}=h^N\mu^\alpha(z_\gamma)$. The
scheme corresponds to the spatial discretization  \eqref{eq:interp2}
with $k=0$, and converges by Corollary \ref{lemma:lagfl} and Theorem
\ref{thm:fullydisc} if we require 
$|z_\gamma|>r=ch^{\frac1{\alpha+1+\epsilon}}$ for some $c,\epsilon>0$.
\end{remark}

\subsection{A second order discretization of the fractional Laplacian}
\label{sec:pdl}
The fractional Laplacian $-(-\Delta)^{\frac\alpha2}$ can be discretized
by the corresponding power of the discrete Laplacian
\[
\Delta_h\psi(x)=\frac{1}{h^2}\sum_{i=1}^N\psi(x+e_ih)+\psi(x-e_ih)-2\psi(x)
\]
defined via subordination as
\begin{equation}\label{discfractlapN1}
(-\Delta_h)^{\frac{\alpha}{2}}[\psi](x):=\frac{1}{\Gamma(-\frac{\alpha}{2})}\int_0^\infty\left(\e^{t\Delta_h}\psi (x)-\psi(x)\right)\frac{\dd t}{t^{1+\frac{\alpha}{2}}},
\end{equation}
where $\Psi(x,t):=\e^{t\Delta_h}\psi (x)$ is the solution of the semi-discrete heat equation
\begin{equation}\label{eq:heateq1}
\left\{ \begin{array}{lll}
\displaystyle \dell_t\Psi=\Delta_h \Psi \quad&\qquad \text{for} \qquad&\quad (x,t)\in\R^N \times(0,\infty),\\[0.1cm]
\Psi(x,0)=\psi(x) \quad&\qquad \text{for} \qquad&\quad x\in\R^N.
\end{array}
\right.
\end{equation}
The solution of \eqref{eq:heateq1} has an explicit representation formula,
$$\Psi(x,t):=\e^{t\Delta_h}\psi
(x)=\sum_{\beta}\psi(x-z_\beta)G\Big(\beta,\frac{t}{h^2}\Big),
$$
where \(
G(\beta,t)=\e^{-2Nt}\prod_{i=1}^N I_{|\beta_i|}(2t)
\)
is the  fundamental solution of \eqref{eq:heateq1} and $I_m$ 
the modified Bessel function of first kind and order $m\in 
\N$.  Moreover $G\geq0$, $G(\beta,t)=G(-\beta,t)$,
and $\sum_\beta G(\beta,t)=1$, see e.g. \cite{CiRoStToVa15,LiRo18}.  

The original idea of this discretization is due to Ciaurri et al. 
\cite{CiRoStToVa15,CiRoStToVa16}. Here the authors obtain a
Local Truncation Errors of $O(h^{2-\alpha})$ for $N=1$. These results were then
extended to $N>1$ in \cite{CuDTG-GPa17a,CuDTG-GPa17b} for various boundary value
problems and the whole space case. Here the
authors also improve the local truncation error to $O(h^2)$
independently of $\alpha\in(0,2)$. This means that the approximation
$(-\Delta_h)^{\frac{\alpha}{2}}$ preserves the $O(h^2)$ error bound of
the discrete Laplacian $\Delta_h$. 

We now show that \eqref{discfractlapN1} is an \textbf{admissible}
operator. First we express \eqref{discfractlapN1} in the form
\eqref{FD}. This result has essentially been proved in \cite{CiRoStToVa15}.

\begin{lemma}
  \label{dfl_FD}
Let $h>0$, $\alpha\in(0,2)$, and $N\geq1$. Then $(-\Delta_h)^{\frac{\alpha}{2}}$ given by \eqref{discfractlapN1} is an operator of the form \eqref{FD},
\begin{equation}\label{eq:discflexp}
(-\Delta_h)^{\frac{\alpha}{2}}[\psi](x)=\sum_{\beta\not=0}(\psi(x+z_\beta)-\psi(x))K_{\beta,h},
\end{equation}
where
$K_{\beta,h}=\frac{1}{h^\alpha}\frac{1}{\Gamma(-\frac{\alpha}{2})}\int_0^\infty G(\beta,t)\frac{\dd t}{t^{1+\frac{\alpha}{2}}}.$
\end{lemma}
\begin{proof}
By  \eqref{discfractlapN1}, the representation formula for
$\Psi$, and $\sum_\beta G(\beta,\frac{t}{h^2})=1$, 
\begin{equation*}
\begin{split}
(-\Delta_h)^{\frac{\alpha}{2}}[\psi](x):=&\frac{1}{\Gamma(-\frac{\alpha}{2})}\int_0^\infty\bigg(\sum_{\beta}\psi(x-z_\beta)G\Big(\beta,\frac{t}{h^2}\Big)-\psi(x)\bigg)\frac{\dd t}{t^{1+\frac{\alpha}{2}}}\\
=&\frac{1}{\Gamma(-\frac{\alpha}{2})}\int_0^\infty\sum_{\beta}\left(\psi(x-z_\beta)-\psi(x)\right)G\Big(\beta,\frac{t}{h^2}\Big)\frac{\dd t}{t^{1+\frac{\alpha}{2}}}\\
=&\sum_{\beta}\left(\psi(x-z_\beta)-\psi(x)\right)\frac{1}{\Gamma(-\frac{\alpha}{2})}\int_0^\infty G\Big(\beta,\frac{t}{h^2}\Big)\frac{\dd t}{t^{1+\frac{\alpha}{2}}}.
\end{split}
\end{equation*}
The change of variables $\tau=t/h^2$ finishes the proof.
\end{proof}

\begin{remark}
In dimension $N=1$ a more explicit expression for
$K_{\beta,h}$ is given in \cite{CiRoStToVa16}:
$K_{j,h}:=\frac{1}{h^\alpha}\frac{2^\alpha\Gamma(\frac{1+\alpha}{2})\Gamma(|j|-\frac{\alpha}{2})}{\sqrt{\pi}|\Gamma(-\frac{\alpha}{2})|\Gamma(|j|+1+\frac{\alpha}{2})} \ \text{for} \ j\in \Z, \ j\not=0.$
\end{remark}

\begin{lemma}
Assume $h>0$, $\alpha\in(0,2)$, and $N\geq1$. The family
$\{-(-\Delta_h)^{\frac{\alpha}{2}}\}_{h>0}$ given by
\eqref{discfractlapN1} (or \eqref{eq:discflexp}) is an
\textbf{admissible} approximation of
$-(-\Delta)^{\frac{\alpha}{2}}$ with \textbf{Local Truncation Error}
\[
\|(-\Delta_h)^{\frac{\alpha}{2}}[\psi]-(-\Delta)^{\frac{\alpha}{2}}[\psi]\|_{L^p(\R)}= O(h^2),\]
for $p=\{1,\infty\}$ and $\psi \in C_c^\infty(\R^N)$.
\end{lemma}
\begin{proof}
By Lemma \ref{dfl_FD}, $-(-\Delta_h)^{\frac{\alpha}{2}}$ is an operator in
the form \eqref{FD} explicitly given by \eqref{eq:discflexp}. We show
that it is in the class \eqref{nuas}. It is clear that 
$K_{\beta,h}$ is nonnegative and symmetric in $\beta$ since these properties are
shared by $G$. Next,  
\[
C_2:=\sum_{\beta\not=0} \int_{1}^\infty  G(\beta,t)\frac{\dd t}{t^{1+\frac{\alpha}{2}}}= \int_{1}^\infty  \sum_{\beta\not=0} G(\beta,t)\frac{\dd t}{t^{1+\frac{\alpha}{2}}}\leq  \int_{1}^\infty  \frac{\dd t}{t^{1+\frac{\alpha}{2}}}=\frac{2}{\alpha},
\]
and
\[
C_1:=\sum_{\beta\not=0} \int_{0}^1  G(\beta,t)\frac{\dd t}{t^{1+\frac{\alpha}{2}}}=  \int_{0}^1  \sum_{\beta\not=0} G(\beta,t)\frac{\dd t}{t^{1+\frac{\alpha}{2}}}= \int_{0}^1 (1- G(0,t))\frac{\dd t}{t^{1+\frac{\alpha}{2}}}.
\]
By regularity and the properties of $G$,
\(
|1-G(0,t)|=|G(0,0)-G(0,t)|\leq C t
\)
for $C=\max_{\xi\in[0,t]}\{\partial_t G(0,t)\}$, and then
\(
C_1\leq C\int_0^1 t \frac{\dd t}{t^{1+\frac{\alpha}{2}}}=C\frac{2}{2-\alpha}.
\)
We conclude that
\[
\sum_{\beta\not=0}K_{\beta,h}=\frac{1}{h^\alpha}\frac{1}{\Gamma(-\frac{\alpha}{2})} (C_1+C_2)<+\infty,
\]
and that $-(-\Delta_h)^{\frac{\alpha}{2}}$ is in the class
\eqref{nuas} by
Lemma \ref{lemma:suitdisc}.

Now  we need to show that the discretization is consistent.  We proceed as in \cite{CuDTG-GPa17a,CuDTG-GPa17b}. Using directly the semigroup formulation \eqref{discfractlapN1}, we get that
\[
(-\Delta_h)^{\frac{\alpha}{2}}[\psi](x)-(-\Delta)^{\frac{\alpha}{2}}[\psi](x)= \frac{1}{\Gamma(-\frac{\alpha}{2})}\int_0^\infty\left(\e^{t\Delta_h}\psi (x)-\e^{t\Delta}\psi (x)\right)\frac{\dd t}{t^{1+\frac{\alpha}{2}}}
\]
where $\e^{t\Delta}\psi (x)$ is the solution of the heat equation
with initial condition $\psi$. Assume for the moment that the following estimate holds for $p=\{1,\infty\}$:
\begin{equation}\label{eq:estimheaterror}
\|\e^{t\Delta_h}\psi(\cdot) -\e^{t\Delta}\psi(\cdot)\|_{L^p(\R^N)}\leq\left\{
\begin{split}
Ct h^2 \quad &\textup{for} \quad 0\leq t\leq1, \\
C\frac{h^2}{t} \quad &\textup{for} \quad t\geq1,
\end{split}
\right.
\end{equation}
for some $C>0$ depending on $\psi$. Then,
\begin{equation*}
\begin{split}
\|(-\Delta_h)^{\frac{\alpha}{2}}[\psi]-(-\Delta)^{\frac{\alpha}{2}}[\psi]\|_{L^p(\R^N)}&\leq \frac{1}{\Gamma(-\frac{\alpha}{2})}\int_0^\infty\|\e^{t\Delta_h}\psi(\cdot) -\e^{t\Delta}\psi(\cdot)\|_{L^p(\R^N)}\frac{\dd t}{t^{1+\frac{\alpha}{2}}}\\
&\leq C h^2 \int_0^1 \frac{\dd t}{t^{\frac{\alpha}{2}}}+Ch^2\int_1^\infty \frac{\dd t}{t^{2+\frac{\alpha}{2}}}\leq \tilde{C} h^2.
\end{split}
\end{equation*}
From the $L^\infty$-estimate and Lemma \ref{prop:equivcons}, we also
have that \eqref{eq:unifLevy} holds.

It only remains  to prove estimate \eqref{eq:estimheaterror}. Recall that  
\begin{equation*}\label{eq:heatsolconv}
\e^{t\Delta}\psi(x)=\int_{\R^N}\psi(x-y)F(y,t)\dd y \quad \textup{with} \quad F(x,t)=\frac{1}{(4\pi t)^{N/2}}\e^{-\frac{|x|^2}{4t}},
\end{equation*}
and set $\tau(x,t):= \partial_t \e^{t\Delta}\psi(x)- \Delta_h
\e^{t\Delta}\psi (x)$. Since $\e^{t\Delta}\psi$ is smooth, a
Taylor expansion argument and the properties of $F, D^4 F$  show that
\begin{align}\label{eq:estimatelocaltrunc}
\|\tau(\cdot,t)\|_{L^p}&\leq h^2\|D^4
e^{t\Delta}\psi\|_{L^p}\leq
\begin{cases}
h^2\|D^4\psi\|_{L^p}\|F\|_{L^1} \leq C h^2 \quad \textup{for} \quad 0\leq t\leq1, \\
h^2\|\psi\|_{L^p}\|D^4F\|_{L^1} \leq C\frac{h^2}{t^2} \quad \textup{for} \quad t\geq1.
\end{cases}
\end{align}
Now let $E(x,t):=\e^{t\Delta_h}\psi(x) -\e^{t\Delta}\psi(x)$ and note that
$\partial_t E(x,t) = \Delta_h E(x,t)+\tau(x,t)$ and $E(x,0)=0$. The
weak maximum principle (for \eqref{eq:heateq1}) and a standard $L^1$ bound then
immediately yield 
$$\|E(\cdot,t)\|_{L^p}\leq \int_0^t\|\tau(\cdot,s)\|_{L^p}\dd s, $$
and  \eqref{eq:estimheaterror} follows from \eqref{eq:estimatelocaltrunc}.
\end{proof}

%%%%%%%%%%%%%%%%%%%%%%%%%%%%%%%%%%%%%%%%%%%%%%%%%%%%

\subsection{Approximation of the nonlinearity}\label{sec:aproxNL}
As we saw in Section \ref{sec:CFL}, we need to impose the condition
\eqref{CFL} when the schemes have some explicit part (i.e. when $0\leq\theta<1$). 
This condition  requires the nonlinearity  $\varphi$ to be
(locally) Lipschitz. But our results can handle merely continuous
$\varphi$. If $\varphi$ is not locally 
Lipschitz as e.g. in the fast diffusion case, we must 
replace it by a Lipschitz approximation to get explicit monotone schemes.
To be precise, we approximate
$\varphi$ by a sequence of nondecreasing Lipschitz functions
$\varphi^\epsilon$ converging  locally uniformly as
$\epsilon=o_h(1) \to 0^+$. The \eqref{CFL} condition is then
\[
\tag{CFL}
\Delta t (1-\theta) L_{\varphi^\epsilon} \nu_{h}(\R^N)\leq1. 
\]
 
Note that $L_{\varphi^\epsilon}\to\infty$ as $\epsilon\to 0^+$ making \eqref{CFL} a more and more restrictive condition as $h$ approaches zero. 

There are several ways of choosing the nonlinearity $\varphi^\epsilon$
in an \textbf{admissible} way. Two simple and general choices (cf. e.g. \cite{DTEnJa18a,Vaz07}) are
\begin{equation*}
\varphi^\epsilon(\xi):=
(\varphi*\omega_\epsilon)(\xi)-(\varphi*\omega_\epsilon)(0)\qquad\text{and}\qquad\varphi^\epsilon(\xi)+
\epsilon \xi, 
\end{equation*}
where $\omega_\epsilon$ is a standard mollifier in $\R$. However, in
many applications $\varphi$ is non-Lipschitz only at
the origin. A well-known  example is the Fast Diffusion Equation where
$\varphi(\xi)= \xi^m$ for $0<m<1$. In this case an easier and more
efficient choice is
\begin{equation*}
\varphi^\epsilon(\xi)=\left\{
\begin{array}{lll}
\varphi(\xi + \epsilon)-\varphi(\epsilon) &\textup{if} & \xi \geq0\\
\varphi(\xi - \epsilon)-\varphi(-\epsilon)  & \textup{if} &\xi <0.
\end{array}
\right.
\end{equation*} 
Clearly  $\varphi^\epsilon \to
\varphi$ locally uniformly and
$\varphi^\epsilon$ is Lipschitz with Lipschitz constant
\[
L_{\varphi^\epsilon}= |(\varphi^\epsilon)'(0)|= \frac{m}{\epsilon^{1-m}}.
\]
Moreover, this
approximation enjoys the very interesting property of preserving the
zero level sets of the solution since $\varphi^\epsilon
(0)=0=\varphi(0)$.  

%%%%%%%%%%%%%%%%%%%%%%%%%%%%%%%%%%%%%%%%%%%%%%%%%%%%
%%%%%%%%%%%%%%%%%%%%%NEW SECTION%%%%%%%%%%%%%%%%%%%%%%%
%%%%%%%%%%%%%%%%%%%%%%%%%%%%%%%%%%%%%%%%%%%%%%%%%%%%
\section{Numerical experiments in 1D}
\label{sec:numsim}

In this section, we test our numerical schemes on interesting
special cases of \eqref{E}--\eqref{IC} in one space dimension  that
involve the fractional Laplacian,
\begin{equation}\label{StefanProb}
\dell_t u(x,t) +(-\Delta)^{\frac{\alpha}{2}} [\varphi(u(\cdot,t))](x)=g(x,t)\qquad\text{in}\qquad Q_T:=\R\times(0,T)
\end{equation}
for $\alpha\in(0,2)$.
All the schemes are of the form \eqref{FullyDiscNumSch1}, 
and since our initial data and right-hand sides will be smooth, we simply
take $U_\beta^0=u_0(x_\beta)$ and $F_\beta^j=f(x_\beta,t_j)$. We
consider explicit  ($\theta=0$)  and implicit  ($\theta=1$)  schemes and the following spatial
discretizations of the nonlocal operators:
\begin{enumerate}[1)]
\item MpR $=$ Midpoint Rule + trivial discretization for singular part
\item FOI $=$ First Order Interpolation + trivial discretization for singular part
\item SOI $=$ Second Order Interpolation + adapted vanishing viscosity
  %for singular part
\item PDL $=$ Powers of the Discrete Laplacian
\end{enumerate}
All experiments have been run on equidistant grids in space and time,
with $\Delta t$ such that \eqref{CFL} holds and the overall order of
convergence is determined by the spatial discretization.
To compute the solutions we restrict to a
(sufficiently) large  bounded spatial domain
$I$ and set the numerical solution equal zero outside. See 
Section \ref{sec:TruncationDomainEffect} for some numerical tests on how the
size of the domain affects the error. 

The error is calculated either in $L^1(I)$ or $L^\infty(I)$
at a certain time $T>0$. To compute the error, we consider examples with (i)
known exact solutions, or (ii) we force a nice function
$v$ to be a solution by taking
\begin{equation}\label{rhsStefanProb}
g(x,t)=\dell_t v(x,t) +(-\Delta)^{\frac{\alpha}{2}} [\varphi(v(\cdot,t))](x),
\end{equation}
or (iii) we compute the errors numerically. In the latter case we
assume the error $E$ satisfies $E=Ch^\gamma$, take $h_j=c2^{-j}$, and compute an
estimate of the rate $\gamma$ as
$$
\gamma=\log_2\left(\frac{E_{j-1}}{E_{j}}\right).
$$ 

\subsection{Fractional Heat Equation (explicit scheme)}
\label{sec:FractionalHeatEquationExplicit}

We consider \eqref{StefanProb} with $\varphi(\xi)=\xi$, $\alpha=1$,
and $g\equiv0$. This is a Fractional Heat Equation with explicit fundamental
solution $K(x,t)=\frac{t}{t^2+|x|^2}$. 
We take $u_0(x)=K(x,1)$ so that the exact solution is
\begin{equation}\label{CauchyDistr}
u(x,t)=K(x,t+1).
\end{equation}
We also take $I=[-5000,5000]$.  By \eqref{CauchyDistr}, $u\sim10^{-8}$ 
outside of $I$, and hence is negligible in the error 
analysis. In Figure \ref{m1s1} and Table  \ref{tab:m1s1} we show the
error and rates  at time $T=1$.

\begin{figure}[!htbp] 
\includegraphics[width=0.75\textwidth]{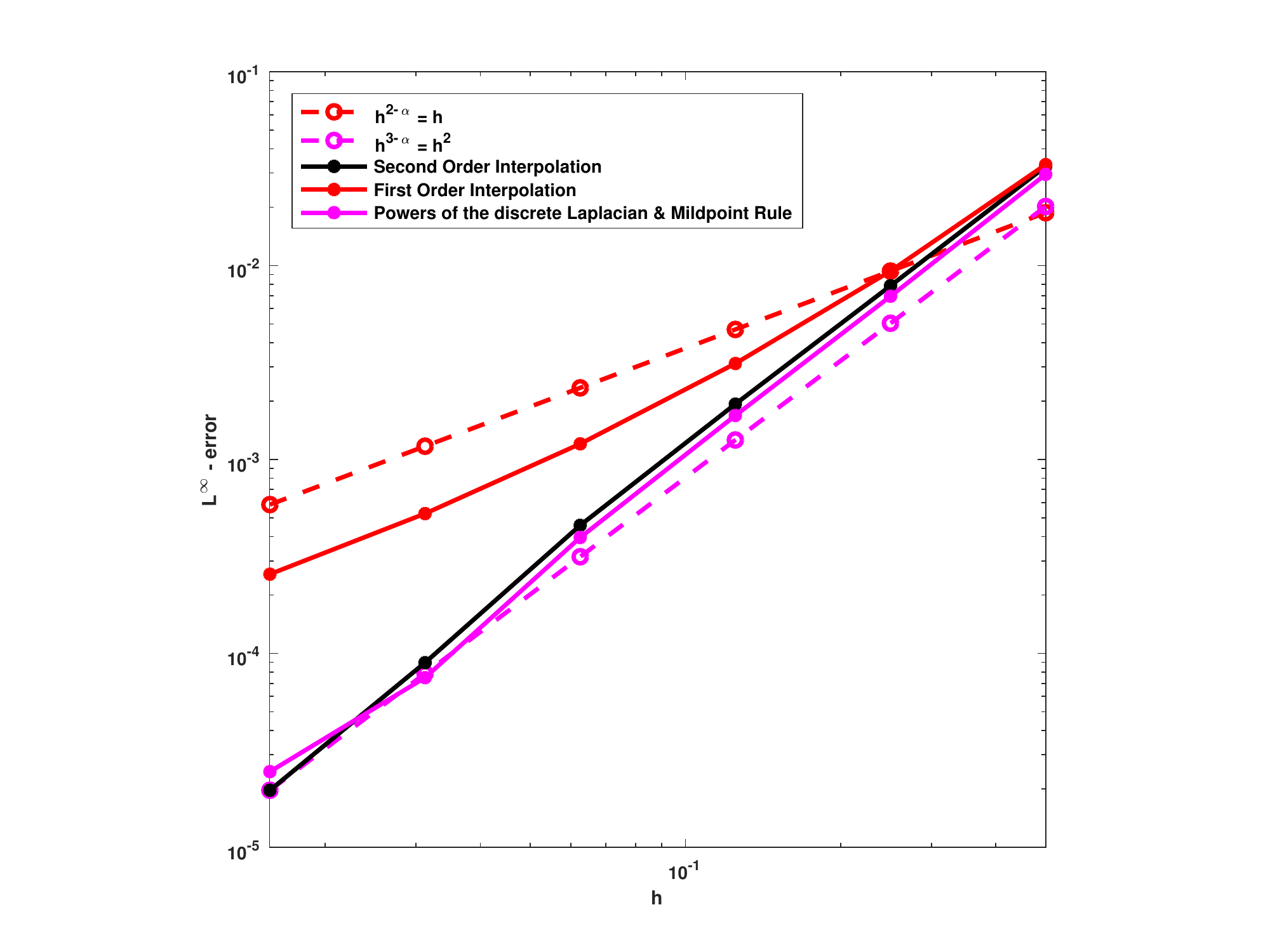}
\vspace{-5mm}\caption{\footnotesize$L^\infty$-error at $T=1$ with $I=[-5000,5000]$  for the exact solution \eqref{CauchyDistr} of \eqref{StefanProb} with $\alpha=1$, $\varphi(\xi)=\xi$ and $g\equiv0$.}
\label{m1s1}
\end{figure}

\begin{table}[!htbp]\footnotesize
\centering
\begin{tabular}{l|ll|ll|ll|ll}
 \ \ $h$&   MpR                    &  &     FOI                  &  &        SOI               &  &        PDL               &  \\ \cline{2-9} 
 & \multicolumn{1}{l|}{error} & $\gamma$  & \multicolumn{1}{l|}{error} & $\gamma$  & \multicolumn{1}{l|}{error} & $\gamma$  & \multicolumn{1}{l|}{error} & $\gamma$ \\ \hline
 $5.00$e-$1$& \multicolumn{1}{l|}{$2.95$e-$2$} &  & \multicolumn{1}{l|}{$ 3.31$e-$2 $} &  & \multicolumn{1}{l|}{$ 3.24$e-$3 $} &  & \multicolumn{1}{l|}{$2.95$e-$2$} &                       \\
 $2.50$e-$1$& \multicolumn{1}{l|}{$6.94 $e-$3 $} &2.08  & \multicolumn{1}{l|}{$9.40 $e-$3 $} &1.82  & \multicolumn{1}{l|}{$ 7.89$e-$3 $} &2.04  & \multicolumn{1}{l|}{$6.94 $e-$3 $} &2.08                       \\
$1.25$e-$1$& \multicolumn{1}{l|}{$1.68 $e-$3 $} &2.04  & \multicolumn{1}{l|}{$3.12 $e-$3 $} &1.58  & \multicolumn{1}{l|}{$ 1.93 $e-$3 $} &2.03  & \multicolumn{1}{l|}{$1.68 $e-$3 $} &2.04                       \\
$6.25$e-$2$& \multicolumn{1}{l|}{$3.95 $e-$4 $} &2.09  & \multicolumn{1}{l|}{$1.20 $e-$4 $} &1.37  & \multicolumn{1}{l|}{$ 4.57 $e-$4 $}&2.08  & \multicolumn{1}{l|}{$3.95 $e-$4 $} &2.09                     \\
$3.13$e-$2$& \multicolumn{1}{l|}{$7.50 $e-$5 $} &2.40  & \multicolumn{1}{l|}{$5.26 $e-$4 $} &1.19  & \multicolumn{1}{l|}{$ 8.96 $e-$5 $} & 2.35  & \multicolumn{1}{l|}{$7.50 $e-$5 $} &2.40                       \\
$1.56$e-$2$& \multicolumn{1}{l|}{$2.45 $e-$5 $} &1.61  & \multicolumn{1}{l|}{$2.56 $e-$4 $} &1.03  & \multicolumn{1}{l|}{$1.97 $e-$ 5$} &2.18  & \multicolumn{1}{l|}{$2.45 $e-$5 $} &1.61                      
\end{tabular}
\vspace{2mm}
\caption{\footnotesize$L^\infty$-error at $T=1$ with $I=[-5000,5000]$  for the exact solution \eqref{CauchyDistr} of \eqref{StefanProb} with $\alpha=1$, $\varphi(\xi)=\xi$ and $g\equiv0$.}
\label{tab:m1s1}
\end{table}

\noindent{\bf Conclusion:} The Midpoint Rule and the Powers of the
Discrete Laplacian coincide when $\alpha=1$ and $N=1$ (see Section 
\ref{sec:discspa}). The theoretical convergence rates are confirmed
for all tested methods. Note that when $h=1.56\textup{e-}2$, the rate
for MpR and PDL is $\gamma=1.61$ in stead of $\gamma\sim2$. This is a
consequence of how the rates are calculated since the previous
value $\gamma=2.40$ is much better than predicted. If we calculate the
rate for $h=1.56\textup{e-}2$ 
with respect to $h=6.25\textup{e-}2$, we get $\gamma=2.01$.

\subsection{Fractional Porous Medium Equation  (explicit scheme)}

We consider \eqref{StefanProb} with $\varphi(\xi)=\xi^2$. An explicit
solution of that problem is known when $\alpha=\frac{1}{3}$
(cf. \cite{Hua14}), but its slow decay at infinity makes it difficult
to find a reasonably small computational domain $I$. 
To overcome this issue we impose $v=(t+1)\e^{-|x|^2}$ as
a solution by taking $g$ according to \eqref{rhsStefanProb}. Note that
now the solution has exponential decay. The errors and rates of
convergence are shown for $I=[-100,100]$ at time $T=1$ for
$\alpha=0.5$ (resp. $\alpha=1.5$) in Figure \ref{m2s0515} and Table
\ref{tab:m2s05} (resp. Table \ref{tab:m2s15}).

\begin{figure}[!htbp]
 \hspace{-0.3cm}
  \begin{minipage}[b]{0.5\textwidth}
    \includegraphics[width=\textwidth]{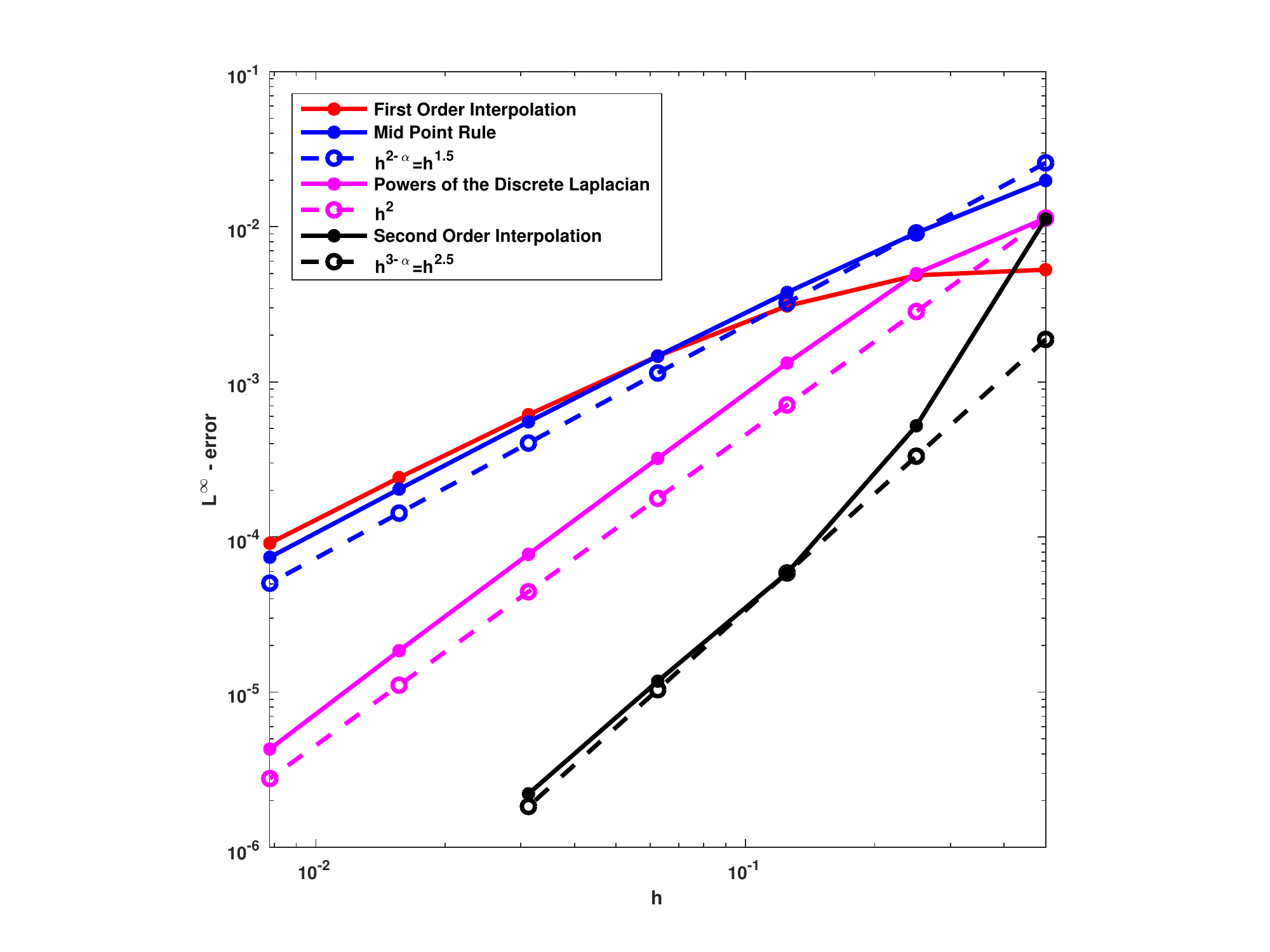}
  \end{minipage}
  \begin{minipage}[b]{0.5\textwidth}
    \includegraphics[width=\textwidth]{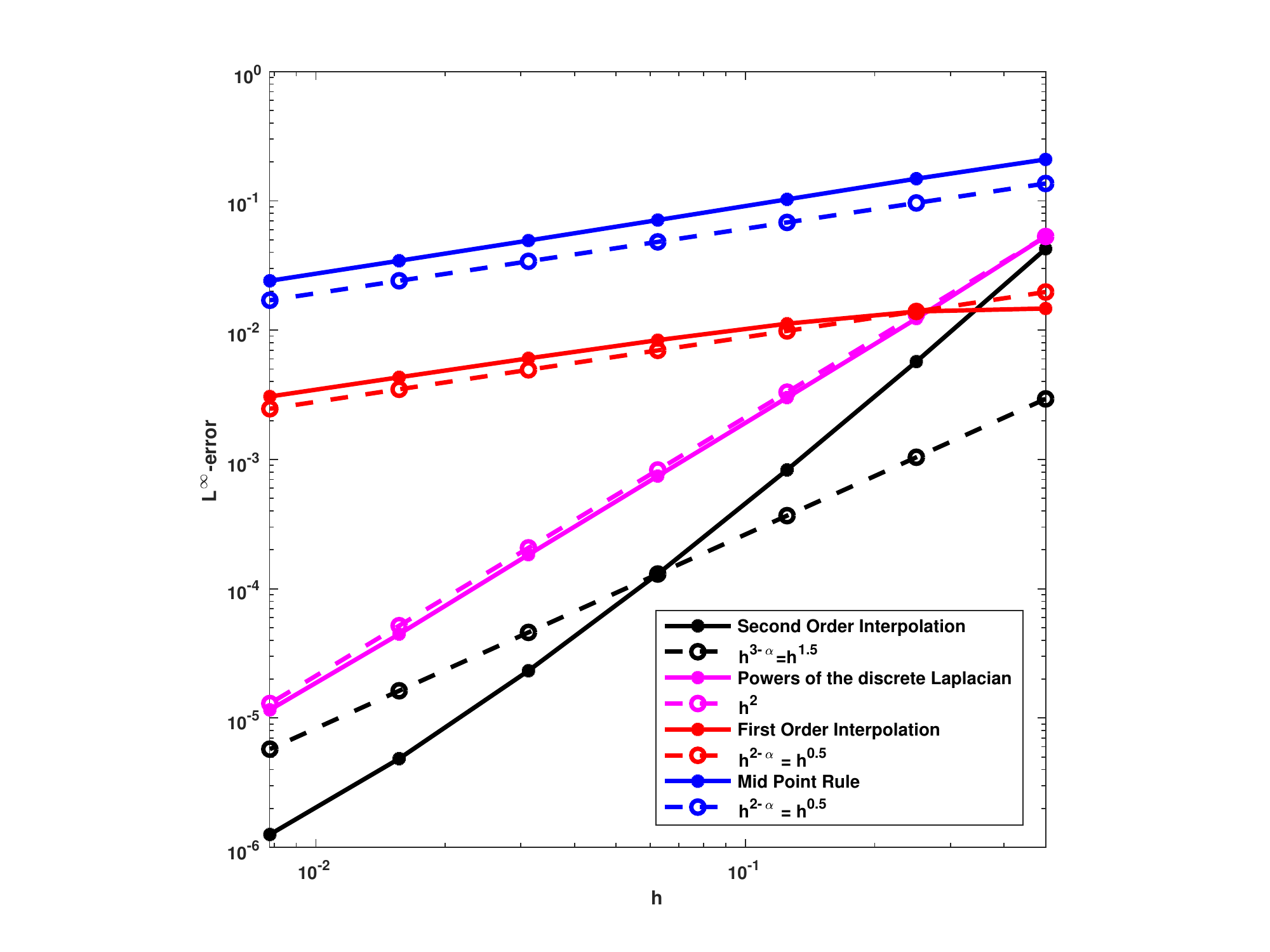}
  \end{minipage}
  \vspace{-5mm}
  \caption{\footnotesize$L^\infty$-error at $T=1$ with $I=[-100,100]$ for \eqref{StefanProb} with $\varphi(\xi)=\xi^2$ and $g$ by \eqref{rhsStefanProb} with $v(x,t)=(t+1)\e^{-|x|^2}$. Left: $\alpha=0.5$. Right: $\alpha=1.5$.}
  \label{m2s0515}
\end{figure}

\begin{table}[!htbp]\footnotesize
\centering
\begin{tabular}{l|ll|ll|ll|ll}
 \ \ $h$&   MPR                    &  &     FOI                  &  &        SOI               &  &        PDL               &  \\ \cline{2-9} 
 & \multicolumn{1}{l|}{error} & $\gamma$  & \multicolumn{1}{l|}{error} & $\gamma$  & \multicolumn{1}{l|}{error} & $\gamma$  & \multicolumn{1}{l|}{error} & $\gamma$ \\ \hline
 $5.00$e-$1$& \multicolumn{1}{l|}{$1.98 $e-$2 $} &        & \multicolumn{1}{l|}{$ 5.29$e-$3 $} &  & \multicolumn{1}{l|}{$1.12 $e-$ 2$} &  & \multicolumn{1}{l|}{$1.14 $e-$2 $} &                       \\
 $2.50$e-$1$& \multicolumn{1}{l|}{$9.13 $e-$3 $} &1.22  & \multicolumn{1}{l|}{$ 4.86$e-$3 $} &0.12  & \multicolumn{1}{l|}{$5.21 $e-$4 $} &4.42  & \multicolumn{1}{l|}{$5.00 $e-$3 $} & 1.19                      \\
$1.25$e-$1$& \multicolumn{1}{l|}{$3.77 $e-$3 $} &1.28  & \multicolumn{1}{l|}{$ 3.09$e-$3 $} &0.65  & \multicolumn{1}{l|}{$5.87 $e-$ 5$} &3.15  & \multicolumn{1}{l|}{$1.32 $e-$3 $} & 1.91                      \\
$6.25$e-$2$& \multicolumn{1}{l|}{$1.47 $e-$3 $} &1.36  & \multicolumn{1}{l|}{$ 1.47$e-$3 $} &1.07  & \multicolumn{1}{l|}{$1.18 $e-$5 $} &2.31  & \multicolumn{1}{l|}{$3.21 $e-$4 $} &  2.04                     \\
$3.13$e-$2$& \multicolumn{1}{l|}{$ 5.53$e-$4 $} &1.40 & \multicolumn{1}{l|}{$ 6.15$e-$4 $} &1.25  & \multicolumn{1}{l|}{$2.21 $e-$6$} &2.41  & \multicolumn{1}{l|}{$7.74 $e-$5 $} &  2.05                     \\
$1.56$e-$2$& \multicolumn{1}{l|}{$2.04 $e-$4 $} &1.44  & \multicolumn{1}{l|}{$ 2.41$e-$4 $} &1.35  & \multicolumn{1}{l|}{--------} &-----  & \multicolumn{1}{l|}{$ 1.85$e-$5 $} &    2.06                         \\
$7.81$e-$3$& \multicolumn{1}{l|}{$7.41 $e-$5 $} &1.46  & \multicolumn{1}{l|}{$ 9.13$e-$ 5$} &1.40  & \multicolumn{1}{l|}{--------} &-----  & \multicolumn{1}{l|}{$ 4.30$e-$6 $} &     2.11                  
\end{tabular}
\vspace{2mm}
\caption{\footnotesize$L^\infty$-error at $T=1$ with $I=[-100,100]$ for \eqref{StefanProb} with $\alpha=0.5$, $\varphi(\xi)=\xi^2$ and $g$ by \eqref{rhsStefanProb} with $v(x,t)=(t+1)\e^{-|x|^2}$.}
\label{tab:m2s05}
\end{table}

\begin{table}[!htbp]\footnotesize
\centering
\begin{tabular}{l|ll|ll|ll|ll}
 $h$&   MPR                    &  &     FOI                  &  &        SOI               &  &        PDL               &  \\ \cline{2-9} 
 & \multicolumn{1}{l|}{error} & $\gamma$  & \multicolumn{1}{l|}{error} & $\gamma$  & \multicolumn{1}{l|}{error} & $\gamma$  & \multicolumn{1}{l|}{error} & $\gamma$ \\ \hline
 $5.00$e-$1$& \multicolumn{1}{l|}{$2.10 $e-$1 $} &  & \multicolumn{1}{l|}{$ 1.47$e-$ 2$} &  & \multicolumn{1}{l|}{$4.26 $e-$2 $} &  & \multicolumn{1}{l|}{$5.30 $e-$ 2$} &                       \\
 $2.50$e-$1$& \multicolumn{1}{l|}{$1.49 $e-$1 $} & 0.49  & \multicolumn{1}{l|}{$ 1.40$e-$2 $} &0.08 & \multicolumn{1}{l|}{$5.71 $e-$3 $} &2.90  & \multicolumn{1}{l|}{$1.23 $e-$ 2$} & 2.11                      \\
$1.25$e-$1$& \multicolumn{1}{l|}{$1.03 $e-$1 $} &0.53  & \multicolumn{1}{l|}{$ 1.12$e-$2 $} &0.31  & \multicolumn{1}{l|}{$8.30 $e-$4 $} &2.78  & \multicolumn{1}{l|}{$3.01 $e-$ 3$} &  2.03                     \\
$6.25$e-$2$& \multicolumn{1}{l|}{$7.11 $e-$2 $} & 0.53 & \multicolumn{1}{l|}{$ 8.37$e-$ 3$} &0.42  & \multicolumn{1}{l|}{$ 1.30$e-$4 $} &2.67  & \multicolumn{1}{l|}{$ 7.44$e-$ 4$} &   2.16                    \\
$3.13$e-$2$& \multicolumn{1}{l|}{$4.93 $e-$2 $} & 0.53  & \multicolumn{1}{l|}{$6.05 $e-$ 3$} & 0.47 & \multicolumn{1}{l|}{$2.32 $e-$5 $} &2.48  & \multicolumn{1}{l|}{$ 1.83$e-$4 $} &  2.02                     \\
$1.56$e-$2$& \multicolumn{1}{l|}{$3.44 $e-$2 $} &0.52  & \multicolumn{1}{l|}{$4.32 $e-$ 3$} & 0.49  & \multicolumn{1}{l|}{$4.85 $e-$6 $} &2.25  & \multicolumn{1}{l|}{$ 4.46$e-$ 5$} &  2.04
 \\
$7.81$e-$3$& \multicolumn{1}{l|}{$2.41 $e-$2 $} &0.51  & \multicolumn{1}{l|}{$3.07 $e-$3 $} &0.49  & \multicolumn{1}{l|}{$1.25 $e-$6 $} &1.94  & \multicolumn{1}{l|}{$1.16 $e-$5 $} &   1.95                    
\end{tabular}
\vspace{2mm}
\caption{\footnotesize$L^\infty$-error at  $T=1$ with $I=[-100,100]$ for \eqref{StefanProb} with $\alpha=1.5$, $\varphi(\xi)=\xi^2$ and $g$ by \eqref{rhsStefanProb} with $v(x,t)=(t+1)\e^{-|x|^2}$.}
\label{tab:m2s15}
\end{table}

\noindent{\bf Conclusion:} When $\alpha=0.5$ all the expected rates are recovered. For the Second Order Interpolation, we excluded the last two rows because the error was much smaller compared to the other methods -- see Table \ref{tab:m2s05}. It is also worth noting that when $\alpha=1.5$ -- see Table \ref{tab:m2s15}, the expected rate of convergence for SOI is $\gamma=1.5$ while the rate obtained by the experiment is $\sim2.0$.

\begin{figure}[!htbp]
\centering
\includegraphics[width=0.75\textwidth]{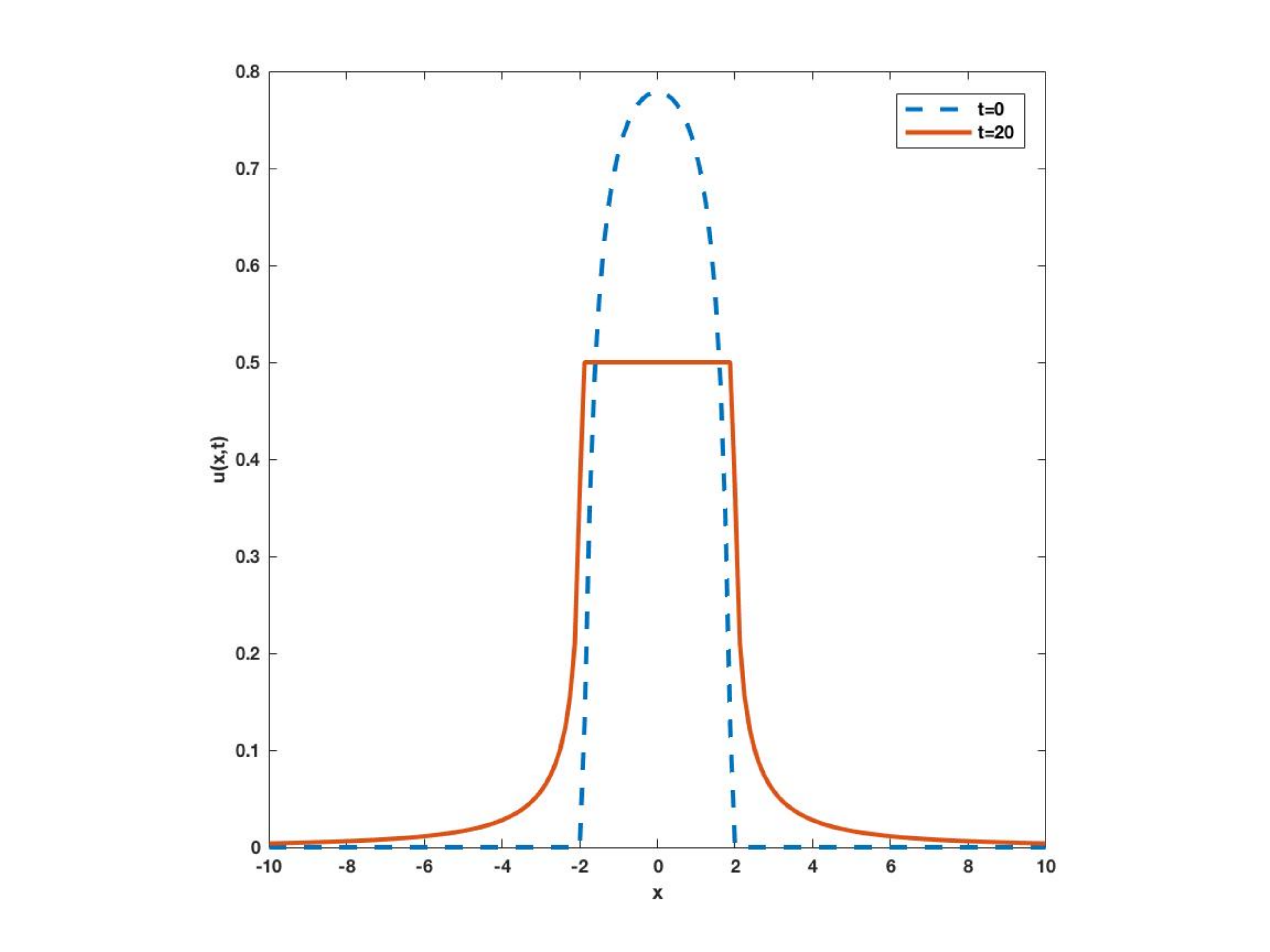}
\vspace{-5mm}
\caption{\footnotesize Solution of \eqref{StefanProb} for  $\alpha=1$, $\varphi(\xi)=\max\{0,\xi-0.5\}$, and $g\equiv0$.}
\label{StefanSol}
\end{figure}

\begin{figure}[!htbp] 
\includegraphics[width=0.75\textwidth]{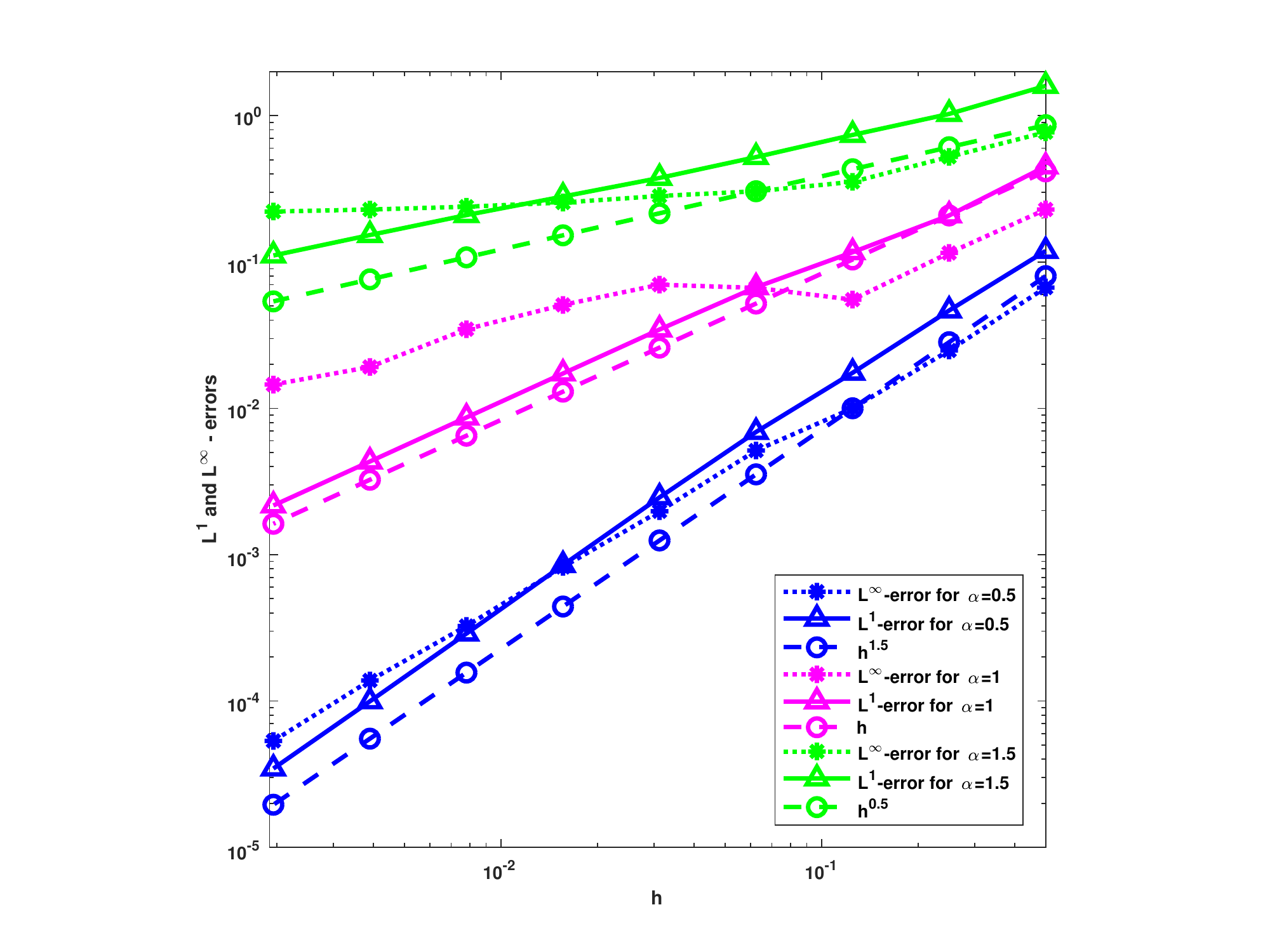}
\vspace{-5mm}
\caption{\footnotesize $L^\infty$- and $L^1$-errors with MpR at $T=1$ with $I=[-100,100]$ for \eqref{StefanProb} with $\varphi$ given by \eqref{varphiST} and $g$ by \eqref{rhsStefanProb} with $v(x,t)=(t+1)\e^{-|x|^2}$.}
\label{StefanComp}
\end{figure}

\subsection{Stefan type problem  (explicit scheme)}
\label{sec:StefanTyepProblemExplicit}

First we consider \eqref{StefanProb} with
$\varphi(\xi)=\max\{0,\xi-0.5\}$ (a globally Lipschitz function) and
$g\equiv0$. Solutions of such problems can loose regularity as we see
in Figure \ref{StefanSol} (above). Here we have used $C_\textup{c}^\infty(\R)$
initial data $u_0(x)=\e^{-\frac{1}{4-x^2}}\mathbf{1}_{[-2,2]}(x)$.

Next we consider \eqref{StefanProb} with 
\begin{equation}\label{varphiST}
\varphi(\xi)=\left\{
\begin{array}{lll}
\xi&\textup{if}& \xi<0.2\\
0.2&\textup{if}&0.2\leq \xi <0.4\\
\xi-0.2&\textup{if}& \xi\geq 0.4,
\end{array}
\right.
\end{equation}
and impose $v(x,t)=(t+1)\e^{-|x|^2}$ as the solution by taking $g$ as
in \eqref{rhsStefanProb}. We run experiments for $\alpha=0.5$,
$\alpha=1$, and $\alpha=1.5$ up to $T=1$ with $I=[-100,100]$ using the
Midpoint Rule. The results are given in Figure \ref{StefanComp} (above) and
Table \ref{tab:stefanLinf} (below). 

\noindent{\bf Conclusion:} It is interesting to note that even when
the solution $u\in C^\infty(\R^N)$, $\varphi(u)$ is just a Lipschitz
function. Therefore $-(-\Delta)^\frac{1}{2}[\varphi(u)]$ and
$-(-\Delta)^\frac{3}{4}[\varphi(u)]$ do not exist in every point and 
$u$ is not a classical solution. This fact
has a strong effect on the convergence rates in $L^\infty$, but not in
$L^1$, see Figure \ref{StefanComp} and Table
\ref{tab:stefanLinf}.

\begin{remark}
  \label{Stefan}
In \cite{BrChQu12} the authors study a nonlocal Stefan
problem of the form \eqref{E} with $\sigma\equiv0$ and
$\mu$ having a nonnegative continuous compactly supported density
(\eqref{muas} holds). This case is easier than the above example since
$\Operator$ is bounded. Nonetheless it has discontinuous solutions
that are computed numerically in Chapter 7 of \cite{BrChQu12}. Our
results provide a rigorous justification for these computations.
\end{remark}

\begin{table}[!htbp]\footnotesize
\centering
\begin{tabular}{l|llll|llll|llll}
 $h$&   $\alpha=\frac{1}{2}$&  & & &     $\alpha=1 $              &  & & &        $\alpha=\frac{3}{2}$         & &      &   \\ \cline{2-13} 
 & \multicolumn{1}{l|}{$ L^\infty$}& \multicolumn{1}{l|}{$\gamma$}  & \multicolumn{1}{l|}{$ L^1$}& \multicolumn{1}{l|}{$\gamma$} & \multicolumn{1}{l|}{$L^\infty$} & \multicolumn{1}{l|}{$\gamma$} &  \multicolumn{1}{l|}{$ L^1$}& \multicolumn{1}{l|}{$\gamma$}  & \multicolumn{1}{l|}{$L^\infty$} & \multicolumn{1}{l|}{$\gamma$} & \multicolumn{1}{l|}{$ L^1$}& $\gamma$   \\ \hline
 $5.0$e-$1$& \multicolumn{1}{l|}{$ 6.7$e-$ 2$}& \multicolumn{1}{l|}{}  &\multicolumn{1}{l|}{$1.3 $e-$ 1$}&& \multicolumn{1}{l|}{$2.3 $e-$ 1$} & \multicolumn{1}{l|}{} &\multicolumn{1}{l|}{$4.5 $e-$ 1$} &  & \multicolumn{1}{l|}{$7.7$e-$1 $} &  \multicolumn{1}{l|}{}  & \multicolumn{1}{l|}{$1.6 $e-$ 0$} &                    \\
 $2.5$e-$1$& \multicolumn{1}{l|}{$2.5 $e-$ 2$}&\multicolumn{1}{l|}{1.4}  &\multicolumn{1}{l|}{$4.7 $e-$ 2$}&1.4 & \multicolumn{1}{l|}{$ 1.2$e-$ 1$} &\multicolumn{1}{l|}{1} & \multicolumn{1}{l|}{$ 2.1$e-$1 $} &1.1& \multicolumn{1}{l|}{$5.2 $e-$1 $} & \multicolumn{1}{l|}{0.6}  & \multicolumn{1}{l|}{$1.0 $e-$ 0$} &  0.6                  \\
$1.3$e-$1$& \multicolumn{1}{l|}{$ 1.0$e-$ 2$}& \multicolumn{1}{l|}{1.3} &\multicolumn{1}{l|}{$ 1.8$e-$ 2$}&1.4& \multicolumn{1}{l|}{$ 5.5$e-$ 2$} &\multicolumn{1}{l|}{ 1.1} &\multicolumn{1}{l|}{$1.2 $e-$ 1$} & 0.8& \multicolumn{1}{l|}{$3.5 $e-$1 $} & \multicolumn{1}{l|}{0.6}  &  \multicolumn{1}{l|}{$ 7.4$e-$ 1$} &       0.5       \\
$6.3$e-$2$& \multicolumn{1}{l|}{$ 5.2$e-$ 3$}& \multicolumn{1}{l|}{1.0}  &\multicolumn{1}{l|}{$6.9 $e-$ 3$}&1.3& \multicolumn{1}{l|}{$ 6.7$e-$ 2$} &  \multicolumn{1}{l|}{$<$0}& \multicolumn{1}{l|}{$6.7 $e-$ 2$} & 0.8 & \multicolumn{1}{l|}{$3.0 $e-$1 $} &  \multicolumn{1}{l|}{0.2} &  \multicolumn{1}{l|}{$ 5.2$e-$ 1$} &        0.5                \\
$3.1$e-$2$& \multicolumn{1}{l|}{$2.0 $e-$ 3$}& \multicolumn{1}{l|}{1.4}   &\multicolumn{1}{l|}{$ 2.5$e-$ 3$}  &1.5& \multicolumn{1}{l|}{$7.0 $e-$ 2$} & \multicolumn{1}{l|}{$<$0} & \multicolumn{1}{l|}{$3.5 $e-$2 $} & 1.0& \multicolumn{1}{l|}{$2.8 $e-$ 1$} &   \multicolumn{1}{l|}{0.1}    & \multicolumn{1}{l|}{$ 3.8$e-$ 1$} &       0.5           \\
$1.6$e-$2$& \multicolumn{1}{l|}{$ 8.2$e-$ 4$}& \multicolumn{1}{l|}{1.3}  &\multicolumn{1}{l|}{$ 8.5$e-$ 4$}&1.5& \multicolumn{1}{l|}{$ 5.1$e-$ 2$} &  \multicolumn{1}{l|}{0.5} &\multicolumn{1}{l|}{$ 1.7$e-$ 2$} & 1.0 & \multicolumn{1}{l|}{$2.6 $e-$1 $} & \multicolumn{1}{l|}{0.2} & \multicolumn{1}{l|}{$ 2.8$e-$ 1$} & 0.4  \\
$7.8$e-$3$& \multicolumn{1}{l|}{$ 3.3$e-$ 4$}& \multicolumn{1}{l|}{1.3} &\multicolumn{1}{l|}{$ 2.9$e-$ 4$}& 1.6& \multicolumn{1}{l|}{$3.5 $e-$ 2$} & \multicolumn{1}{l|}{0.6} & \multicolumn{1}{l|}{$ 8.5$e-$ 3$} &  1.0& \multicolumn{1}{l|}{$ 2.4$e-$1 $} & \multicolumn{1}{l|}{0.1} & \multicolumn{1}{l|}{$ 2.1$e-$ 1$} &  0.4\\
$3.9$e-$3$& \multicolumn{1}{l|}{$ 1.4$e-$ 4$}& \multicolumn{1}{l|}{1.2}  &\multicolumn{1}{l|}{$ 9.9$e-$ 5$}&1.5& \multicolumn{1}{l|}{$ 1.9$e-$ 2$} & \multicolumn{1}{l|}{0.9} & \multicolumn{1}{l|}{$ 4.3$e-$ 3$} & 1.0& \multicolumn{1}{l|}{$ 2.3$e-$1 $} &  \multicolumn{1}{l|}{0.1}   & \multicolumn{1}{l|}{$ 1.5$e-$ 1$} &     0.5         \\
$2.0$e-$3$& \multicolumn{1}{l|}{$ 5.3$e-$ 5$}& \multicolumn{1}{l|}{1.4}  &\multicolumn{1}{l|}{$ 3.4$e-$ 5$}&1.5& \multicolumn{1}{l|}{$ 1.5$e-$ 2$} & \multicolumn{1}{l|}{0.4} &\multicolumn{1}{l|}{$ 2.2$e-$3 $} & 1.0& \multicolumn{1}{l|}{$2.2 $e-$1 $} &   \multicolumn{1}{l|}{0.1}   &   \multicolumn{1}{l|}{$ 1.1$e-$ 1$} &   0.5                      
\end{tabular}
\vspace{2mm}
\caption{\footnotesize$L^\infty$- and $L^1$-errors with MpR at $T=1$ with $I=[-100,100]$ for \eqref{StefanProb} with $\varphi$ given by \eqref{varphiST} and $g$ by \eqref{rhsStefanProb} with $v(x,t)=(t+1)\e^{-|x|^2}$.}
\label{tab:stefanLinf}
\end{table}

\subsection{Fractional Fast Diffusion Equation}
In the fast diffusion case a new difficulty appears: The
nonlinearity, $\varphi(\xi)=\xi^m$ for $m\in(0,1)$, is no longer locally
Lipschitz, and the \eqref{CFL} condition can only hold for implicit
schemes or under approximation of $\varphi$. Even in the local case
there are few results for this case, and the results in this paper are
as far as we know the first in the fractional case.

\begin{figure}[!htbp] 
\centering
\includegraphics[width=0.75\textwidth]{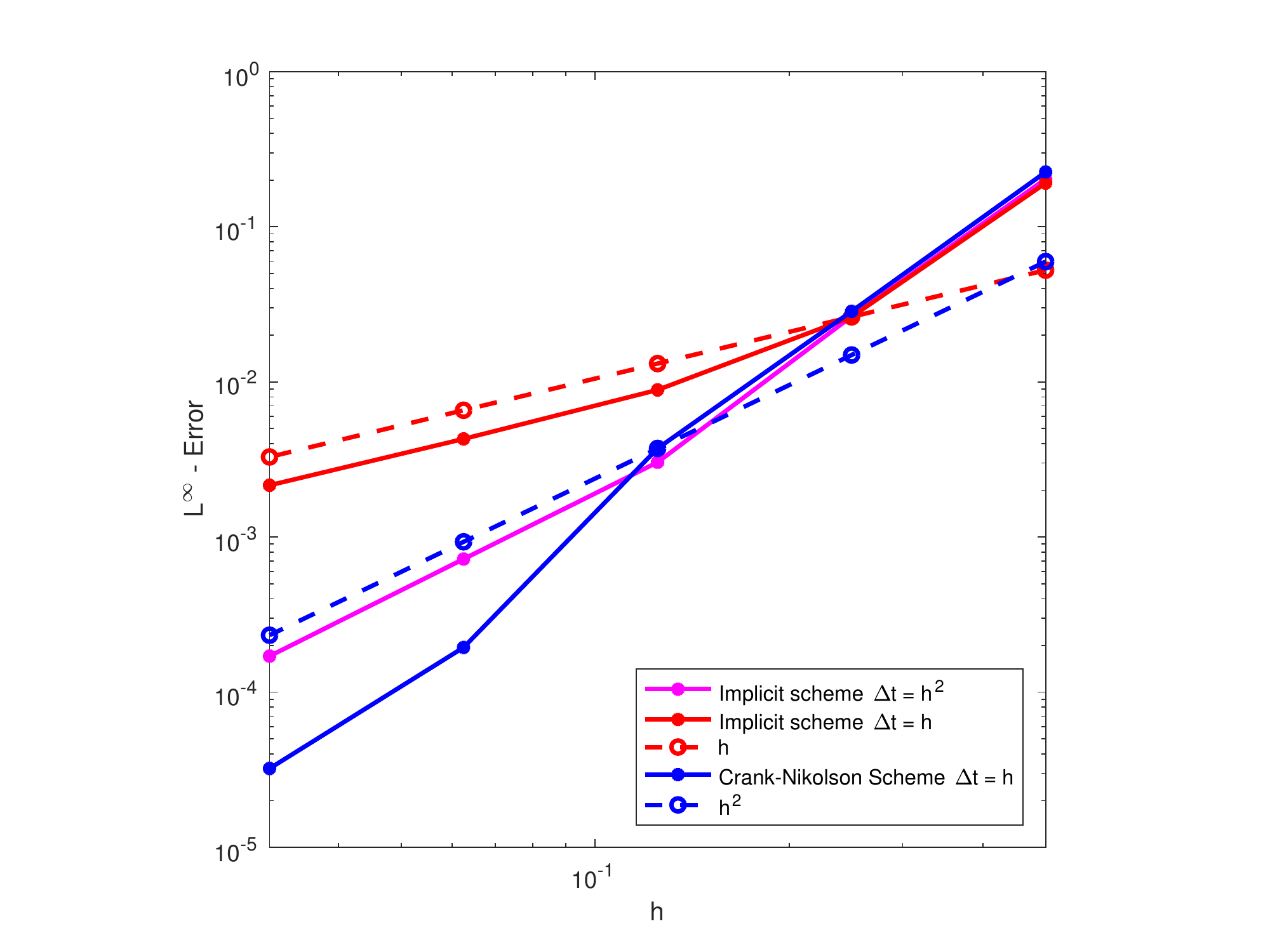}
\vspace{-5mm}
\caption{\footnotesize $L^\infty$-error with MpR at $T=1$ with $I=[-4,4]$ for \eqref{StefanProb} with $\varphi(\xi)=\xi^{\frac{1}{2}}$, $\alpha=1$ and $g$ given by \eqref{rhsStefanProb} with $v(x,t)=\sqrt{(t+1)}\e^{-|x|^8}$.}
\label{Impl_meth}
\end{figure}

\begin{table}[!htbp]\footnotesize
\centering
\begin{tabular}{l|ll|ll|ll}
 $h$&  Implicit                    & $\Delta t\sim h$ &     Implicit                  & $\Delta t\sim h^2$ &        Cr-Ni        &$\Delta t\sim h$          \\ \cline{2-7} 
 & \multicolumn{1}{l|}{error} & $\gamma$  & \multicolumn{1}{l|}{error} & $\gamma$  & \multicolumn{1}{l|}{error} & $\gamma$    \\ \cline{1-7} 
 $5.00$e-$1$& \multicolumn{1}{l|}{$1.91 $e-$1 $} &  & \multicolumn{1}{l|}{$ 2.03$e-$ 1$} &  & \multicolumn{1}{l|}{$2.25 $e-$1 $}   \\
 $2.50$e-$1$& \multicolumn{1}{l|}{$2.63 $e-$2 $} & 2.86  & \multicolumn{1}{l|}{$ 2.64$e-$2 $} &2.95 & \multicolumn{1}{l|}{$2.85 $e-$2 $} &2.98                        \\
$1.25$e-$1$& \multicolumn{1}{l|}{$8.88 $e-$3 $} & 1.56  & \multicolumn{1}{l|}{$ 3.03$e-$3 $} &3.12  & \multicolumn{1}{l|}{$3.73 $e-$3 $} &2.93                       \\
$6.25$e-$2$& \multicolumn{1}{l|}{$4.29 $e-$3 $} & 1.05 & \multicolumn{1}{l|}{$ 7.21$e-$ 4$} &2.07  & \multicolumn{1}{l|}{$ 1.94$e-$4 $} &4.26                     \\
$3.13$e-$2$& \multicolumn{1}{l|}{$2.16 $e-$3 $} & 1.00  & \multicolumn{1}{l|}{$1.71 $e-$ 4$} & 2.08 & \multicolumn{1}{l|}{$3.22 $e-$5 $} &2.59                        \\               
\end{tabular}
\vspace{2mm}
\caption{\footnotesize$L^\infty$-error with MpR at $T=1$ with $I=[-4,4]$ for \eqref{StefanProb} with $\varphi(\xi)=\xi^{\frac{1}{2}}$, $\alpha=1$ and $g$ given by \eqref{rhsStefanProb} with $v(x,t)=\sqrt{(t+1)}\e^{-|x|^8}$.}
\label{tab:Impl_meth}
\end{table}

\subsubsection{Implicit scheme}

Implicit schemes automatically satisfy \eqref{CFL} at the price of
having to solve a nonlinear system of equations at every time
step. This is computationally very expensive. For simplicity we use
the standard nonlinear solver ``\texttt{fsolve}'' in Matlab. A more
adapted solver could probably significantly reduce the computational
time.  Due to the computational cost, we take a very small
domain and a very regular solution. We consider
\eqref{StefanProb} with $\varphi(\xi)=\xi^{\frac{1}{2}}$, $\alpha=1$
and $g$ given by \eqref{rhsStefanProb} with
$v(x,t)=\sqrt{(t+1)}\e^{-|x|^8}$ (the solution). We run the
experiments with the Midpoint Rule up to time $T=1$ in the domain
$I=[-4,4]$. For the time discretization we choose $\Delta t \sim h$
and $\Delta t \sim h^2$. We also consider the Crank-Nicolson method ($\theta=\frac{1}{2}$)  with
$\Delta t \sim h$ which means that
\eqref{CFL} is satisfied. The results are shown in Figure
\ref{Impl_meth} and Table \ref{tab:Impl_meth} (above).

\smallskip
\noindent{\bf Conclusion:} For the implicit method with the Midpoint
Rule (a second  order  method when $\alpha=1$), the expected error is
$O(h^2+\Delta t)$. When $\Delta t\sim h$ the error is clearly governed
by $\Delta t$ and so the rate of convergence is linear.  For $\Delta
t\sim h^2$, the error introduced by the time discretization is
proportional to $h^2$ and so the rate of convergence is quadratic. For
the Crank-Nicolson method, the expected error is $O(h^2+\Delta t^2)$,
so with $\Delta t\sim h$ we should see second order
convergence. However the observed rates are better, see Table
\ref{tab:Impl_meth}. The Crank-Nicolson methods seems to be better than
both implicit methods in terms of computational time and accuracy.

\begin{figure}[!htbp]
\centering
\includegraphics[width=0.75\textwidth]{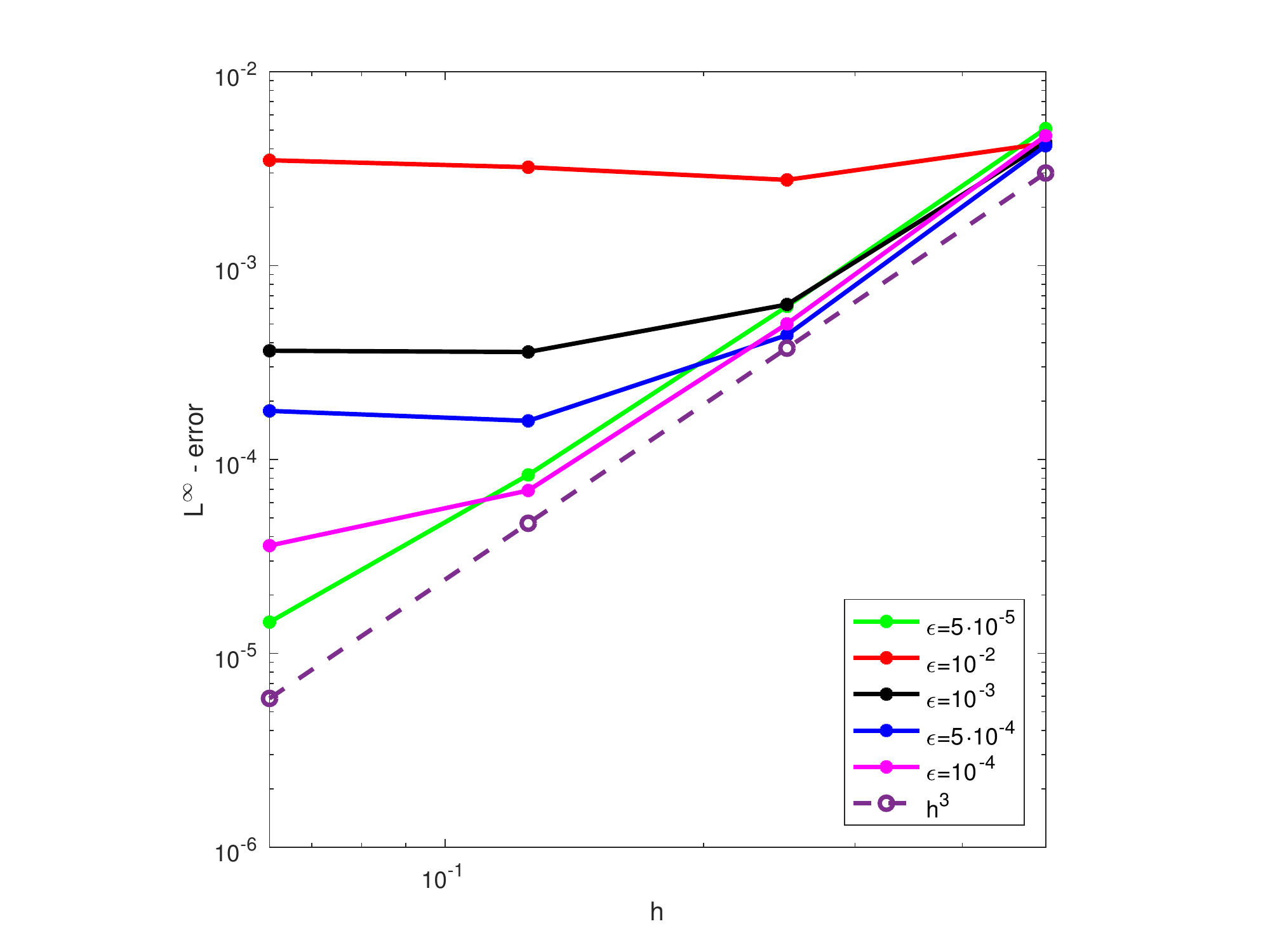}
\vspace{-5mm}
\caption{\footnotesize$L^\infty$-errors with SOI at $T=1$ with $I=[-1000,1000]$ for the exact solution \eqref{exactSolm} of \eqref{StefanProb} with $g\equiv 0$, $\varphi(\xi)=\xi^{0.6}$ and $\alpha=1.5$, approximating the nonlinearity by $\varphi^\epsilon(\xi)=(\xi+\epsilon)^{0.6}-\epsilon^{0.6}$.}
\label{eps_aprox}
\end{figure}

\subsubsection{Explicit scheme approximating the nonlinearity}
\label{explFD}
We consider \eqref{StefanProb} with $g\equiv0$,
$\varphi(\xi)=\xi^{m}$, $m=0.6$, and $\alpha=1.5$, and note that by 
\cite{Hua14} the solution is given by
\begin{equation}
\label{exactSolm}
v(x,t)= \lambda(t+1)^{-\beta} \left(1+(|x|(t+1)^{-\beta})^2\right)^{-\frac{\alpha+1}{2}}
\end{equation}
for $\beta=\frac{1}{m-1+\alpha}$ and $\lambda=\big(\frac{2^{\alpha-1}}{\beta}\frac{\Gamma((1+\alpha)/2)}{\Gamma((3-\alpha)/2)}\big)^{1/(1-m)}$.
We use the approach of Section \ref{sec:aproxNL} and implement
an explicit method
with the Second Order Interpolation + Adapted Vanishing Viscosity
(SOI) + approximate nonlinearity
$\varphi^\epsilon(\xi)=(\xi+\epsilon)^{0.6}-\epsilon^{0.6}$.
We run experiments for
 up to time $T=1$ with $I=[-1000,1000]$ for $\Delta t$'s satisfying
 \eqref{CFL} for $\varphi^\epsilon$. The results are shown in Figure \ref{eps_aprox} (above) and Table \ref{tab:eps_aprox} (below).

\begin{table}[!htbp]\footnotesize
\centering
\begin{tabular}{l|lll|lll|llll}
 $h$&   &     $\epsilon=5$e-$4$              &  & & $\epsilon=1$e-$4$ &                 & &    $\epsilon=5$e-$5$  &   \\ \cline{2-10} 
& \multicolumn{1}{l|}{$\Delta t$} & \multicolumn{1}{l|}{error} & \multicolumn{1}{l|}{$\gamma$} &  \multicolumn{1}{l|}{$ \Delta t$}& \multicolumn{1}{l|}{error}  & \multicolumn{1}{l|}{$\gamma$} & \multicolumn{1}{l|}{$\Delta t$} & \multicolumn{1}{l|}{error}& $\gamma$   \\ \hline
\multicolumn{1}{l|}{ $5.00$e-$1$}& \multicolumn{1}{l|}{$9.39 $e-$3 $}& \multicolumn{1}{l|}{$4.14 $e-$ 3$} &  \multicolumn{1}{l|}{} &\multicolumn{1}{l|}{$4.93 $e-$ 3$} &  \multicolumn{1}{l|}{$4.67 $e-$ 3$}  & \multicolumn{1}{l|}{} &  \multicolumn{1}{l|}{$1.97$e-$3$}  & \multicolumn{1}{l|}{$5.09 $e-$ 3$} &  \\

\multicolumn{1}{l|}{ $2.50$e-$1$}& \multicolumn{1}{l|}{$3.32 $e-$3 $}& \multicolumn{1}{l|}{$ 4.38$e-$ 4$} &\multicolumn{1}{l|}{$3.24$} & \multicolumn{1}{l|}{$ 1.74$e-$3 $} & \multicolumn{1}{l|}{$5.00 $e-$ 4$}&\multicolumn{1}{l|}{$3.22 $} & \multicolumn{1}{l|}{$6.16$e-$4$}  & \multicolumn{1}{l|}{$6.16 $e-$ 4$} &  3.06                  \\

\multicolumn{1}{l|}{$1.25$e-$1$}& \multicolumn{1}{l|}{$1.17 $e-$3 $}& \multicolumn{1}{l|}{$ 1.58$e-$ 4$} &\multicolumn{1}{l|}{$1.47$} &\multicolumn{1}{l|}{$6.17 $e-$ 4$} & \multicolumn{1}{l|}{$6.92 $e-$ 5$}& \multicolumn{1}{l|}{$2.85$} & \multicolumn{1}{l|}{$8.32$e-$5$}  &  \multicolumn{1}{l|}{$ 8.32$e-$ 5$} &       2.89     \\

\multicolumn{1}{l|}{$6.25$e-$2$}&\multicolumn{1}{l|}{$ 4.15$e-$ 4$}& \multicolumn{1}{l|}{$ 1.78$e-$ 4$}& \multicolumn{1}{l|}{$-0.17$} &  \multicolumn{1}{l|}{$2.18 $e-$ 4$}& \multicolumn{1}{l|}{$3.60$e-$5$} & $0.94$ & \multicolumn{1}{l|}{$1.44$e-$5$} &  \multicolumn{1}{l|}{$ 1.45$e-$ 5$  } &     $2.52$          
\end{tabular}
\vspace{2mm}
\caption{\footnotesize$L^\infty$-errors with SOI at $T=1$ with $I=[-1000,1000]$ for the exact solution \eqref{exactSolm} of \eqref{StefanProb} with $f\equiv 0$, $\varphi(\xi)=\xi^{0.6}$ and $\alpha=1.5$, approximating the nonlinearity by $\varphi^\epsilon(\xi)=(\xi+\epsilon)^{0.6}-\epsilon^{0.6}$. Note that we have only included the most accurate approximations here.}
\label{tab:eps_aprox}
\end{table}

\noindent{\bf Conclusion:} Here the expected error is
$O(h^{3-\alpha}+\Delta t)$ or $O(h^{1.5})$ when $\alpha=1.5$ and $\Delta t\sim h^{\alpha}$. Since we
are approximating the nonlinearity, the \eqref{CFL} condition becomes
more and more restrictive when $\epsilon$ is decreased. When we fix $\epsilon$
and let $h,\Delta t \to 0^+$, the error stops decreasing at some
point as can be seen in Figure
\ref{eps_aprox} and Table \ref{tab:eps_aprox}. However, very good
 results are obtained for small $\epsilon$.  Compared with
the implicit method of the previous section, 
the present method is better both in terms of computational times (we
are able to deal with much bigger domains) and accuracy (we reach
almost $10^{-6}$ with   
$h=6.25$e$-2$ instead of $h=3.13$e$-2$).

\subsection{On the truncation of the domain}
\label{sec:TruncationDomainEffect}

To test numerically the effect of the
restriction to a large bounded domain, we consider the
Fractional Heat Equation and the explicit solution of Section
\ref{sec:FractionalHeatEquationExplicit}. An explicit PDL
scheme is run up to $T=1$ on a sequence of increasing domains. We
take $\Delta t\sim h^2$ which satisfies \eqref{CFL} and ensures that
the space discretization errors dominate. See Figure \ref{diffdomains}
for the results. 
We also test the minimal error that can be reached for a fixed domain
for different values of $\alpha$. Here we consider the Fractional
Porous Medium Equation with $m=(3-\alpha)/(1+\alpha)$ and explicit
solutions from \cite{Hua14}. See Figure \ref{diffdomains2}.

\begin{figure}[!htbp]
\centering
\includegraphics[width=\textwidth]{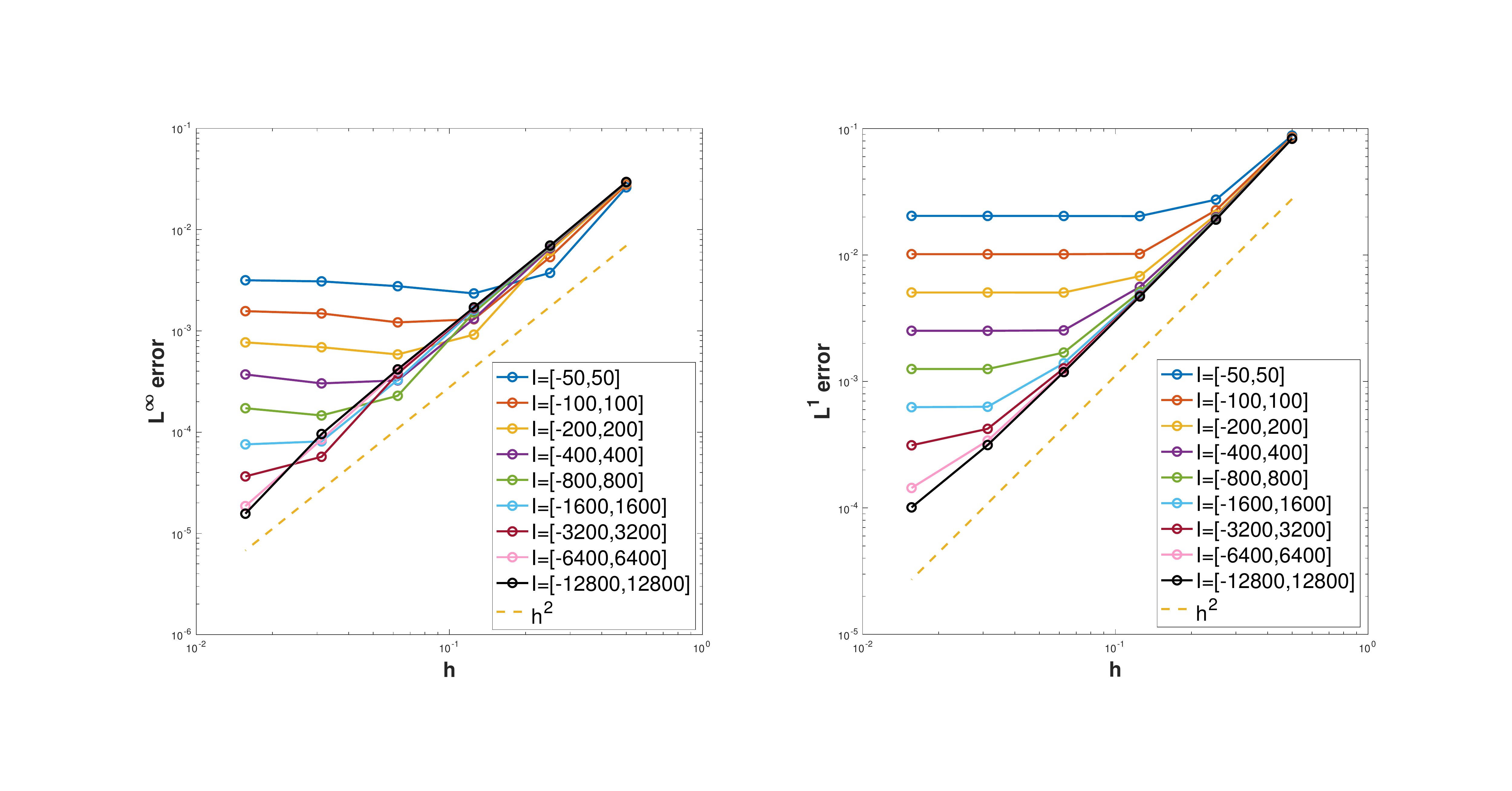}
\vspace{-12mm}
\caption{\footnotesize$L^\infty$- and $L^1$-errors with PDL at $T=1$ with different interval sizes $I$ for the exact solution \eqref{CauchyDistr} of \eqref{StefanProb} with $\alpha=1$, $\varphi(\xi)=\xi$, and $g\equiv 0$.}
\label{diffdomains}
\end{figure}

\medskip
\noindent{\bf Conclusion:} The expected
error for our schemes are $O(h^{2}+\Delta t)=O(h^{2})$ since $\Delta
t\sim h^{2}$. For each fixed domain $I$, we see from Figure
\ref{diffdomains} that the errors decrease as $h\to0^+$ down to some
threshold below which there is no improvement. At these thresholds, the
dominant error comes from the truncation of the domain. As expected, these
thresholds decrease as the size of the domain increases.
 Figure \ref{diffdomains2} shows that when you vary $\alpha$, the minimal
 reachable errors for a fixed domain $I$ is of the order of
 $\textup{length}(I)^{-\alpha}$. This could be an indication that the
 error due to the truncation of the domain is determined by the tail
 behaviour of the L\'evy measure -- here
 $\dd\mu(z)\sim |z|^{-1-\alpha}\dd z$. Such behaviour would be consistent with the analytical results of
 \cite{BrCh13} for tempered L\'evy models (which do not include the
 fractional Laplacian). Another analytical approach using Barenblatt solutions (cf. \cite{Vaz14}) can be found in \cite{DTe14}.

\begin{figure}[!htbp]
\centering
\includegraphics[width=0.75\textwidth]{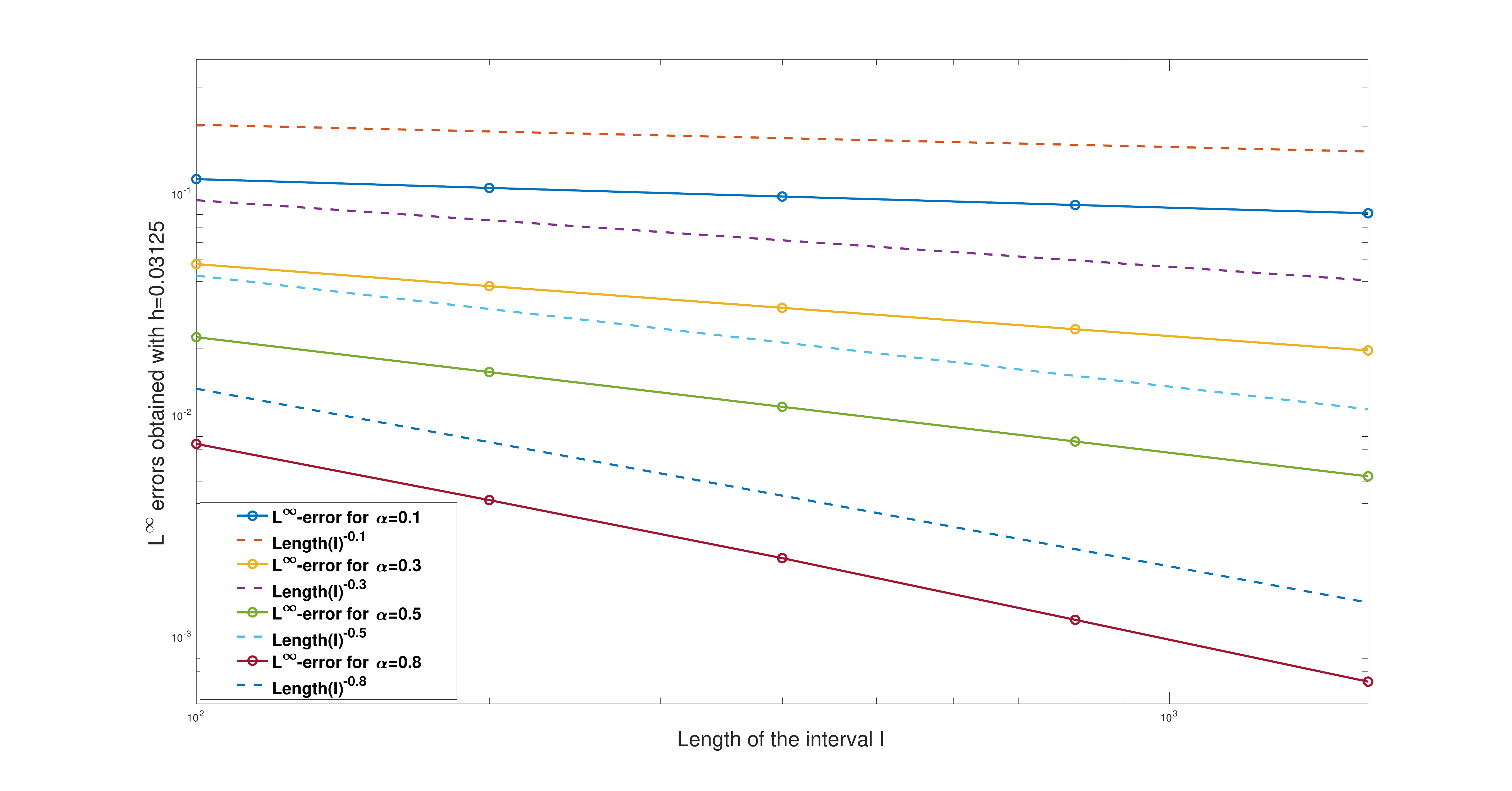}
\vspace{-5mm}
\caption{\footnotesize Minimal $L^\infty$- and $L^1$-errors with PDL at $T=1$ with different interval sizes $I$ for the exact solution of \eqref{StefanProb} with $\varphi(\xi)=\xi^{\frac{3-\alpha}{1+\alpha}}$, and $g\equiv 0$.}
\label{diffdomains2}
\end{figure}

%%%%%%%%%%%%%%%%%%%%%%%%%%%%%%%%%%%%%%%%%%%%%%%%%%%%
%%%%%%%%%%%%%%%%%%%%%NEW SECTION%%%%%%%%%%%%%%%%%%%%%%%
%%%%%%%%%%%%%%%%%%%%%%%%%%%%%%%%%%%%%%%%%%%%%%%%%%%%
\section{Numerical experiments in 2D}
\label{sec:numsim2D}

In this section we test our numerical schemes on more complicated
problems in two space dimensions which has a much richer solution
structure and a more interesting evolution. We consider Stefan problems of the type
\begin{equation}\label{2Dexample}
\dell_t u(x,y,t) +\Operator_i [\varphi(u(\cdot,\cdot,t))](x,y)=0\qquad\text{in}\qquad Q_T:=\R^2\times(0,T)
\end{equation}
for $i=1,2$, $\varphi(\xi)=\max\{0,\xi-1\}$, and both
$x$-directed nonlocal diffusion and ``diagonal'' $\vec{\sigma}^T=(\frac{1}{2}, \frac{47}{100})$-directed local diffusion,
\begin{align}\label{eq:2dexnonlocdom}
\Operator_1[\psi](x,y)&:=(\vec{\sigma}\cdot D)^2[\psi](x,y)+(-\partial_{xx})^{\frac{1}{4}}[\psi](x,y),\\
\label{eq:2dexlocdom}
\Operator_2[\psi](x,y)&:=(\vec{\sigma}\cdot D)^2[\psi](x,y)+\frac{1}{10}(-\partial_{xx})^{\frac{1}{4}}[\psi](x,y).
\end{align}
Note that nonlocal
diffusion is stronger in $\Operator_1$ than in $\Operator_2$.

As in Section \ref{sec:numsim}, we use an equidistant grid in space and
time and restrict to a (large) bounded spatial domain $I_1\times
I_2\subset\R^2$ (setting the numerical solution equal 
zero outside). The errors are computed numerically under the same
assumptions as before. We run an explicit scheme ($\theta=1$)
with  $\Delta t\sim h^2$ such that \eqref{CFL} holds and the overall order of
convergence is determined by the spatial discretization. Since
$x+\vec{\sigma}\eta$ is not aligned with the spatial grid $\Grid$, we
discretize the local term 
by \eqref{AproxLoc2} and $\eta=\sqrt{h}$ which leads to $O(h)$ errors
(cf. Lemma \ref{discretelocal}). For the  nonlocal
diffusion, we use the 1-dimensional version of Powers of the Discrete
Laplacian (PDL) of Section \ref{sec:pdl} which 
has $O(h^2)$ errors.  We also choose a rough initial data (see Figure
\ref{2DStefanCompNonlocal} and Figure \ref{2DStefanCompLocal}) given by
\[
u_0(x,y)= 3 (\indik_{S_1}(x,y) -  \indik_{S_2}(x,y))  + 4 \indik_{S_3}(x,y)
\]
where $S_1=\{(x,y)\in \R^2 \ : |x|<5\  \& \ |y|<5 \}$, $S_2=\{(x,y)\in
\R^2 \ : |x|<2\  \& \ |y|<2 \}$ and $S_3=\{(x,y)\in \R^2 \ : 3<|x|<4\
\& \ 3<|y|<4 \}$.
We run the experiments on the domain $I_1\times I_2=[-100,100]\times
[-10,10]$. The different sizes of the domain in the $x$- and
$y$-directions are adapted to the combination of compactly supported
data and degenerate operators \eqref{eq:2dexnonlocdom} and
\eqref{eq:2dexlocdom}. These operators are nonlocal in the
$x$-direction which requires a wide domain there. The size of the
domain in the $y$-direction can be smaller because the Stefan problem
 with local diffusion has finite speed of propagation
(see e.g. the introduction of \cite{BrChQu12}). 

In figure \ref{fig:fakefig} we list the relative $L^\infty$- and
$L^1$-errors (errors divided by the $L^\infty$- and $ L^1$-norms,
respectively, of the solution) for the numerical solution of \eqref{2Dexample} with $i=1$. In Figure \ref{2DStefanCompNonlocal} (resp. Figure \ref{2DStefanCompLocal}) we plot, for different times, the numerical solution of \eqref{2Dexample} with $i=1$ (resp. $i=2$).

\renewcommand\tabularxcolumn[1]{m{#1}}

\addvspace{\medskipamount}
\noindent
\begin{tabularx}{\linewidth}{ @{} X  X @{} }
\hspace{-3cm}\includegraphics[width=\textwidth]{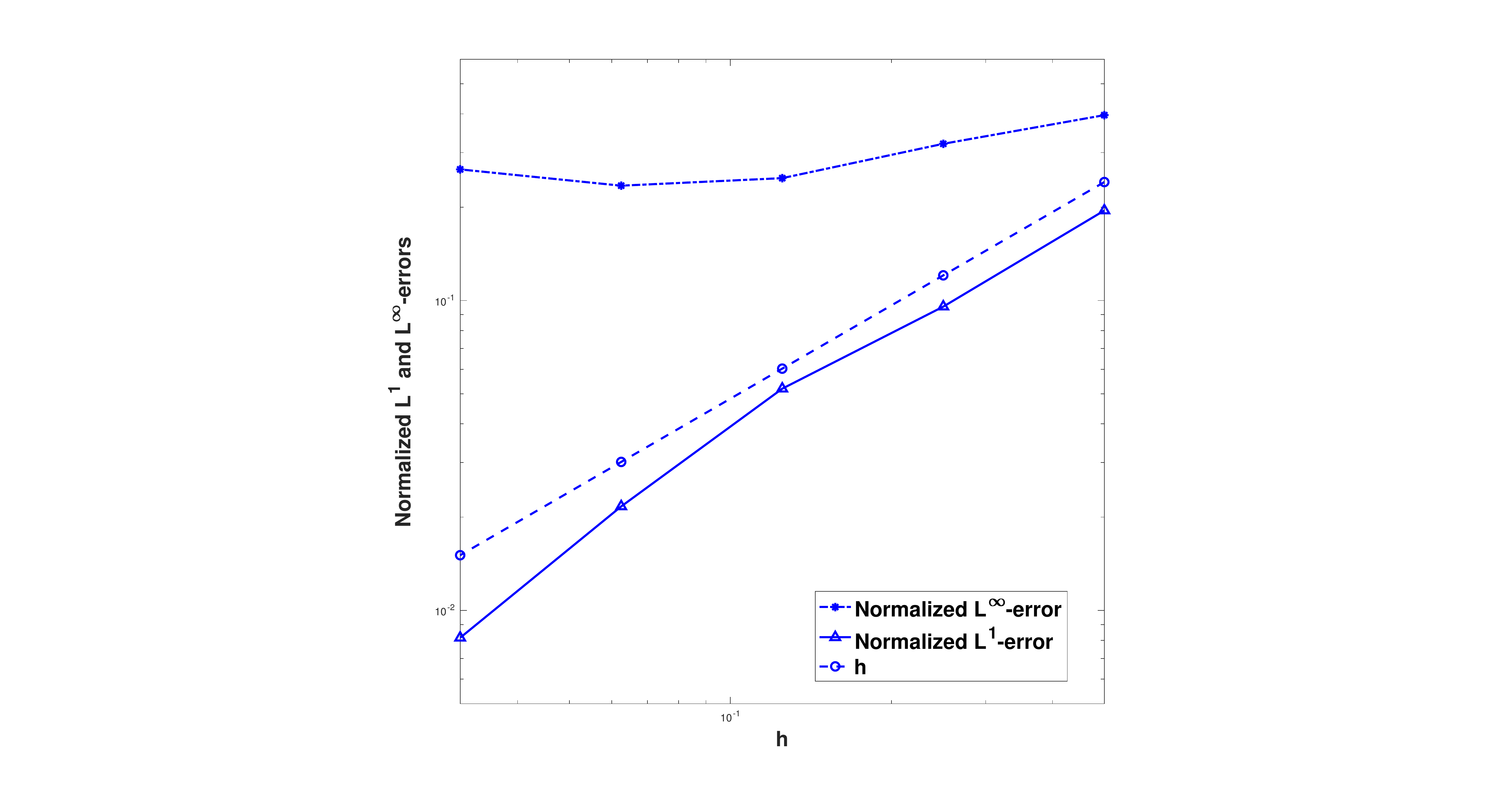}
 &

\begin{tabular}{l|llll}
 $h$ & \multicolumn{1}{l|}{$ L^\infty$}& \multicolumn{1}{l|}{$\gamma$}  & \multicolumn{1}{l|}{$ L^1$}& \multicolumn{1}{l}{$\gamma$}   \\ \hline
 $5.0$e-$1$& \multicolumn{1}{l|}{$ 3.9$e-$ 1$}& \multicolumn{1}{l|}{}  &\multicolumn{1}{l|}{$2 $e-$ 1$}&     \\
 $2.5$e-$1$& \multicolumn{1}{l|}{$3.2 $e-$ 1$}&\multicolumn{1}{l|}{0.3}  &\multicolumn{1}{l|}{$9.6 $e-$ 2$}&1        \\
$1.3$e-$1$& \multicolumn{1}{l|}{$ 2.5$e-$ 1$}& \multicolumn{1}{l|}{0.4} &\multicolumn{1}{l|}{$ 5.2$e-$ 2$}&0.9   \\
$6.3$e-$2$& \multicolumn{1}{l|}{$ 2.3$e-$ 1$}& \multicolumn{1}{l|}{0.1}  &\multicolumn{1}{l|}{$2.2 $e-$ 2$}&1.3            \\
$3.1$e-$2$& \multicolumn{1}{l|}{$2.6 $e-$ 1$}& \multicolumn{1}{l|}{-0.2}   &\multicolumn{1}{l|}{$ 8.2$e-$ 3$}  &1.4\\                   
\end{tabular}
\end{tabularx}
\vspace{-1cm}
\begin{figure}[h!]   
\caption{\footnotesize Relative $L^\infty$- and $L^1$-errors with \eqref{AproxLoc2} and PDL discretizations at $T=1$ with $I_1\times I_2=[-100,100]\times [-10,10]$ for \eqref{2Dexample} with $\varphi(\xi)=\max\{0,\xi-1\}$ and the diffusion operator given by \eqref{eq:2dexnonlocdom}. 
}\label{fig:fakefig}
\end{figure}

\begin{figure}[h!]   
\centering
\begin{subfigure}[b]{1.15\textwidth}
\hspace{-1cm}\includegraphics[width=\textwidth]{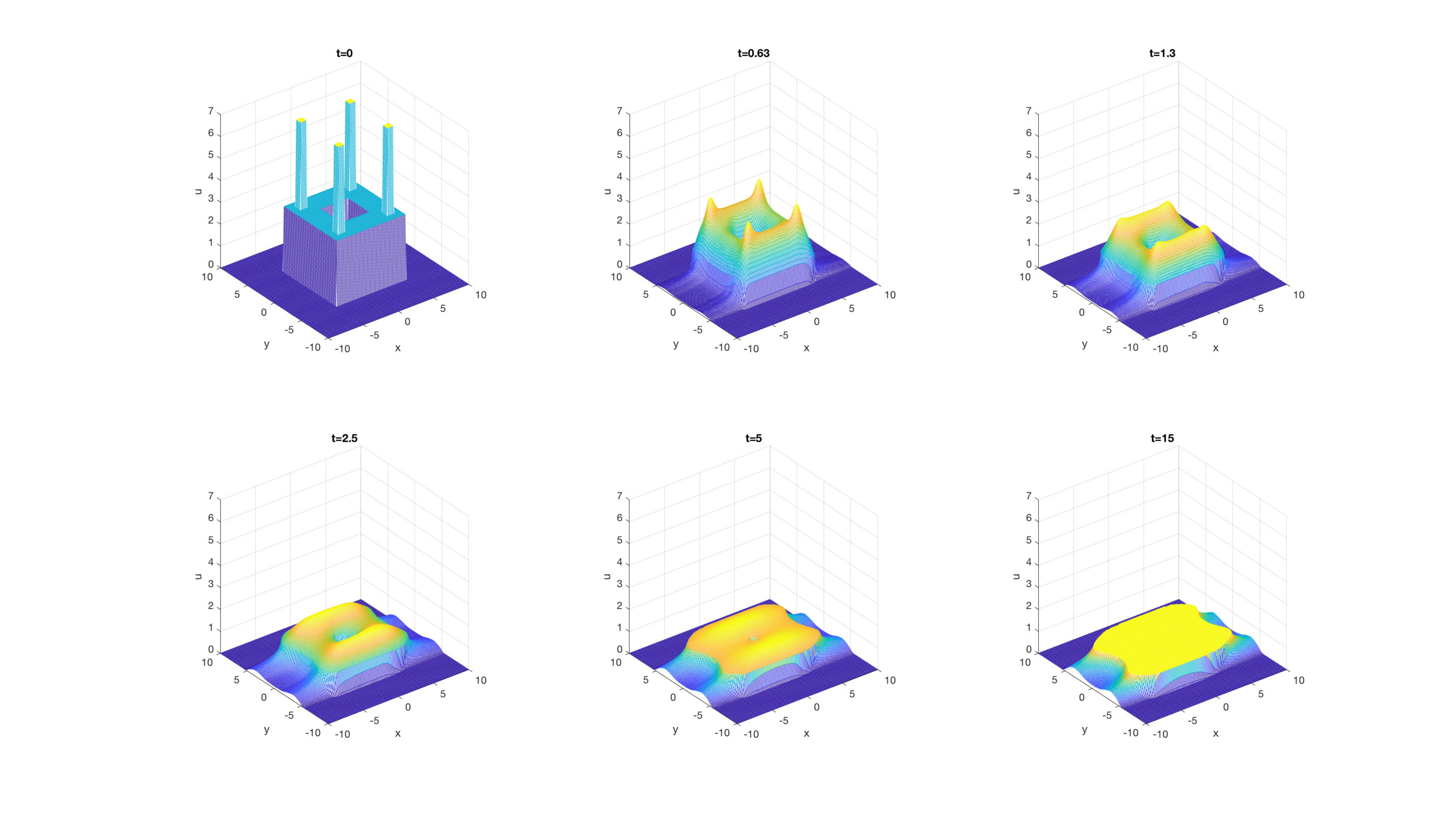}
\end{subfigure}
\begin{subfigure}[b]{1.15\textwidth}
\hspace{-1cm}\includegraphics[width=\textwidth]{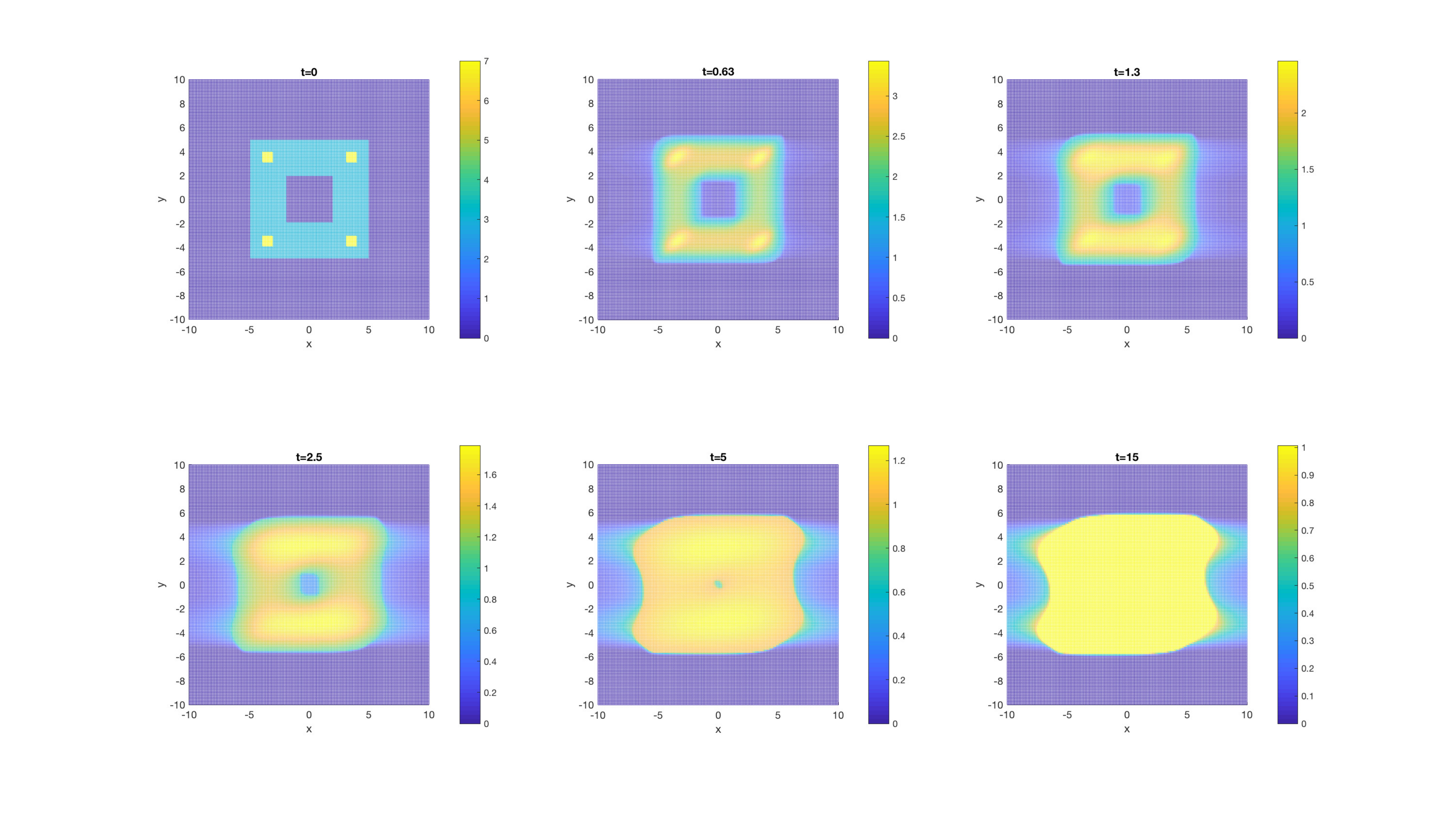}
\end{subfigure}
\vspace{-10mm}
\caption{\footnotesize Stronger nonlocal diffusion. Solution of \eqref{2Dexample} with $\varphi(\xi)=\max\{0,\xi-1\}$ and the diffusion operator given by \eqref{eq:2dexnonlocdom}.
}
\label{2DStefanCompNonlocal}
\end{figure}

\begin{figure}[h!]   
\centering
\begin{subfigure}[b]{1.15\textwidth}
\hspace{-1cm}\includegraphics[width=\textwidth]{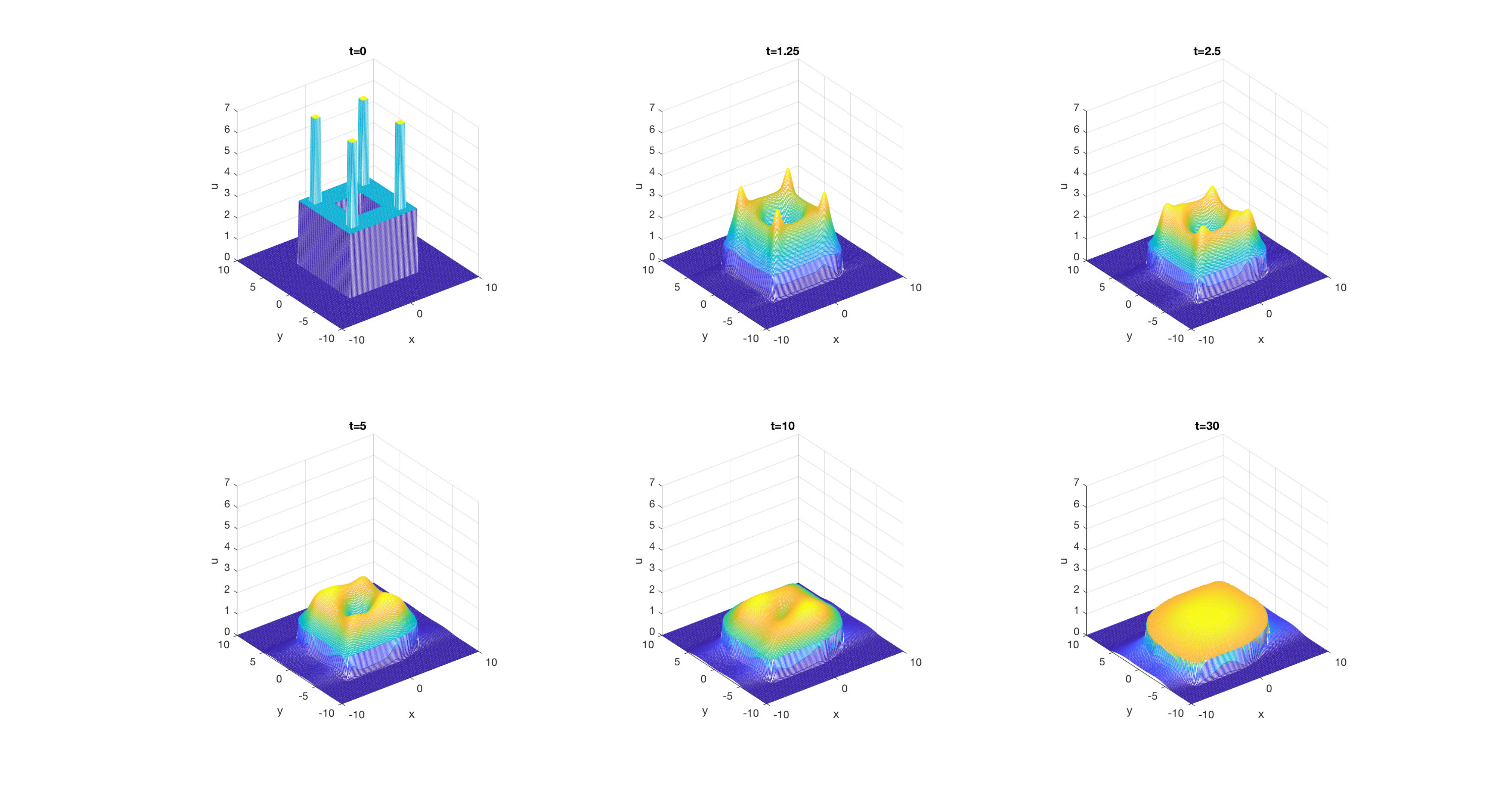}
\end{subfigure}
\begin{subfigure}[b]{1.15\textwidth}
\hspace{-1cm}\includegraphics[width=\textwidth]{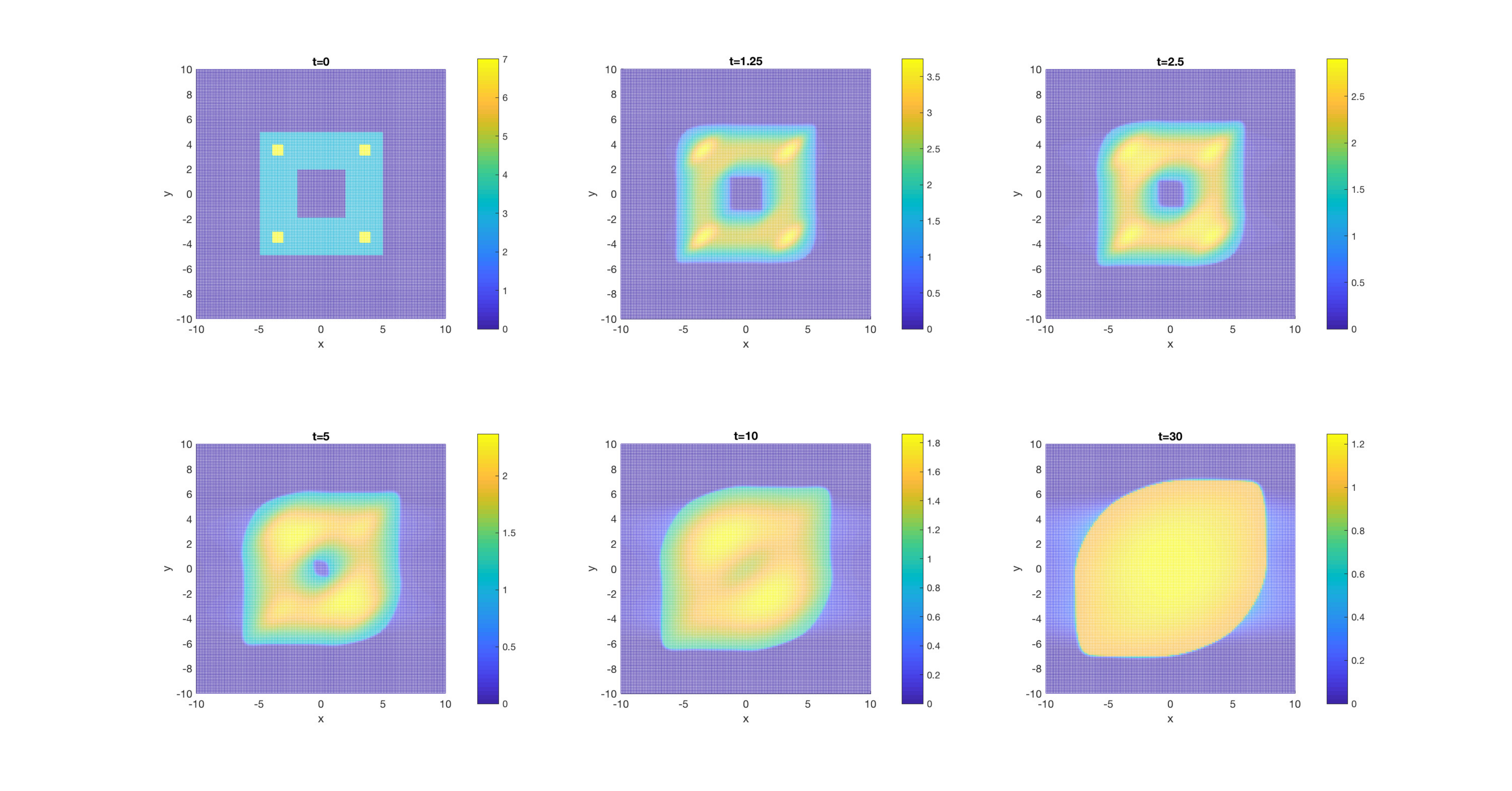}
\end{subfigure}
\vspace{-10mm}
\caption{\footnotesize Weak nonlocal diffusion. Solution of \eqref{2Dexample} with $\varphi(\xi)=\max\{0,\xi-1\}$ and the diffusion operator given by \eqref{eq:2dexlocdom}.}
\label{2DStefanCompLocal}
\end{figure}

\smallskip
\noindent{\bf Conclusion:} Figure \ref{fig:fakefig} confirms the
$O(h)$ convergence in $L^1$ predicted by Lemma
\ref{discretelocal}. As in Section
\ref{sec:StefanTyepProblemExplicit}, there is no  
convergence in $L^\infty$. In Figure \ref{2DStefanCompNonlocal} there
are pronounced long ``bands'' due to the nonlocal diffusion, and the
effect of the almost diagonal local diffusion is also visible. In Figure
\ref{2DStefanCompLocal} the nonlocal diffusion is 
weaker and the local diffusion dominates. The ``long bands'' are not
very pronounced. Both figures exhibit regions of discontinuity of the
solution in the $y$-direction (where there is only local
diffusion). Such behaviour is well-known for local Stefan problems.

\section*{Acknowledgements}

All authors were supported by the Toppforsk (research excellence)
project Waves and Nonlinear Phenomena (WaNP), grant no. 250070 from
the Research Council of Norway. F.~del Teso was also supported by the
BERC 2018-2021 program from the Basque Government, BCAM Severo Ochoa
excellence accreditation SEV-2017-0718 from Spanish Ministry of
Economy and Competitiveness (MINECO), the ERCIM  ``Alain
Bensoussan'' Fellowship programme and the ``Juan de la Cierva - formaci\'on" program (FJCI-2016-30148).  We would like to
thank the anonymous referees for helping us improve the paper and
Indranil Chowdhury for a careful reading and useful comments.

\end{document}